\documentclass{article}
\usepackage[T5,T1,OT1]{fontenc}
\usepackage{times}
\usepackage[margin=1in]{geometry}
\usepackage{natbib}
\usepackage{amsmath,amssymb,amsthm,mathtools,bm}
\usepackage{booktabs}
\usepackage{longtable}
\usepackage{graphicx}
\usepackage{array}
\usepackage{microtype}
\usepackage[hypertexnames=false]{hyperref}
\hypersetup{
  hidelinks,
  pdftitle={Decentralized Optimization with Cross-Coupled Mixed Affine Constraints},
  pdfauthor={Ewsey Obzherin, Ilya Khomchenko, Natalia Shelegeda, Nhat Trung Nguyen, Demyan Yarmoshik, Alexander Rogozin, Alexander Gasnikov}
}
\usepackage{url}
\usepackage{enumitem}
\usepackage{xcolor}
\usepackage{algorithm}
\usepackage{float}
\usepackage{algpseudocode}

\newtheorem{theorem}{Theorem}[section]
\newtheorem{lemma}[theorem]{Lemma}
\newtheorem{proposition}[theorem]{Proposition}
\newtheorem{corollary}[theorem]{Corollary}
\newtheorem{assumption}{Assumption}

\newcommand{\R}{\mathbb{R}}
\newcommand{\Range}{\operatorname{Range}}
\newcommand{\Ker}{\operatorname{Ker}}
\newcommand{\diag}{\operatorname{diag}}
\newcommand{\col}{\operatorname{col}}

\newcommand{\chix}{\chi_{\times}}
\newcommand{\chiJ}{\chi_{\times,J}}
\newcommand{\CCMAC}{\textsc{CC-MAC}}
\newcommand{\eps}{\varepsilon}
\newcommand{\one}{\mathbf{1}}
\newcommand{\Wp}{\mathcal W}
\newcommand{\cF}{c_F}
\newcommand{\RF}{\operatorname{RF}}
\newcommand{\UF}{\operatorname{UF}}

\title{Decentralized Optimization with Cross-Coupled Mixed Affine Constraints}
\author{
Ewsey Obzherin$^{1,2}$, Ilya Khomchenko$^{1}$, Natalia Shelegeda$^{1}$,\\
Nhat Trung Nguyen$^{1}$, Demyan Yarmoshik$^{1}$, Alexander Rogozin$^{1}$,\\
Alexander Gasnikov$^{1,3}$\\[0.6em]
\small $^{1}$Moscow Independent Research Institute of Artificial Intelligence\\
\small $^{2}$FusionBrain Lab\\
\small $^{3}$Innopolis University
}
\date{}

\begin{document}
\maketitle

\begin{abstract}
We study decentralized optimization with cross-coupled mixed affine constraints, where local and shared variables interact through two affine channels. We show that the intrinsic difficulty of combining separately well-conditioned channels is governed by their Friedrichs angle. This geometry induces a cross-coupling factor that cannot be removed by channelwise preconditioning and governs the additional affine-oracle and communication complexity. We develop an accelerated decentralized method with matching minimax guarantees in the smooth strongly convex regime and extend the framework to smooth and nonsmooth convex objectives. Experiments confirm the predicted dependence on cross-channel geometry and network conditioning.
\end{abstract}

\section{Introduction}
Decentralized first-order optimization is 
a popular approach for solving problems where the data is distributed across compute nodes, and the communication between nodes is slow, unreliable or restricted by privacy concerns \citep{lian2017can,gorbunov2022recent}.
It is best studied in the \textit{consensus optimization} case where the global variable $x$ is shared by locally stored terms $f_i(x)$ of the objective  $\min_x \sum_i f_i(x)$, which corresponds to data-parallel learning.
This setup has received thorough practical examination \citep{assran2019stochastic,ying2021bluefog,yuan2022decentralized} and has a well-developed complexity theory with optimal algorithms and matching lower complexity bounds \citep{scaman2017optimal,scaman2018nonsmooth,kovalev2020optimal,tyurin2026optimality}. 

While in consensus optimization all compute nodes work with the same shared variable $x$, many practical applications require more general coupling between variables.
Recently, optimal algorithms were established for the case of \textit{coupled constraints}, where local variables $x_i$ are subject to a kind of resource-allocation constraint 
$\min_x \sum_i f_i(x_i) \text{ s.t. } \sum_i (A_i x_i - b_i) = 0$ \citep{yarmoshik2025coupled,yarmoshik2026mixed}.
This is a generalization of consensus optimization, which also makes it possible to express model-parallel learning formulations and different sharing constraints in control and operations research \citep{yarmoshik2026mixed, boyd2011distributed}.

In these settings, optimal complexities are described by the objective, network, and affine condition numbers \citep{scaman2017optimal,scaman2018nonsmooth,salim2022optimal,yarmoshik2025coupled,yarmoshik2026mixed}. 
Mathematically, the core question is how the conditioning of one problem component affects the others. We study it for local variables $x_i$ and a shared variable $z$. Existing optimal methods cover coupled constraints $\sum_i(A_i x_i-b_i)=0$ \citep{yarmoshik2025coupled} and the mixed-affine model in which $z$ enters only node-wise constraints \citep{yarmoshik2026mixed}. For $n$ agents, we define the cross-coupled mixed-affine problem as
\begin{equation}
\boxed{
\begin{aligned}
\min_{x_1,\ldots,x_n,z}\quad &\sum_{i=1}^n f_i(x_i,z),\\
\text{s.t.}\quad &\sum_{i=1}^n(A_i x_i+B_i z-b_i)=0,\\
& C_i x_i=c_i,\qquad D_i z=e_i,\qquad i=1,\ldots,n.
\end{aligned}}
\tag{\CCMAC}
\label{eq:ccmac}
\end{equation}
Thus $z$ enters both the aggregate equality and the node-wise shared-variable constraints. The canonical reduction separates these equations into two operators $K_1$ and $K_2$. Let $U=\Range(K_1^\top)$ and $V=\Range(K_2^\top)$; for any subspace $S$, let $P_S$ denote its Euclidean orthogonal projector. After exact separate left whitening and stacking the channels as $\widehat K=[\widehat K_1;\widehat K_2]$, the joint normal operator is
\begin{equation}
\widehat K^\top\widehat K=P_U+P_V.
\label{eq:intro-projector-sum}
\end{equation}
When $U$ and $V$ contain nearly aligned but nonidentical directions, the smallest positive eigenvalue of \eqref{eq:intro-projector-sum} can be arbitrarily small. Thus two channels can be separately ideal and jointly ill conditioned. Exact overlap is benign---a common direction contributes eigenvalue $2$---whereas \emph{near-overlap} is the source of difficulty.

\paragraph{Where the CC-MAC structure arises.}
The formulation is useful when a system has genuinely local decisions and a shared decision, and both must enter the same aggregate equality. Three representative examples are:
\begin{itemize}[leftmargin=1.2em,itemsep=0.5pt,topsep=1pt,parsep=0pt]
\item \emph{Vertical federated learning.} Agent $i$ owns a feature block and its weight $x_i$, while $z$ is the shared prediction vector, as in vertically partitioned federated learning \citep{hardy2017private}. The identity $Xx-z=0$ can be written as $\sum_i(X_ix_i-\alpha_i z)=0$ with $\sum_i\alpha_i=1$, so $A_i=X_i$ and $B_i=-\alpha_iI$. Consensus or admissibility constraints on copies of $z$ form the second channel. This is also the full-data model used in Section~\ref{sec:experiments}.
\item \emph{Networked resource allocation and control.} Local dispatch or control variables $x_i$ must satisfy equipment equations $C_ix_i=c_i$, while a shared reserve, reference trajectory, or market schedule $z$ satisfies $D_iz=e_i$. Aggregate balance has the form $\sum_i(A_ix_i+B_iz-b_i)=0$: the shared schedule changes every node's contribution to the same system-wide balance, as in distributed optimization and ADMM models \citep{boyd2011distributed}.
\item \emph{Personalized models with global calibration.} Local corrections $x_i$ are combined with a common model $z$ \citep{hanzely2020personalized,liang2020localglobal}. A population-level calibration, fairness, or moment-matching equality aggregates local statistics and the common predictor as $\sum_i(A_ix_i+B_iz-b_i)=0$ \citep{agarwal2018fair}; local identifiability and shared-policy equalities remain in the second channel.
\end{itemize}
In all three cases $B_i\ne0$ is structural rather than a cosmetic reparameterization: deleting it changes the feasible model, while merging the two equation classes changes which operations are locally available and therefore changes the oracle being charged.

Setting $B_i=0$ recovers the previous mixed-affine model; in our canonical reformulation the two channel row spaces are then orthogonal and the cross-coupling factor equals one. This means only that no \emph{additional cross-channel} conditioning remains after the channels are normalized separately. It does not make the primitive affine realization unit-cost: products with $A_i,C_i,D_i$ and the communication needed to implement a normalized channel remain charged, while $N_B=0$.

The two canonical \CCMAC{} channels are block operators and may each contain arbitrarily many affine equations. After separate normalization, exact common row directions are spectrally benign while near-overlap is hard. If $c_F$ is the Friedrichs cosine after removing exact common directions and $\theta_F:=\arccos c_F\in(0,\pi/2]$ (with $\theta_F=\pi/2$ when either reduced subspace is zero), the residual two-channel factor is
\begin{equation}
\chix=\sqrt{\frac{1+c_F}{1-c_F}}
=\cot\!\left(\frac{\theta_F}{2}\right).
\label{eq:intro-cross-factor}
\end{equation}
Thus $\chix$ is exactly the conditioning left after within-channel conditioning has been removed for the prescribed channel decomposition. Principal- and Friedrichs-angle tools are classical \citep{bjorckgolub1973angles,deutsch1995angle}; our use identifies the residual factor created by the canonical \CCMAC{} split and its resource consequences. Separate preprocessing cannot rotate the two row spaces, whereas mixing equations across channel classes changes the charged oracle. The fixed-$J$ extension is in Appendix~\ref{app:weightedprojectors}.

The resulting picture is summarized by the following contributions.

\paragraph{Contributions.}
\begin{itemize}[leftmargin=1.2em,itemsep=0.7pt,topsep=1pt,parsep=0pt]
\item \textbf{A genuinely cross-coupled decentralized model.} We introduce \CCMAC{}, in which the shared block enters both the global coupled equality and node-wise affine constraints.  The canonical reduction isolates two affine channels whose interaction is absent from the predecessor mixed-affine model: setting $B_i=0$ makes the channel row spaces orthogonal and recovers $\chix=1$.
\item \textbf{A CC-MAC specialization of classical angle geometry.}  After separate whitening, the nontrivial spectrum of the prescribed joint normal operator is determined by the Friedrichs angle and the residual condition number is exactly $\chix^2=(1+c_F)/(1-c_F)$.  Classical angle and perturbation inequalities \citep{knyazevargentati2002principal} explain the barrier; our contribution is to identify the canonical \CCMAC{} pair, attain the barrier by channelwise preprocessing, and propagate it through the resource ledger. Exact common directions are deduplicated before the angle is measured.
\item \textbf{Resource-separated smooth strongly convex guarantees.}  For a fixed execution representation with a supplied initial primal--dual energy bound, Cross-APAPC uses $O(\sqrt{\kappa_f}\log_+(R_E/\eps))$ gradients, $O(\chix\sqrt{\kappa_f}\log_+(R_E/\eps))$ normalized-affine calls, and the corresponding supplied-realization communication budget.  The matching smooth-SC statement concerns the gradient and normalized-affine resource orders; the physical-communication lower construction is radius-explicit. Appendix~\ref{app:lower-bound-program} states the oracle and radius scope explicitly.
\item \textbf{A log-free smooth-convex extension.}  A retained-dual continuation with polynomial feasibility repair closes the multiplicative target-accuracy logarithmic gap of the regularization-based mixed-affine analysis.  Under a fixed scale-invariant multiplier budget we obtain the matching class-level orders $\Theta(T_{\rm cvx})$, $\Theta(\chix T_{\rm cvx})$, and $\Theta(\chix\sqrt{\kappa_W}T_{\rm cvx})$.  We also give nonsmooth resource-separated upper bounds and a fixed-$J$ normalized-oracle extension.
\item \textbf{Algorithms and full-data validation.}  The main text gives Cross-APAPC and the appendix supplies its normalization/filtering primitives; the nonsmooth extension states the additional projection-oracle access required by its recursive-sliding subroutine.  On H100 full-data LIBSVM experiments, all $182$ controlled geometry/network runs and all $48$ robustness runs reach the target accuracy; the measured scaling exponents are $0.988$ for normalized affine work in $\chix$ and $0.980$ for communication in $\chix\sqrt{\kappa_W}$.  Five-seed full-data comparisons report both normalized-oracle and raw primitive ledgers.
\end{itemize}

\paragraph{Relation to prior and concurrent work.}
The coupled-constraint method of \citet{yarmoshik2025coupled} supplies the direct algorithmic and lower-bound lineage for a single global affine mechanism, while \citet{yarmoshik2026mixed} separates local and shared affine products when the shared block does not enter the coupled equality. \CCMAC{} keeps those native costs visible but introduces the nonorthogonal canonical pair created by $B_i\ne0$. The comparison is therefore between different oracle decompositions, not equal-cost matrix units. Concurrent work of \citet{jha2026twogaps} studies one common variable constrained to the intersection of agent-private sets, with projection/constrained-solve access. For affine sets its gap is $1-c_F^2$ between product feasibility and consensus. Here $c_F$ is instead the angle between two normalized \CCMAC{} channel row spaces and the oracle separately charges matrix--vector affine products, gradients, and neighbor rounds.

\paragraph{Closest asymptotic comparison.}
To separate the new cross-channel effect from previously known objective, affine, and network conditioning, define
$L_\eps:=\log(1/\eps)$, $T_{\rm sc}:=\sqrt{\kappa_f}L_\eps$,
$T_{\rm cvx}:=\sqrt{L_fR^2/\eps}$,
$\chi_K:=\sqrt{\kappa_+(K^\top K)}=\sqrt{\kappa(K)}$,
$\Lambda_{\rm cpl}:=\sqrt{\widehat\kappa_A}$,
$\Lambda_{\rm mix}:=\sqrt{\widetilde\kappa_{AC}}(1+\sqrt{\kappa_C})+\sqrt{\widehat\kappa_{\widetilde C^\top}}$, and
$\Gamma_{\rm mix}:=(\sqrt{\widetilde\kappa_{AC}}+\sqrt{\widehat\kappa_{\widetilde C^\top}})\sqrt{\kappa_W}$.
Table~\ref{tab:closest-comparison} reports the published leading upper orders in the resource units used by each paper; for \citet{yarmoshik2026mixed}, the affine column sums their separately reported $A,C,\widetilde C$ products. Our affine entry counts normalized joint-stack calls, while our physical-round realization depends additionally on the supplied channel-normalization implementation. The raw $A/B/C/D$ products and physical-round conversion are expanded in Appendix~\ref{app:resources}. Consequently the affine and communication columns are model/realization dependent and should not be read as equal-unit wall-clock comparisons.
\begin{table}[t]
\centering
\caption{Closest first-order comparison in the cited works' reported resource models. Affine and communication units are model/realization dependent: predecessor rows use their native accounting, whereas our affine entries count normalized stack products and Appendix~\ref{app:resources} expands our solver into raw matrix products and physical rounds. A dash means that the cited work does not give the corresponding decentralized result in that regime.}
\label{tab:closest-comparison}
\vspace{1mm}
\scriptsize
\setlength{\tabcolsep}{2.0pt}
\begin{tabular}{@{}llccc@{}}
\toprule
Work & regime/model & $N_\nabla$ & affine work (native model) & $N_{\rm comm}$\\
\midrule
\citet{salim2022optimal} & smooth SC, $Kx=b$ & $O(T_{\rm sc})$ & $O(\chi_KT_{\rm sc})$ & --\\
\citet{yarmoshik2025coupled} & smooth SC, coupled & $O(T_{\rm sc})$ & $O(\Lambda_{\rm cpl}T_{\rm sc})$ & $O(\Lambda_{\rm cpl}\sqrt{\kappa_W}T_{\rm sc})$\\
\citet{yarmoshik2026mixed} & smooth SC, mixed & $O(T_{\rm sc})$ & $O(\Lambda_{\rm mix}T_{\rm sc})$ & $O(\Gamma_{\rm mix}T_{\rm sc})$\\
\textbf{This work} & \textbf{smooth SC, CC-MAC} & $\mathbf{O(T_{\rm sc})}$ & $\mathbf{O(\chix T_{\rm sc})}$ & $\mathbf{O(\chix\sqrt{\kappa_W}T_{\rm sc})}$\\
\citet{yarmoshik2026mixed} & smooth convex, mixed & $O(T_{\rm cvx}L_\eps)$ & $O(\Lambda_{\rm mix}T_{\rm cvx}L_\eps)$ & $O(\Gamma_{\rm mix}T_{\rm cvx}L_\eps)$\\
\textbf{This work} & \textbf{smooth convex, CC-MAC} & $\mathbf{\Theta(T_{\rm cvx})}$ & $\mathbf{\Theta(\chix T_{\rm cvx})}$ & $\mathbf{\Theta(\chix\sqrt{\kappa_W}T_{\rm cvx})}$\\
\bottomrule
\end{tabular}
\vspace{-1mm}
\end{table}

\noindent\textbf{The new dependence is geometric rather than another within-channel condition number:} after separate normalization, the residual affine factor is exactly $\chix$.  \textbf{In the smooth-convex regime all three leading normalized-resource orders are log-free;} the regularization-based mixed-affine bound of \citet{yarmoshik2026mixed} retains a multiplicative $L_\eps$ factor.  The table is therefore a leading-order oracle comparison rather than an equal-unit raw-cost ranking; Appendix~\ref{app:resources} gives the explicit conversion for our solver.

\section{Problem Class and Canonical Reformulation}
\label{sec:problem}
We follow the notation and oracle model of \citet{yarmoshik2026mixed}. Agent $i$ has $x_i\in\R^{d_i}$, while $z\in\R^q$ is shared, and
\[
\begin{gathered}
A_i\in\R^{m\times d_i},\qquad B_i\in\R^{m\times q},\qquad b_i\in\R^m,\\
C_i\in\R^{p_i\times d_i},\qquad c_i\in\R^{p_i},\qquad
D_i\in\R^{r_i\times q},\qquad e_i\in\R^{r_i},
\end{gathered}
\]
where $f_i:\R^{d_i}\times\R^q\to\R$ is local. Agent $i$ stores its displayed data, evaluates $\nabla f_i$, and applies its matrices and adjoints.

\paragraph{Oracle and accuracy conventions.}
We use the resource-separated linear-span model of \citet{salim2022optimal,yarmoshik2025coupled}. A normalized affine call applies the separately normalized joint stack or its adjoint; raw $A/B/C/D$ products, graph polynomials, setup, and physical neighbor rounds are charged in Appendix~\ref{app:resources}. One multiplication by symmetric $W\succeq0$, $\Ker W=\operatorname{span}\{\one\}$, is one neighbor round, with $\kappa_W:=\lambda_{\max}(W)/\lambda_{\min}^+(W)$ (and $\kappa_W=1$, zero rounds, for $n=1$).

For nonzero $M$ and positive-semidefinite $G$, write
\[
\kappa(M):=\frac{\sigma_{\max}(M)^2}{\sigma_{\min}^+(M)^2},
\qquad
\kappa_+(G):=\frac{\lambda_{\max}(G)}{\lambda_{\min}^+(G)},
\]
so $\kappa(M)=\kappa_+(M^\top M)$. Empty channels are omitted. Unless stated otherwise, each $f_i$ is $L_f$-smooth and $\mu_f$-strongly convex, with $\kappa_f=L_f/\mu_f$.

In the smooth strongly convex regime, $\eps$-accuracy means $\|w_{\rm out}-w^\star\|\le\eps$ in the execution space, and $N_r$ counts resource-$r$ operations to the first such output. Initialization, certificates, output map, and an upper bound $R_E$ on the APAPC energy of Appendix~\ref{app:algorithmproof} belong to the instance. The empirical error in Section~\ref{sec:experiments} is a separate statistic.

\subsection{Canonical channels and a decentralized implementation lift}
Each agent keeps a local copy $z_i\in\R^q$ of the shared variable. With $d:=\sum_i d_i$, write
\[
\bm x=\col(x_1,\ldots,x_n)\in\R^d,\qquad
\bm z=\col(z_1,\ldots,z_n)\in\R^{nq},\qquad u:=\col(\bm x,\bm z).
\]
Define
\begin{align}
A&:=\diag(A_1,\ldots,A_n), & B&:=\diag(B_1,\ldots,B_n),\nonumber\\
C&:=\diag(C_1,\ldots,C_n), & D&:=\diag(D_1,\ldots,D_n),\label{eq:stacked-data}\\
\bm b&:=\col(b_1,\ldots,b_n), & \bm c&:=\col(c_1,\ldots,c_n),\qquad \bm e:=\col(e_1,\ldots,e_n).\nonumber
\end{align}
The lifted objective is $F(u):=\sum_i f_i(x_i,z_i)$.  We use
\[
\Wp_m:=W\otimes I_m,\qquad \Wp_q:=W\otimes I_q,
\qquad E_m:=n^{-1/2}(\one^\top\otimes I_m).
\]
Thus $E_m r=0$ means that the blocks of $r$ sum to zero, while $\Wp_q\bm z=0$ is exactly $z_1=\cdots=z_n$.

Crucially, we define the geometry \emph{before} introducing any auxiliary localization variable.  The canonical reduced channels are
\begin{equation}
K_1:=\begin{bmatrix}E_mA&E_mB\\ C&0\end{bmatrix},\qquad
K_2:=\begin{bmatrix}0&D\\0&\Wp_q\end{bmatrix},\qquad
K:=\begin{bmatrix}K_1\\K_2\end{bmatrix},
\label{eq:channels}
\end{equation}
with right-hand sides $v_1:=\col(E_m\bm b,\bm c)$ and $v_2:=\col(\bm e,0)$.  This system acts only on the original variables and their canonical local copies.  The first channel contains the coupled equality and local constraints on $x_i$; the second contains node-specific constraints on the shared variable and consensus.

Each $K_j$ is a block family. Appendix~\ref{app:weightedprojectors} gives the fixed-$J$ projector-sum reduction; Appendix~\ref{app:equivalence} gives an exact graph-local lift whose row-space distortion is controlled by Proposition~\ref{prop:decentralized-bridge}.
\begin{theorem}[Canonical reduction and decentralized realization]
\label{thm:equivalence}
A tuple $(x_1,\ldots,x_n,z)$ is feasible for~\eqref{eq:ccmac} iff its local-copy vector $u$ satisfies $K_1u=v_1$ and $K_2u=v_2$. Appendix~\ref{app:equivalence} gives an equivalent graph-local lift and proves that the canonical $c_F$ is independent of its auxiliary scalings.
\end{theorem}
Each graph-filtered local action costs $O(\sqrt{\bar\kappa_W})$ neighbor rounds before the separate-normalization costs of Appendices~\ref{app:ccmac-normalization} and~\ref{app:resources}. If every $B_i=0$, the canonical row spaces are orthogonal and $\chix=1$ (Corollary~\ref{cor:old-orthogonal}); only the cross factor and $N_B$ vanish, not the $A/C/D$ or communication costs.

\section{The Cross-Angle Geometry}\label{sec:angle}
Separately normalized channels yield a sum of row-space projectors (Proposition~\ref{prop:multi-channel}). For two canonical \CCMAC{} classes its spectrum is described by principal angles; the Friedrichs angle removes benign exact overlap. Appendix~\ref{app:two-agent-example} keeps the feasible problem fixed while making this angle arbitrarily small.

We first isolate the linear-algebraic core. Let $M_1:\R^d\to\R^{p_1}$ and $M_2:\R^d\to\R^{p_2}$ be two nonzero linear maps. Set
\[
U=\Range(M_1^\top),\qquad V=\Range(M_2^\top).
\]
Put $I:=U\cap V$. For a consistent affine pair, Lemma~\ref{lem:affine-dedup} retains exact common equations in one channel without changing feasibility; near-common directions remain. The Friedrichs cosine is
\begin{equation}
\cF(U,V):=\sup\bigl\{|\langle u,v\rangle|:\ u\in U\cap I^\perp,\ v\in V\cap I^\perp,\ \|u\|=\|v\|=1\bigr\}.
\label{eq:friedrichs}
\end{equation}
Set $\cF(U,V)=0$ if either reduced subspace is zero. When $U\cap V=\{0\}$, this is the largest cosine of the principal angles.

For the canonical pair we henceforth suppress the arguments and write
$c_F:=c_F(\Range(K_1^\top),\Range(K_2^\top))$. The single cross-coupling factor used in the paper is
\begin{equation}
\chix:=\sqrt{\frac{1+c_F}{1-c_F}}.
\end{equation}
In the core model $C=D=0$, Proposition~\ref{prop:heterogeneity} reduces $c_F$ to an $m\times m$ generalized eigenproblem in $A_i,B_i$. Theorem~\ref{thm:angle} gives the weighted bounds; after consistent deduplication and exact separate whitening,
\begin{equation}
\kappa\!\left(\begin{bmatrix}\widehat M_1\\ \widehat M_2\end{bmatrix}\right)
=\frac{1+c_F}{1-c_F}=\chix^2.
\label{eq:exactcondition}
\end{equation}
Theorem~\ref{thm:blockwise-opt} shows that this value is the infimum over all invertible within-channel preprocessing, so the penalty is not a whitening artifact.

\section{Accelerated Method from the Cross-Angle Geometry}
\label{sec:algorithm}
We now turn the geometry into an algorithmic statement. The construction reuses the affine-constrained acceleration of \citet{salim2022optimal} and the decentralized polynomial reductions developed in the line of \citet{kovalev2020optimal,yarmoshik2025coupled,yarmoshik2026mixed}. Algorithmically, the number of channel classes enters only through the spectrum of the normalized joint stack: Proposition~\ref{prop:multi-channel} gives condition $O(\chiJ^2)$ for $J$ classes. The two-channel Friedrichs factor is the closed-form specialization relevant to \CCMAC{}.

Exact whitening is unnecessary. Corollary~\ref{cor:inexact-whitening} proves that fixed spectral-equivalence accuracy preserves the two-channel order, while Proposition~\ref{prop:poly-normalization} and Lemma~\ref{lem:structured-normalization} give a constructive decentralized realization from the certificates in Assumption~\ref{ass:spectral-certificates}. The canonical channels $K_1,K_2$ therefore define the \CCMAC{} geometry, whereas the graph-filtered local channels $\mathcal L_1,\mathcal L_2$ implement it; Proposition~\ref{prop:decentralized-bridge} controls the corresponding angle distortion. Raw matrix-product and neighbor-round costs are accounted for separately in Appendix~\ref{app:resources}.

\paragraph{Algorithmic organization.}
Let $M_j=K_j$ in the normalized canonical oracle and $M_j=\mathcal L_j$ in the decentralized implementation. Channelwise maps construct $\widehat M_j=S_jM_j$ with fixed spectral-equivalence accuracy; thus the hat denotes the normalized execution operator, not a claim that canonical $K_j$ is directly local. Algorithm~\ref{alg:cross-apapc} gives the complete outer iteration here; Appendix~\ref{app:algorithmproof} specifies $\operatorname{Cheb}_N$, normalization, certificates, and the proof. The convex continuation and nonsmooth interface are also deferred to their appendices.

\begin{algorithm}[H]
\caption{\textsc{Cross-APAPC}: accelerated affine optimization after separate channel normalization}
\label{alg:cross-apapc}
\begin{algorithmic}[1]
\Require $w^0$, $\widehat K$, $\widehat v$, smooth strongly convex $\Phi$, $L_\Phi$, $\mu_\Phi$, bounds $\lambda_\pm$, horizon $T$
\State $N\gets\lceil\sqrt{\lambda_+/\lambda_-}\rceil$, $\zeta\gets\mu_\Phi$
\State $\tau\gets\min\{1,\frac12\sqrt{19/(15\kappa_\Phi)}\}$, $\eta\gets(4\tau L_\Phi)^{-1}$, $\theta\gets15/(19\eta)$
\State $w_f^0\gets w^0$, $p^0\gets0$
\For{$k=0,1,\ldots,T-1$}
  \State $w_g^k\gets\tau w^k+(1-\tau)w_f^k$
  \State $w^{k+1/2}\gets(1+\eta\zeta)^{-1}[w^k-\eta(\nabla\Phi(w_g^k)-\zeta w_g^k+p^k)]$
  \State $r^k\gets\theta[w^{k+1/2}-\operatorname{Cheb}_N(w^{k+1/2};\widehat K,\widehat v)]$
  \State $p^{k+1}\gets p^k+r^k$
  \State $w^{k+1}\gets w^{k+1/2}-\eta(1+\eta\zeta)^{-1}r^k$
  \State $w_f^{k+1}\gets w_g^k+\frac{2\tau}{2-\tau}(w^{k+1}-w^k)$
\EndFor
\State \Return $w^T$ and its original-variable coordinate projection
\end{algorithmic}
\end{algorithm}

\begin{theorem}[Accelerated two-channel upper bound]
\label{thm:algorithm}
Fix an execution representation and suppose the inputs include a valid upper bound $R_E$ on the initial APAPC primal--dual energy defined in Appendix~\ref{app:algorithmproof}. Under the smooth strongly convex assumptions and fixed-accuracy separate normalization, Cross-APAPC returns an $\eps$-accurate \CCMAC{} primal solution with
\begin{equation}
\begin{aligned}
N_{\nabla f}&=O(T_{\rm sc,E}),&
N_{\rm aff}&=O(\widehat\chi T_{\rm sc,E}),\\
N_{\rm comm}&=O\!\left(\widehat\chi T_{\rm sc,E}\sum_j s_j^{\rm comm}\right),&
T_{\rm sc,E}&:=\sqrt{\kappa_f}\log_+(R_E/\eps).
\end{aligned}
\label{eq:sc-upper-compact}
\end{equation}
where $\log_+(t):=\max\{1,\log t\}$ and $\widehat\chi:=\sqrt{\lambda_+/\lambda_-}=O(\chix)$ for constant-factor-tight certificates. A neighbor-round specialization requires a supplied normalized-channel realization and accounts for all additional communication primitives. This is a fixed-instance energy-controlled guarantee, not a uniform claim over a class that controls only the primal initial radius. The fixed-$J$ normalized-oracle extension and the raw primitive ledger are in Appendices~\ref{app:weightedprojectors} and~\ref{app:resources}.
\end{theorem}

\paragraph{Primitive realization.}
The normalized count in Theorem~\ref{thm:algorithm} is converted to raw work only after fixing an implementation. For the certified construction of Appendix~\ref{app:resources}, put $T_{\Xi,E}:=\sqrt{\kappa_f}\,\Xi\log_+(R_E/\eps)$. Then
\[
N_A,N_B=O(T_{\Xi,E}\sqrt{\bar k_1}),\quad
N_C=O(T_{\Xi,E}\sqrt{\bar k_1\bar\kappa_C}),\quad
N_D=O(T_{\Xi,E}\sqrt{\bar k_2}),
\]
and
\[
N_{\rm aff}^{\rm prim}=O(T_{\Xi,E}\Lambda_{\rm in}),\qquad
N_{\rm comm}=O\!\left(T_{\Xi,E}(\sqrt{\bar k_1}+\sqrt{\bar k_2})\sqrt{\bar\kappa_W}\right).
\]
For constant-factor certificates $\Xi=O(\chix)$. If $B=0$, then $N_B=0$, but the $A,C,D$ and communication costs remain. This translation is the basis for comparing our normalized theorem-level unit with the native matrix products counted by \citet{yarmoshik2025coupled,yarmoshik2026mixed}.

\subsection{Minimax optimality}
\label{sec:minimax}
The formal parameter-bounded normalized-channel class and the infimum--supremum resource definition are given in Appendix~\ref{app:lower-bound-program}. The hard families separately isolate gradient, affine, and physical-communication resources while keeping the other legal resources free.

\begin{theorem}[Scope of the smooth strongly convex lower bounds]
\label{thm:minimax-sc}
For the normalized-channel linear-span model with budgets $\kappa_f\le\bar\kappa$, $\chix\le\Xi$, $\kappa_W\le\Omega$, and fixed-accuracy normalization, the constructions in Appendix~\ref{app:lower-bound-program} separately give the standard objective lower order and cross-angle affine lower order,
\begin{equation}
\mathfrak C_{\nabla}=\Omega\!\left(\sqrt{\bar\kappa}\log\frac1\eps\right),\qquad
\mathfrak C_{\rm aff}=\Omega\!\left(\sqrt{\bar\kappa}\,\Xi\log\frac1\eps\right),
\label{eq:minimax-compact}
\end{equation}
in the stated nontrivial parameter ranges. The current physical-path construction gives the radius-explicit communication lower bound stated there; it is not promoted here to a uniform $\Omega(\Xi\sqrt\Omega\log(1/\eps))$ equality at fixed absolute accuracy. Accordingly, Theorem~\ref{thm:algorithm} and the lower statements use their own explicit radius and oracle hypotheses.
\end{theorem}

\section{Convex and Nonsmooth Extensions}
\label{sec:beyond-sc}
The same normalized-channel geometry governs the other convex regimes. Appendix~\ref{app:smooth-convex-proof} defines the scale-invariant smooth-convex objective/feasibility criterion and multiplier budget $\Lambda_{\rm cvx}\le\bar\Lambda$; Appendix~\ref{app:nonsmooth-proof} states the corresponding full-domain nonsmooth assumptions and criterion. Write
\[
T_{\rm cvx}:=\sqrt{L_fR^2/\eps},\qquad
T_{\rm ns}:=M_{3R}R/\eps,\qquad
T_{\rm nsc}:=M_{3R}/\sqrt{\mu_f\eps}.
\]

\begin{theorem}[Log-free smooth-convex minimax complexity]
\label{thm:smooth-convex-logfree}
For every fixed multiplier budget $\bar\Lambda\ge1$ and $T_{\rm cvx}\ge50$,
\begin{equation}
\mathfrak C^{\rm cvx}_{\nabla}=\Theta(T_{\rm cvx}),\qquad
\mathfrak C^{\rm cvx}_{\rm aff}=\Theta(\Xi T_{\rm cvx}),\qquad
\mathfrak C^{\rm cvx}_{\rm comm}=\Theta(\Xi\sqrt\Omega\,T_{\rm cvx}).
\label{eq:cvx-logfree-compact}
\end{equation}
One explicit continuation method attains all three upper orders simultaneously. Its final feasibility repair costs only an additive $O_{\bar\Lambda}(\widehat\chi\log T_{\rm cvx})$ affine term, dominated by $O_{\bar\Lambda}(\widehat\chi T_{\rm cvx})$; hence there is no multiplicative target-accuracy logarithm.
\end{theorem}

\begin{theorem}[Nonsmooth cross-coupled upper bounds]
\label{thm:nonsmooth-upper}
For the full-domain convex criterion of Appendix~\ref{app:nonsmooth-proof}, and in a normalized model that additionally provides exact Euclidean projections onto the phase primal and dual balls, Cross-Sliding satisfies
\begin{equation}
N_{\partial f}=O(T_{\rm ns}^2),\quad N_{\rm aff}=O(\chix T_{\rm ns}),
\label{eq:ns-upper-main}
\end{equation}
and after restarting in the $\mu_f$-strongly convex case,
\begin{equation}
N_{\partial f}=O(T_{\rm nsc}^2),\quad N_{\rm aff}=O(\chix T_{\rm nsc}).
\label{eq:nsc-upper-main}
\end{equation}
The first-order counts have the standard centralized lower bounds; no matching $\chix$-dependent nonsmooth affine/communication lower bound is claimed. A physical decentralized round bound additionally requires an implementation and accounting of the global ball projections and scalar reductions, and is not inferred solely from a per-affine-call graph factor. For any fixed $J\ge2$, the normalized-oracle upper bounds replace $\chix$ by the projector-sum factor $\chiJ$; raw primitive and physical-round extensions require a regime-specific realization and are not obtained by analogy with the smooth-SC ledger.
\end{theorem}

\section{Experiments}
\label{sec:experiments}
The experiments are designed to test two different claims separately: (i) whether the theorem-level resource dependence on the cross geometry and the network appears with the predicted exponent, and (ii) whether the method remains stable on full, non-toy data and benefits from separate channel normalization under a common implementation.  The complete protocol, the controlled scaling plots, primitive-versus-normalized ledgers, and all additional tables are in Appendix~\ref{app:experiments}.

\paragraph{Full-data CC-MAC benchmark.}
We use vertically partitioned LIBSVM binary-regression instances \citep{changlin2011libsvm}. For $X=[X_1,\ldots,X_n]$ we solve
\begin{equation}
\min_w\ \frac{1}{2N}\|Xw-y\|^2+\frac{\lambda}{2}\|w\|^2,\qquad \lambda=10^{-2},
\label{eq:main-vfl-objective}
\end{equation}
through local weights $w_i$ and local copies $z_i\in\R^N$ of the shared prediction vector.  The two affine channels are
\begin{equation}
\frac1{\sqrt n}\left(\sum_{i=1}^nX_iw_i-\sum_{i=1}^n\alpha_i z_i\right)=0,
\qquad z_i-\frac1n\sum_{j=1}^nz_j=0,
\label{eq:main-vfl-channels}
\end{equation}
with $\sum_i\alpha_i=1$.  On the consensus subspace, \eqref{eq:main-vfl-channels} is exactly $Xw-z=0$ for every admissible $\alpha$, so changing $\alpha$ changes the two-channel geometry without changing the underlying feasible prediction relation.  This gives a representation-preserving controlled test of $\chix$, rather than a comparison of different optimization problems.

All reported runs use the complete training split, deterministic PyTorch FP64 on an NVIDIA H100, column-partitioned features, the same zero original-primal initialization, and a high-accuracy centralized reference. For the scaling and robustness suites, a run is successful when
\begin{equation}
\operatorname{RelErr}(w^k):=
\frac{\|w^k-w^\star\|^2}{\|w^0-w^\star\|^2}\le10^{-6}.
\label{eq:main-exp-stop}
\end{equation}
The method-comparison suites use the same threshold $10^{-6}$.
\paragraph{Full-data robustness.}
All $48/48$ Cross-APAPC configurations on full \texttt{a9a}, \texttt{mushrooms}, \texttt{w8a}, and \texttt{rcv1} reach \eqref{eq:main-exp-stop}; this includes the $47\,236$-feature sparse \texttt{rcv1} task with both $8$ and $16$ agents. Appendix~\ref{app:experiments} gives the dimensions, topologies, seeds, and primitive-resource ledger.

\paragraph{Five-seed method comparison.}
We compare six matrix-free implementations on full \texttt{a9a} and \texttt{mushrooms} with $n=8$, a ring network, feature-partition seeds $0,\ldots,4$, and target $10^{-6}$.  Cross-APAPC and joint APAPC share the same state representation, centralized reference, stopping rule, and normalized-channel interface.  The Yarmoshik-type generic coupled reduction, Tracking-ADMM, DPMM, and Condat--V{\~u} \citep{condat2013primal,vu2013splitting} curves are our CUDA ports/adaptations of the corresponding recurrences under this common benchmark interface; they are not presented as the authors' native software or native tuning.  Figures~\ref{fig:experiments-a9a}--\ref{fig:experiments-mushrooms} show every seed trajectory, their full empirical min--max envelope, and a median guide. Gradient evaluations and neighbor rounds are the common cross-method resources; native affine products are included as a diagnostic ledger but are not treated as equal-cost oracle units across different decompositions. All $60/60$ method--seed runs reach the target.

\begin{figure}[t]
\centering
\includegraphics[width=0.99\textwidth]{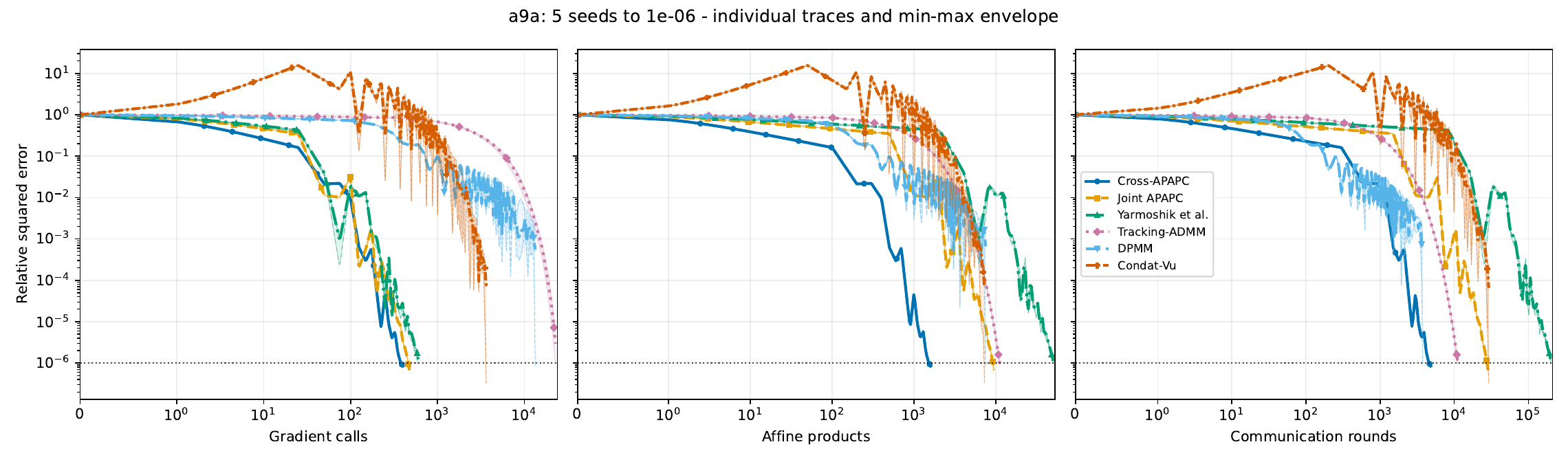}
\caption{Full-data \texttt{a9a} on H100 over five feature-partition seeds. Thin traces are individual seeds, shaded regions and dotted boundaries are the empirical min--max interval, and the thick line is the median. Panels use gradient evaluations, native affine products, and physical neighbor rounds; the horizontal dotted line is the $10^{-6}$ target.}
\label{fig:experiments-a9a}
\end{figure}

\begin{figure}[t]
\centering
\includegraphics[width=0.99\textwidth]{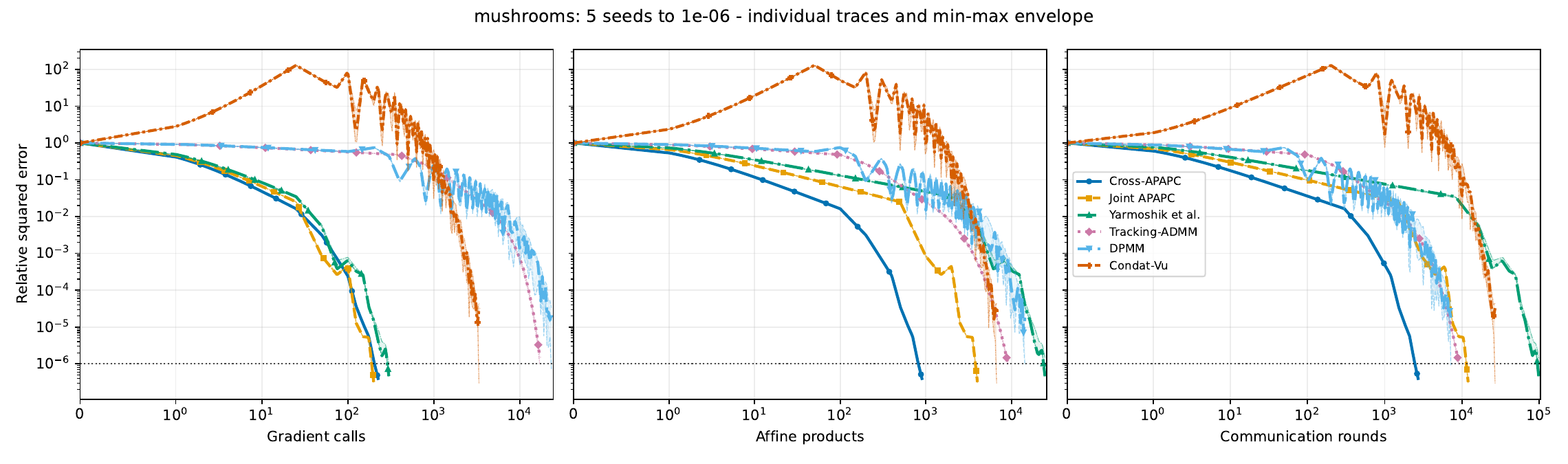}
\caption{Full-data \texttt{mushrooms} under the same five-seed protocol and interval convention as Figure~\ref{fig:experiments-a9a}. Cross-APAPC and joint APAPC differ only in separate versus joint channel normalization; the remaining curves are common-interface CUDA ports/adaptations detailed in Appendix~\ref{app:experiments}.}
\label{fig:experiments-mushrooms}
\end{figure}

The cleanest controlled algorithmic comparison is Cross versus Joint.  At the $10^{-6}$ threshold, the counts are identical across the five partitions for each method.  On \texttt{a9a}, Cross uses $400$ versus $475$ gradients, $800$ versus $4\,750$ normalized affine calls, and $4\,800$ versus $28\,500$ neighbor rounds; on \texttt{mushrooms}, the respective counts are $225$ versus $200$, $450$ versus $2\,000$, and $2\,700$ versus $12\,000$. Thus separate normalization reduces the normalized affine and communication ledgers by factors $5.94$ and $4.44$, without implying uniform gradient or wall-clock dominance. Appendix~\ref{app:experiments} reports the raw primitive interpretation.

\section{Discussion and Conclusions}
The results separate three sources of difficulty that are conflated by a single condition number: objective curvature, network propagation, and cross-channel geometry. After within-channel conditioning is removed, the two affine classes contribute the intrinsic factor $\chix$, while physical communication contributes $\sqrt{\kappa_W}$. Exact overlap is benign because common equations can be deduplicated; nearly coincident but nonidentical directions are hard because they leave a small positive singular value. The log-free smooth-convex result preserves this separation under its stated multiplier and spectral budgets.

This viewpoint is useful when the channel decomposition has operational meaning and cannot be replaced by arbitrary row mixing. Examples include vertical feature partitioning with a shared prediction variable, distributed estimation with local measurements coupled to a global state, and networked control or resource-allocation models combining nodewise physics with shared coordination constraints. In such problems separate normalization is implementable within the access policy, whereas a globally optimal joint preconditioner may require forbidden data aggregation or a different oracle. The Friedrichs factor then identifies when separate processing is sufficient and when the interface between otherwise well-conditioned channels is the actual bottleneck.

The experiments support this interpretation rather than a universal ranking of implementations. The controlled sweeps recover near-unit exponents in $\chix$ and $\chix\sqrt{\kappa_W}$; the full-data comparisons show a reproducible normalized-oracle advantage over joint normalization, including the additional $16$-agent \texttt{a8a} stress test in Appendix~\ref{app:experiments}. Native affine products and wall time remain implementation dependent, and the two-seed appendix interval is descriptive rather than inferential.

The scope is therefore deliberate. The nonsmooth result assumes exact projection access and claims no matching cross-dependent lower bound. A normalized affine call is not a universal raw-matvec or wall-clock unit; Appendix~\ref{app:resources} expands it into primitives. Important open questions are matching nonsmooth bounds, adaptive or cheaper normalization without prior spectral certificates, finite-sample tuning rules, and a sharp multi-channel geometry for $J\ge3$. These directions would turn the present two-channel structural characterization into a broader design principle for constrained decentralized optimization.

\clearpage
\subsection*{AI use statement}
In this work, we used generative AI tools to assist with developing and stress-testing theoretical arguments, formulating and refining mathematical claims and hypotheses, checking and editing proofs, reviewing experimental methodology and parameter choices, implementing and debugging research code, constructing synthetic test instances, interpreting diagnostic and experimental results, language translation where needed, surveying and summarizing related literature, and drafting, restructuring, and editing parts of the manuscript. We also used generative AI tools to assist with figure/table preparation, software-code editing, and reference formatting. We did not use generative AI to replace, clean, or reformat the original benchmark datasets or to perform qualitative or thematic data analysis; transcription of research material was not applicable. All AI-assisted mathematical claims and proofs were independently checked by the authors against the stated assumptions, cited literature was verified against the original sources, and AI-assisted code was executed and tested by the authors. We have reviewed all AI-assisted work and take responsibility for the final content of this submission, including text, claims, code, and figures produced with the aid of generative AI.

\appendix
\clearpage
\phantomsection
\section*{Notation Reference}
\label{app:notation}
\addcontentsline{toc}{section}{Notation Reference}

This table retains all notation reused across sections; proof-, algorithm-, experiment-, and construction-local symbols are defined where they appear.

\begingroup
\scriptsize
\setlength{\tabcolsep}{4pt}
\renewcommand{\arraystretch}{1.08}
\begin{longtable}{@{}>{\raggedright\arraybackslash}p{0.235\textwidth}>{\raggedright\arraybackslash}p{0.725\textwidth}@{}}
\caption{Global and recurring notation.}\label{tab:notation-reference}\\
\toprule
Symbol & Meaning \\
\midrule
\endfirsthead
\multicolumn{2}{c}{\tablename\ \thetable\ (continued)}\\
\toprule
Symbol & Meaning \\
\midrule
\endhead
\midrule
\multicolumn{2}{r}{Continued on next page}\\
\endfoot
\bottomrule
\endlastfoot

\multicolumn{2}{@{}l}{\textbf{Problem data and canonical reformulation}}\\[1pt]
$n$, $i$ & Number of agents and agent index, $i\in\{1,\ldots,n\}$. \\
$x_i$, $z$, $z_i$ & Local variable, shared variable, and local copy of the shared variable; $x_i\in\R^{d_i}$ and $z,z_i\in\R^q$. \\
$d_i,q,m,p_i,r_i$ & Local/shared dimensions, coupled-equality dimension, and row counts of $C_i,D_i$. \\
$A_i,B_i,b_i$; $C_i,c_i$; $D_i,e_i$ & Coupled-equality, local-affine, and shared-variable affine data in \eqref{eq:ccmac}. \\
$\bm x,\bm z$; $A,B,C,D$; $\bm b,\bm c,\bm e$ & Stacked variables, block-diagonal matrix stacks, and stacked right-hand sides. \\
$u$; $F(u)$ & Canonical stacked primal variable $u=\col(\bm x,\bm z)$ and lifted objective $F(u)=\sum_i f_i(x_i,z_i)$. \\
$w$; $\Phi(w)$ & Execution-space variable and objective after the stated canonical or decentralized lift. \\
$E_m,\Wp_m,\Wp_q$ & Normalized global-sum map and graph channels: $E_m=n^{-1/2}(\one^\top\otimes I_m)$, $\Wp_m=W\otimes I_m$, $\Wp_q=W\otimes I_q$. \\
$K_1,v_1$; $K_2,v_2$ & Two canonical affine channels and right-hand sides defined in \eqref{eq:channels}. \\
$K,v$ & Joint canonical stack $K=\col(K_1,K_2)$ and $v=\col(v_1,v_2)$. \\
$\mathcal L_j$ & Executable local channel $j$ in the decentralized lift. \\

\addlinespace[2pt]
\multicolumn{2}{@{}l}{\textbf{Network and linear-algebra quantities}}\\[1pt]
$\mathcal G=(\mathcal V,\mathcal E)$; $W$; $\kappa_W$ & Communication graph, symmetric PSD gossip matrix with $\Ker W=\operatorname{span}\{\one\}$, and $\kappa_W=\lambda_{\max}(W)/\lambda_{\min}^+(W)$; for $n=1$, $\kappa_W=1$ and the round count is zero. \\
$\Range(M),\Ker(M)$; $P_{\mathcal S}$ & Range and kernel of a linear map; Euclidean orthogonal projector onto a subspace $\mathcal S$. \\
$\sigma_{\max}(M),\sigma_{\min}^+(M)$ & Largest and smallest positive singular values. \\
$\kappa(M),\kappa_+(G)$ & Squared positive singular-value ratio and positive-eigenvalue ratio; $\kappa(M)=\kappa_+(M^\top M)$. \\

\addlinespace[2pt]
\multicolumn{2}{@{}l}{\textbf{Cross-channel geometry and normalization}}\\[1pt]
$J$; $M_j,U_j$ & Number of channel classes; channel maps and row spaces $U_j=\Range(M_j^\top)$. \\
$G_J,\chi_{\times,J}$ & Projector sum $G_J=\sum_jP_{U_j}$ and multi-channel spectral factor $\chi_{\times,J}=\sqrt{\kappa_+(G_J)}$. \\
$U,V$; $I$ & Two-channel row spaces and their exact intersection $I=U\cap V$. \\
$c_F(U,V),\theta_F$ & Friedrichs cosine after removing $I$ and Friedrichs angle $\theta_F=\arccos c_F$. \\
$\chix$ & Canonical factor $\sqrt{(1+c_F)/(1-c_F)}=\cot(\theta_F/2)$ for $U=\Range(K_1^\top)$ and $V=\Range(K_2^\top)$. \\
$\rho_{\rm loc},\chi_{\rm loc}$ & Friedrichs cosine and cross factor of the executable local pair. \\
$S_j$; $\widehat K_j,\widehat v_j$ & Separate left normalizer and normalized channel/right-hand-side pair. \\
$\widehat K,\widehat v$; $\bar K,\bar v$ & Generic normalized joint stack; $\bar K,\bar v$ denote the fixed normalized pair in the convex and nonsmooth appendices. \\
$\delta_{\rm wh}$ & Fixed relative whitening accuracy, chosen independently of the target accuracy $\eps$. \\
$\Xi$ & Supplied upper certificate on the relevant cross factor, in particular $\Xi\ge\chix$ when stated. \\

\addlinespace[2pt]
\multicolumn{2}{@{}l}{\textbf{Optimization, accuracy, and resources}}\\[1pt]
$L_f,\mu_f,\kappa_f$ & Objective smoothness, strong convexity, and condition number $\kappa_f=L_f/\mu_f$. \\
$w^0,w^\star$; $R_0,R_E,R$ & Initialization/solution, primal radius, supplied smooth-SC initial primal--dual energy bound, and convex/nonsmooth radius budget. \\
$\eps$; $\log_+(t)$ & Target accuracy and the truncated logarithm used in the smooth-SC statements. \\
$N_{\nabla f},N_{\partial f}$ & Numbers of gradient and subgradient evaluations. \\
$N_{\rm aff},N_{\rm comm}$ & Numbers of normalized affine-oracle calls and physical neighbor communication rounds. \\
$N_A,N_B,N_C,N_D$ & Raw products with the primitive affine blocks in a fixed executable realization. \\
$s_j^{\rm comm}$ & Communication cost of one normalized application of channel $j$ in the stated realization. \\
$T_{\rm sc},T_{\rm sc,E},T_{\rm cvx},T_{\rm ns},T_{\rm nsc}$ & Standard first-order scales used in the smooth-SC, smooth-convex, nonsmooth-convex, and nonsmooth-SC statements. \\
$\Lambda_{\rm cvx}$ & Scale-invariant smooth-convex multiplier parameter $\sigma_{\min}^+(\bar K)\|\lambda^\star\|/(L_fR)$. \\
$M_{3R}$ & Local nonsmooth regularity bound on the Euclidean ball $B(w^\star,3R)$. \\
$\bar\kappa,\Xi,\Omega$ & Class budgets on $\kappa_f$, the cross-factor certificate, and $\kappa_W$, respectively. \\

\end{longtable}
\endgroup

\paragraph{Appendix roadmap.}
For reviewer accessibility, Appendix~\ref{app:experiments} first gives the complete experimental protocol and the theorem-aligned scaling validation.  The remaining appendices follow the proof dependency graph: canonical reduction and the decentralized lift; projector/Friedrichs geometry; optimal separate preconditioning; generic normalization and the certified CC-MAC realization; smooth strongly convex upper and lower bounds; smooth-convex and nonsmooth extensions; and primitive resource accounting.

\section{Experimental Protocol, Scaling Validation, and Full H100 Results}
\label{app:experiments}

\paragraph{Hardware, arithmetic, and stopping rule.}
All reported experiments were executed with PyTorch in FP64 on an NVIDIA H100 80GB HBM3 (CUDA 12.8, PyTorch 2.11.0+cu128), with deterministic execution enabled and TF32 disabled.  Sparse feature blocks are stored as CUDA CSR tensors unless the density threshold selects a dense representation.  All methods start from the same zero original-primal point and are evaluated against a high-accuracy centralized reference solution.  We report the relative squared error in \eqref{eq:main-exp-stop}; all scaling, robustness, and method-comparison suites use the $10^{-6}$ stopping threshold.  For each run, our implementation stores a JSON summary and a CSV trace; these records are used to aggregate the reported curves and tables.

\paragraph{Full-data VFL formulation.}
The benchmark is \eqref{eq:main-vfl-objective} with the two CC-MAC channels \eqref{eq:main-vfl-channels}.  For coefficients $\alpha\in\R^n$ satisfying $\sum_i\alpha_i=1$, consensus implies the same relation $Xw-z=0$ independently of $\alpha$.  Hence the controlled geometry experiment varies the representation without changing the original feasible prediction relation or objective.  If $\chi_{\rm targ}\ge1$ and $c=(\chi_{\rm targ}^2-1)/(\chi_{\rm targ}^2+1)$, we choose a unit zero-sum vector $v$ and
\[
\alpha=\frac1n\one+\frac{c}{\sqrt{n(1-c^2)}}v,
\]
which realizes the requested Friedrichs factor up to floating-point roundoff.

\paragraph{Resource accounting.}
A \emph{normalized affine call} is one application of the separately normalized joint stack or its adjoint, exactly as in the theorem-level oracle.  The experiment counter also expands this call into \emph{raw affine products}: the polynomial $K_1$ normalizer contributes its internal $K_1/K_1^\top$ products, while $K_2$ actions charge physical neighbor rounds according to the measured graph condition.  These are different ledgers.  The main-text Cross-versus-Joint statement is made in the common normalized-channel oracle; raw primitive counts are reported separately below.  No wall-clock superiority claim is made.

\subsection{Controlled geometry and network scaling}
The geometry sweep uses the complete \texttt{a9a} and \texttt{w8a} training sets, $n=8$ agents, data seed $307$, feature-partition seed $0$, normalizer degree $28$, and
\[
\chi_{\rm targ}\in\{1,1.25,1.5,2,3,5,8\}.
\]
We combine these values with path, ring, grid, Erd\H{o}s--R\'enyi, and random-regular networks.  Deterministic topologies use one graph realization; the two random graph families use five graph seeds, giving $182$ runs in total.  Every run reaches the target error.  The log--log fits in Figure~\ref{fig:experiments-scaling} use the nonplateau regime $\chix\ge1.5$.  Resampling graph realizations as blocks gives
\begin{equation}
\alpha_{\rm aff}=0.988\ [0.966,1.009],\qquad
\alpha_{\rm comm}=0.980\ [0.959,1.003],
\label{eq:experimental-slopes}
\end{equation}
for $N_{\rm aff}\propto\chix^{\alpha_{\rm aff}}$ and $N_{\rm comm}\propto(\chix\sqrt{\kappa_W})^{\alpha_{\rm comm}}$, respectively.  A two-factor fit gives exponent $0.988$ in $\chix$ and $0.957$ in $\sqrt{\kappa_W}$; the graph-block bootstrap interval for the latter is $[0.886,1.039]$.

\begin{figure}[t]
\centering
\includegraphics[width=0.96\textwidth]{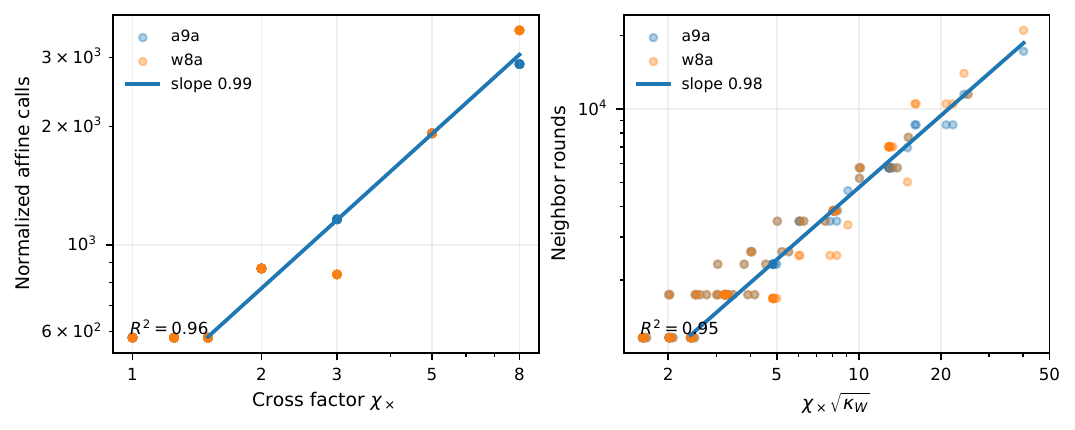}
\caption{Controlled full-data scaling on \texttt{a9a} and \texttt{w8a}.  Left: normalized affine calls versus the realized cross factor $\chix$.  Right: physical neighbor rounds versus $\chix\sqrt{\kappa_W}$.  Fits use the nonplateau regime $\chix\ge1.5$; both slopes are statistically consistent with the unit exponents predicted by the theory.}
\label{fig:experiments-scaling}
\end{figure}

\subsection{Full-data robustness}
We use feature-norm weights $\alpha_i\propto\|X_i\|_F$, polynomial separate normalization of degree $24$, and full training sets for \texttt{a9a}, \texttt{mushrooms}, \texttt{w8a}, and \texttt{rcv1}.  We run $n\in\{8,16\}$, two feature-partition seeds, and ring/random-regular networks (with two random-regular graph seeds), for $48$ total configurations.  All $48$ runs converge.  Table~\ref{tab:full-data-h100} reports medians and interquartile ranges over the available topology/seed realizations for each dataset/agent count; ``raw aff.'' expands the degree-24 channel normalizer.

\begin{table}[t]
\centering\scriptsize
\caption{Full-data Cross-APAPC robustness on H100. Brackets denote $[Q_1,Q_3]$; the final column is the worst terminal relative squared error among the corresponding runs.}
\label{tab:full-data-h100}
\setlength{\tabcolsep}{3.2pt}
\resizebox{\textwidth}{!}{\begin{tabular}{@{}lrrrrrrrr@{}}
\toprule
Dataset & $N$ & $d$ & agents & grad. & norm. aff. & raw aff. & comm. & max err.\\
\midrule
\texttt{a9a} & 32,561 & 123 & 8 & 145 & 580 & 29\,000 & 1\,160 [1\,160, 1\,595] & $3.9\times 10^{-5}$\\
\texttt{a9a} & 32,561 & 123 & 16 & 145 & 580 & 29\,000 & 1\,740 [1\,740, 3\,045] & $3.9\times 10^{-5}$\\
\texttt{mushrooms} & 8,124 & 112 & 8 & 110 & 440 & 22\,000 & 880 [880, 1\,210] & $8.1\times 10^{-5}$\\
\texttt{mushrooms} & 8,124 & 112 & 16 & 145 & 580 & 29\,000 & 1\,740 [1\,740, 3\,045] & $5.6\times 10^{-5}$\\
\texttt{w8a} & 49,749 & 300 & 8 & 145 & 580 & 29\,000 & 1\,160 [1\,160, 1\,595] & $8.3\times 10^{-5}$\\
\texttt{w8a} & 49,749 & 300 & 16 & 220 & 880 & 44\,000 & 2\,640 [2\,640, 4\,620] & $7.9\times 10^{-5}$\\
\texttt{rcv1} & 20,242 & 47,236 & 8 & 150 & 600 & 30\,000 & 1\,200 [1\,200, 1\,650] & $9.7\times 10^{-5}$\\
\texttt{rcv1} & 20,242 & 47,236 & 16 & 170 & 680 & 34\,000 & 2\,040 [2\,040, 3\,570] & $7.4\times 10^{-5}$\\
\bottomrule
\end{tabular}}
\end{table}

The high-dimensional \texttt{rcv1} instance has $20\,242$ samples and $47\,236$ features; it reaches the target for both $8$ and $16$ agents.

\subsection{Five-seed full-data method comparison}
For the broader comparison we use full \texttt{a9a} and \texttt{mushrooms}, $n=8$, a ring network, feature-partition seeds $0,\ldots,4$, and tolerance $10^{-6}$.  Cross-APAPC and joint APAPC share the same matrix-free state, reference solution, and normalized-oracle convention.  The Yarmoshik-type generic coupled reduction, Tracking-ADMM, DPMM, and Condat--V{\~u} curves in Figures~\ref{fig:experiments-a9a}--\ref{fig:experiments-mushrooms} are our CUDA matrix-free ports/adaptations of the corresponding recurrences under the common benchmark interface; they are not claimed to be the authors' native software or native tuning. All $60$ method--seed combinations converge; the largest terminal relative squared error is $9.934\cdot10^{-7}$.

The figures carry the seed-level information directly: thin curves show all five partitions and the envelope is the full empirical min--max interval. The middle panels report each implementation's native affine-action ledger for reproducibility; unlike gradients and physical neighbor rounds, it is not interpreted as a directly comparable oracle unit across methods.

\paragraph{Normalized versus raw Cross/Joint accounting.}
Because Cross-APAPC and joint APAPC are two realizations of the same APAPC template, their normalized-call comparison is clean. At tolerance $10^{-6}$ Cross reduces normalized affine calls and neighbor rounds by factors $5.94$ on \texttt{a9a} and $4.44$ on \texttt{mushrooms}. The separate normalizer nevertheless expands each normalized $K_1$ action into raw primitives, so the native affine-action ledger is reported independently in the middle panels rather than identified with the theorem-level oracle. The figures retain all five traces and their full empirical intervals instead of replacing them by a second median table; no raw-matvec or wall-clock dominance is claimed.

\subsection{Additional 16-agent \texttt{a8a} comparison}
As a larger-agent stress test, we repeat the six-method comparison on the complete \texttt{a8a} training split ($22\,696$ samples and $123$ features) with $n=16$, a ring network, feature-partition seeds $0$ and $1$, and tolerance $10^{-6}$.  All $12/12$ method--seed runs converge.  The worst terminal relative squared error is $9.971\cdot10^{-7}$.  Cross-APAPC reaches the target in $375$ gradient evaluations for both partitions, joint APAPC in $400$, and the Yarmoshik-type reduction in $650$; the remaining methods require between $12\,100$ and $33\,900$ outer iterations depending on the seed and recurrence.  These iteration numbers are descriptive rather than equal-cost comparisons, so Figure~\ref{fig:experiments-a8a-n16} retains the common gradient and neighbor-round ledgers and reports native affine actions separately.

\begin{figure}[H]
\centering
\includegraphics[width=0.99\textwidth]{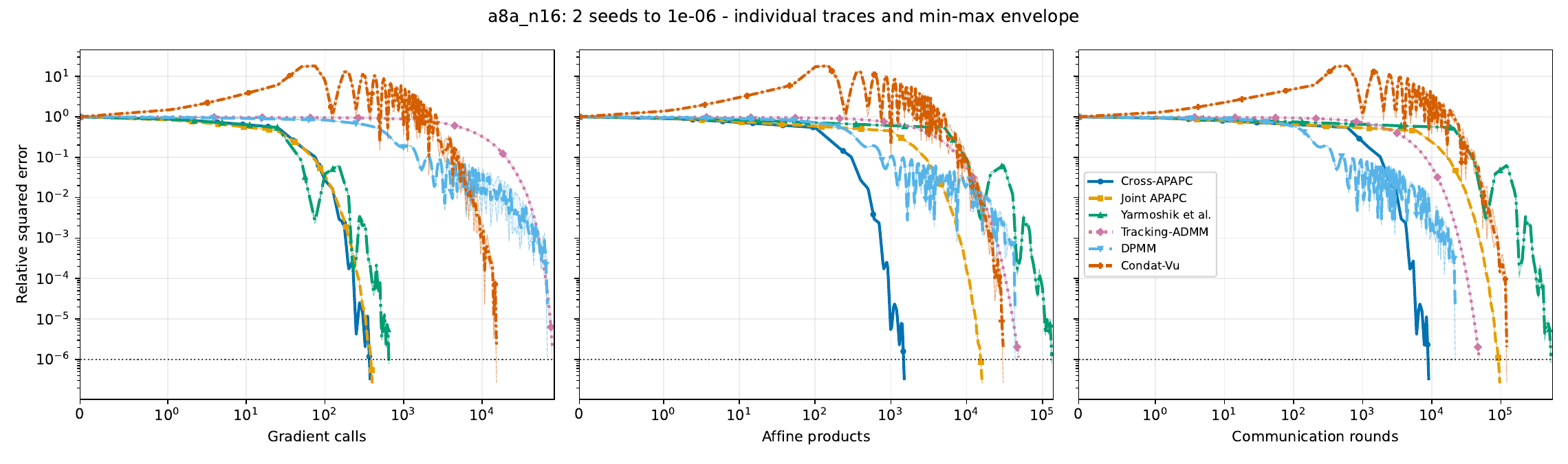}
\caption{Full-data \texttt{a8a} with $16$ agents on H100. Thin curves are the two feature-partition seeds; the shaded band and dotted boundaries show their empirical min--max range, and the thick curve is the median guide. All six methods reach the $10^{-6}$ target. With only two seeds, the band is a descriptive seed interval, not a confidence interval.}
\label{fig:experiments-a8a-n16}
\end{figure}

\paragraph{Reproducibility details.}
The experimental pipeline uses H100 YAML configurations together with dataset/partition and graph generators, CUDA sparse/dense operators, and explicit counters for gradients, normalized affine calls, raw affine products, and neighbor rounds. Per-run JSON summaries and CSV traces are aggregated to generate Figures~\ref{fig:experiments-a9a}--\ref{fig:experiments-a8a-n16}, Figure~\ref{fig:experiments-scaling}, and Table~\ref{tab:full-data-h100}. We also evaluated direct smooth-convex $\varepsilon$-sweeps and normalization ablations, but they are not used for any empirical claim reported here.

\section{Canonical Reduction and Decentralized Lift}\label{app:equivalence}
\begin{lemma}[Network range identities]\label{lem:network-range}
For every $s\ge1$, $\Ker(W\otimes I_s)=\operatorname{span}\{\one\}\otimes\R^s$ and $\Range(W\otimes I_s)=\{\col(u_1,\ldots,u_n):\sum_i u_i=0\}$.
\end{lemma}
For the executable lift introduce $y\in\mathcal Y_\perp:=\Range(W\otimes I_m)$ and impose
\begin{equation}A\bm x+B\bm z+\alpha\Wp_m y=\bm b,\quad \alpha>0,\label{eq:coupled-localized}\end{equation}
with optional consensus-row scale $\gamma>0$, giving
\begin{equation}K^{\rm loc}_{\alpha,\gamma}=\begin{bmatrix}A&B&\alpha\Wp_m\\ C&0&0\\0&D&0\\0&\gamma\Wp_q&0\end{bmatrix}.\label{eq:Ksplit}\end{equation}
\begin{proposition}[Lift-scaling invariance]\label{prop:lift-invariance}
The canonical pair $\Range(K_1^\top),\Range(K_2^\top)$ is independent of $\alpha,\gamma$ and of the positive spectrum/edge weights of any connected symmetric gossip matrix with the same consensus kernel.
\end{proposition}
\begin{proposition}[Canonical-to-local decentralized bridge]
\label{prop:decentralized-bridge}
Assume $n\ge2$. Let $\widetilde W=P_W(W)$, where $P_W(0)=0$, and suppose
\[
0<\underline\sigma_W\le\sigma_{\min}^+(\widetilde W)
\le\sigma_{\max}(\widetilde W)\le\overline\sigma_W,
\qquad
\overline\sigma_W/\underline\sigma_W=O(1).
\]
Put $H=[A\ B]$.
For $H\ne0$, choose a certified $\bar h\ge\|H\|^2$ and set
$\alpha=\sqrt{3\bar h}/\underline\sigma_W$, with $\gamma>0$.
Writing $\widetilde W_s:=\widetilde W\otimes I_s$, define
\begin{equation}
L_1=
\begin{bmatrix}A&B&\alpha\widetilde W_m\\ C&0&0\end{bmatrix},
\qquad
L_2=
\begin{bmatrix}0&D&0\\0&\gamma\widetilde W_q&0\end{bmatrix}.
\label{eq:local-bridge-channels}
\end{equation}
Let $U_{\rm loc}:=\Range(L_1^\top)$, $V_{\rm loc}:=\Range(L_2^\top)$,
$\rho_{\rm loc}:=c_F(U_{\rm loc},V_{\rm loc})$, and
$\chi_{\rm loc}:=\sqrt{(1+\rho_{\rm loc})/(1-\rho_{\rm loc})}$.
On the lifted space with $y\in\mathcal Y_\perp$,
\begin{equation}
\rho_{\rm loc}^2\le\tfrac14+\tfrac34c_F^2,
\qquad
\chi_{\rm loc}\le\sqrt3\,\chix.
\label{eq:bridge-angle}
\end{equation}
The feasible-set penalty below has $\kappa_{\Phi_{\rm loc}}=O(\kappa_f)$.
Each literal $L_j/L_j^\top$ call uses local products and one
$\widetilde W$ action, costing $\deg P_W$ neighbor rounds.
The certified graph polynomial is specified in
Appendix~\ref{app:ccmac-normalization}.
If $H=0$, remove the vacuous coupled equation and auxiliary
variable, take $\Phi_{\rm loc}=F$, and retain only the local $C$ channel and
the shared/consensus channel; the two row spaces are orthogonal.
\end{proposition}
\begin{corollary}[Recovery of the previous mixed model]\label{cor:old-orthogonal}
If $B_i=0$ for all $i$, then $\Range(K_1^\top)\perp\Range(K_2^\top)$ and $\chix=1$.
\end{corollary}
\subsection{Proof of Lemma~\ref{lem:network-range}}
Since $W=W^\top\succeq0$ is connected, $\Ker W=\operatorname{span}\{\one\}$.  The Kronecker identity gives $\Ker(W\otimes I_s)=\operatorname{span}\{\one\}\otimes\R^s$.  Symmetry implies $\Range(W\otimes I_s)=\Ker(W\otimes I_s)^\perp$, which is exactly the set of block vectors with zero block sum. \qed

\subsection{Proof of Theorem~\ref{thm:equivalence}}
Let $r:=A\bm x+B\bm z-\bm b$.  The original coupled equality is $E_mr=0$.  Since $\Ker E_m=\{r:\sum_i r_i=0\}=\Range(\Wp_m)$ by Lemma~\ref{lem:network-range},
\[
E_mr=0\quad\Longleftrightarrow\quad r\in\Range(\Wp_m)
\quad\Longleftrightarrow\quad \exists y:\ r+\alpha\Wp_my=0
\]
for every $\alpha>0$.  Likewise $\Wp_q\bm z=0$ iff $z_1=\cdots=z_n$, and multiplying this row by $\gamma>0$ changes neither its nullspace nor feasibility.  The $C$ and $D$ rows are unchanged.  Therefore the original CC-MAC system, the canonical reduced system $K_1u=v_1,K_2u=v_2$, and every local lift \eqref{eq:Ksplit} have the same feasible original variables and the same objective value. \qed

\subsection{Proof of Proposition~\ref{prop:lift-invariance}}
The canonical matrices in \eqref{eq:channels} contain neither $\alpha$ nor $\gamma$, so their row spaces are manifestly invariant to these auxiliary scalings.  If $W'$ is any other symmetric positive-semidefinite gossip matrix on a connected graph with $\Ker W'=\operatorname{span}\{\one\}$, then
\[
\Range((W'\otimes I_q)^\top)=\Range(W'\otimes I_q)
=(\operatorname{span}\{\one\}\otimes\R^q)^\perp
=\Range(\Wp_q).
\]
Thus replacing the consensus realization changes its positive spectrum and hence its implementation cost, but not the row subspace entering the Friedrichs angle.  The $E_m$ block is the canonical normalized global-sum map and is independent of the graph altogether.  Hence $c_F$ and $\chix$ depend only on the CC-MAC affine data and the fixed Euclidean geometry of $(\bm x,\bm z)$. \qed

\subsection{Proof of Proposition~\ref{prop:decentralized-bridge}}
If $H=0$, then $A=B=0$ and feasibility makes the original coupled
equation vacuous. Removing its auxiliary variable leaves $F$,
local $C$ constraints on $\bm x$, and shared/consensus constraints
on $\bm z$. Their row spaces are orthogonal, so
$\chi_{\rm loc}=\chix=1$. Hence assume $H\ne0$.

Let $\mathcal H_0$ be the original-coordinate subspace and
$\mathcal Y$ the auxiliary-coordinate subspace.
Set $U_\alpha=\Range(L_1^\top)$ and $V=\Range(L_2^\top)$.
Then $V\subset\mathcal H_0$.
A row-space vector of $L_1$ is
\[
(A^\top a+C^\top c,\ B^\top a,\ \alpha\widetilde W_m a).
\]
Decompose $a=a_0+a_\perp$ into consensus and disagreement.
Its auxiliary component is zero precisely when $a_\perp=0$,
and the vectors generated by $(a_0,c)$ form the embedded
canonical row space $U_0=\Range(K_1^\top)$. Thus
\begin{equation}
U_\alpha\cap\mathcal H_0=U_0.
\label{eq:bridge-intersection}
\end{equation}
Write $U_\alpha=U_0\oplus U_e$ orthogonally.
Every $e\in U_e$ is the row
$(H^\top a_\perp,\alpha\widetilde W_m a_\perp)$ minus its
projection onto $U_0$. Therefore
\[
\|P_{\mathcal H_0}e\|\le\|H\|\|a_\perp\|,
\qquad
\|P_{\mathcal Y}e\|\ge
\alpha\underline\sigma_W\|a_\perp\|.
\]
Consequently
\begin{equation}
\|P_{\mathcal H_0}e\|\le\delta\|e\|,
\qquad
\delta:=\frac{\|H\|}
{\sqrt{\|H\|^2+\alpha^2\underline\sigma_W^2}}\le\frac12.
\label{eq:bridge-leakage}
\end{equation}
The intersections $U_\alpha\cap V$ and $U_0\cap V$ agree.
For a unit $v\in V$ orthogonal to this intersection, put
$t=\|P_{U_0}v\|^2\le c_F^2$.
Since $P_{\mathcal H_0}U_e\subset U_0^\perp$,
\[
\|P_{U_e}v\|\le\delta\sqrt{1-t},
\qquad
\|P_{U_\alpha}v\|^2
\le t+\delta^2(1-t)
\le\tfrac14+\tfrac34c_F^2.
\]
This proves the first bound in \eqref{eq:bridge-angle}.
For $s\in[0,1)$,
\[
(1+2s)^2-\frac{(2+s)^2}{4}(1+3s^2)
=\frac34s(1-s^2)(s+4)\ge0.
\]
Taking $s=c_F$ gives
$\rho_{\rm loc}(2+c_F)\le1+2c_F$, which rearranges to
$\chi_{\rm loc}^2\le3\chix^2$.

Define
\begin{equation}
\Phi_{\rm loc}(\xi):=F(\bm x,\bm z)
+\frac r2\|A\bm x+B\bm z+\alpha\widetilde W_m y-\bm b\|^2,
\qquad
\xi=\col(\bm x,\bm z,y),\quad r:=\frac{\mu_f}{2\bar h}.
\label{eq:bridge-penalty}
\end{equation}
The penalty vanishes on the localized equality.
For an increment $(d_u,d_y)$ with $d_y\in\mathcal Y_\perp$,
its Bregman divergence $D_{\Phi_{\rm loc}}(\xi+(d_u,d_y),\xi)$ satisfies
\[
D_{\Phi_{\rm loc}}(\xi+(d_u,d_y),\xi)\ge
\frac{\mu_f}{2}\|d_u\|^2+
\frac r2\|Hd_u+\alpha\widetilde W_m d_y\|^2
\ge
\frac{\mu_f}{4}\|d_u\|^2+
\frac{3\mu_f}{8}\|d_y\|^2.
\]
Here we used
$\|a+b\|^2\ge\frac12\|b\|^2-\|a\|^2$ and
$\|H\|^2\le\bar h$.
Thus $\mu_{\Phi_{\rm loc}}\ge\mu_f/2$, whereas
\[
L_{\Phi_{\rm loc}}\le L_f+r(\|H\|^2+\alpha^2\overline\sigma_W^2)
\le L_f+\frac{\mu_f}{2}
 \left(1+3\frac{\overline\sigma_W^2}{\underline\sigma_W^2}\right).
\]
Hence $\kappa_{\Phi_{\rm loc}}=O(\kappa_f)$.
Initialization $y^0=0$ and the stated channel and gradient actions
preserve $\mathcal Y_\perp$.
A polynomial of degree $\deg P_W$ is evaluated with that many
raw $W$ products. The map $E_m$ defines the canonical geometry
and is not evaluated by this lift. \qed

\subsection{Proof of Corollary~\ref{cor:old-orthogonal}}
When $B=0$, every vector in $\Range(K_1^\top)$ has support only in the $\bm x$ coordinates, whereas every vector in $\Range(K_2^\top)$ has support only in the $\bm z$ coordinates.  The spaces are orthogonal, so $c_F=0$ and $\chix=1$. \qed

\section{Projector Geometry, Multi-Channel Reduction, and Affine Deduplication}\label{app:weightedprojectors}
\subsection{Two weighted projectors}
\begin{lemma}[Spectrum of two weighted projectors]\label{lem:weighted-projectors}
Let $a,b>0$. On a principal-angle plane of $U,V$ with angle $\theta\in(0,\pi/2]$, the eigenvalues of $aP_U+bP_V$ are
\begin{equation}\lambda_\pm(\theta;a,b)=\frac{a+b\pm\sqrt{(a-b)^2+4ab\cos^2\theta}}{2}.\label{eq:projector-spectrum}\end{equation}
The common component $U\cap V$ has eigenvalue $a+b$; exclusive directions have eigenvalue $a$ or $b$. Among principal planes the smallest eigenvalue occurs at the Friedrichs cosine.
\end{lemma}
\subsection{Proof of Lemma~\ref{lem:weighted-projectors}}
Use the canonical principal-angle decomposition of $U+V$. Exact intersections, directions belonging to only one subspace, and two-dimensional principal-angle planes are mutually orthogonal. On a nontrivial plane choose an orthonormal basis $(u,w)$ such that $u\in U$ and a unit vector spanning the corresponding direction of $V$ is
\[
v=\cos\theta\,u+\sin\theta\,w,
\qquad \theta\in(0,\pi/2].
\]
Then
\[
P_U=\begin{bmatrix}1&0\\0&0\end{bmatrix},
\qquad
P_V=\begin{bmatrix}
\cos^2\theta&\cos\theta\sin\theta\\
\cos\theta\sin\theta&\sin^2\theta
\end{bmatrix}.
\]
Therefore
\[
aP_U+bP_V=
\begin{bmatrix}
a+b\cos^2\theta&b\cos\theta\sin\theta\\
b\cos\theta\sin\theta&b\sin^2\theta
\end{bmatrix}.
\]
Its trace is $a+b$ and its determinant is $ab\sin^2\theta$. The quadratic formula gives
\[
\lambda_\pm
=\frac{a+b\pm\sqrt{(a+b)^2-4ab\sin^2\theta}}2
=\frac{a+b\pm\sqrt{(a-b)^2+4ab\cos^2\theta}}2,
\]
which is \eqref{eq:projector-spectrum}. The derivative of $\lambda_-$ with respect to $|\cos\theta|$ is nonpositive, so among nontrivial principal-angle planes its smallest value occurs at the largest nonredundant cosine, i.e., the Friedrichs cosine. One-dimensional exclusive components have eigenvalue $a$ or $b$ and therefore cannot create a smaller value than the formula with $\cos\theta=0$. \qed

\subsection{Multi-channel reduction}
For $J$ separately normalized channel classes with row spaces $U_1,\ldots,U_J$, define
\begin{equation}
G_J:=\sum_{j=1}^J P_{U_j},\qquad \chiJ:=\sqrt{\kappa_+(G_J)}.
\label{eq:multi-channel-factor}
\end{equation}
If all channels are vacuous we set $\chiJ=1$ and charge no affine calls. Pairwise angles need not determine $G_J$ when $J\ge3$.
\begin{proposition}[Multi-channel spectral reduction]
\label{prop:multi-channel}
Suppose separate normalizations satisfy, for universal $0<a\le b$,
\[
aP_{U_j}\preceq \widehat M_j^\top\widehat M_j\preceq bP_{U_j},
\qquad j=1,\ldots,J,
\]
and let $\widehat M=\col(\widehat M_1,\ldots,\widehat M_J)$. Then
\begin{equation}
aG_J\preceq \widehat M^\top\widehat M\preceq bG_J,
\qquad
\kappa(\widehat M)\le \frac ba\,\chiJ^2.
\end{equation}
Exact separate whitening gives equality $\kappa(\widehat M)=\chiJ^2$. For $J=1$, $\chi_{\times,1}=1$. For $J=2$, if $U=U_1$, $V=U_2$ and $\chix$ is the Friedrichs factor defined in Section~\ref{sec:angle}, then
\begin{equation}
\chix\le \chi_{\times,2}\le \sqrt2\,\chix,
\label{eq:two-vs-multi-factor}
\end{equation}
with $\chi_{\times,2}=\chix$ when $U\cap V=\{0\}$.
\end{proposition}
\begin{proof}
Summing the $J$ spectral-equivalence inequalities gives the displayed PSD sandwich; exact whitening has $\widehat M_j^\top\widehat M_j=P_{U_j}$. The $J=2$ comparison follows from Lemma~\ref{lem:weighted-projectors}: exact common directions contribute only eigenvalue $2$, while every nontrivial principal-angle block contributes $1\pm\cos\theta$. Hence $\kappa_+(P_U+P_V)$ lies between $\chix^2$ and $2\chix^2$, with equality to $\chix^2$ in the irredundant case.
\end{proof}

\begin{lemma}[Pairwise angles do not determine multi-channel conditioning]
\label{lem:pairwise-insufficient}
For every $c\in(0,1/2)$ there are two triples of one-dimensional subspaces with identical pairwise Friedrichs cosines, all equal to $c$, but different spectra of $G_3=\sum_{j=1}^3P_{U_j}$.
\end{lemma}
\begin{proof}
Choose unit generators whose Gram matrices are
\[
H_+=\begin{bmatrix}1&c&c\\c&1&c\\c&c&1\end{bmatrix},\qquad
H_-=\begin{bmatrix}1&c&c\\c&1&-c\\c&-c&1\end{bmatrix}.
\]
Both are positive definite for $0<c<1/2$, and all off-diagonal absolute inner products equal $c$. Their eigenvalues are respectively
$\{1-c,1-c,1+2c\}$ and $\{1-2c,1+c,1+c\}$. The nonzero eigenvalues of $\sum_jP_{U_j}$ equal those of the generator Gram matrix, so the two projector sums have different condition numbers despite identical pairwise angles.
\end{proof}

\subsection{Exact common affine directions}

\begin{lemma}[Affine deduplication of exact common row directions]
\label{lem:affine-dedup}
Let $M_j\in\mathbb R^{p_j\times d}$ and suppose
\[
\mathcal F:=\{x:M_1x=v_1,\ M_2x=v_2\}\neq\varnothing.
\]
Set
\[
U=\Range(M_1^\top),\qquad V=\Range(M_2^\top),\qquad I=U\cap V,
\]
and choose orthogonal decompositions
\[
U=I\oplus U_0,\qquad V=I\oplus V_0
\]
with orthonormal basis matrices $Q_I,Q_1,Q_2$ for $I,U_0,V_0$, respectively.
For any $x_\star\in\mathcal F$, put
\[
c_I=Q_I^\top x_\star,\qquad c_1=Q_1^\top x_\star,\qquad c_2=Q_2^\top x_\star.
\]
Then
\[
\mathcal F
=
\{x:Q_I^\top x=c_I,\ Q_1^\top x=c_1,\ Q_2^\top x=c_2\}.
\]
In particular, there is a left map $R_2$ with $R_2M_2=Q_2^\top$ and
$R_2v_2=c_2$ such that
\[
\mathcal F
=
\{x:M_1x=v_1,\ R_2M_2x=R_2v_2\}.
\]
Thus an exact common row direction need be retained in only one affine channel.
\end{lemma}

\begin{proof}
Since $x_\star$ is feasible,
\[
M_1x=v_1
\iff M_1(x-x_\star)=0
\iff x-x_\star\in\Ker(M_1)=U^\perp.
\]
Because $U=I\oplus U_0$, this is equivalent to
$Q_I^\top x=c_I$ and $Q_1^\top x=c_1$.
Likewise, $M_2x=v_2$ is equivalent to
$Q_I^\top x=c_I$ and $Q_2^\top x=c_2$.
The same common right-hand side appears in both descriptions precisely because the joint
system is consistent.

Moreover, $\Range(Q_2)\subseteq\Range(M_2^\top)$, so
$Q_2^\top=R_2M_2$ for some $R_2$. Since $v_2=M_2x_\star$,
$R_2v_2=Q_2^\top x_\star=c_2$. The first channel still imposes the common
equations, hence replacing the second by its $V_0$ component leaves
$\mathcal F$ unchanged.
\end{proof}

The lemma is an algebraic equivalence statement; it does not assert a
decentralized or zero-cost procedure for constructing $R_2$. A merely
near-common direction cannot be removed by this argument: if its principal
angle is $\eta>0$, it belongs to neither exact intersection and its contribution
to the cross factor diverges as $\eta\downarrow0$.

\section{Friedrichs Geometry and Cross-Angle Conditioning}
\label{app:friedrichs-details}

\subsection{Representation-equivalent two-agent example}\label{app:two-agent-example}

Take $f_i(x_i,z)=\tfrac12(x_i^2+z^2)$, no local equations, and
$x_1+x_2+2z=1$. With copies $z_1,z_2$, every $t\in\R$ gives the equivalent constraints
\[
x_1+x_2+(1+t)z_1+(1-t)z_2=1,\qquad z_1-z_2=0.
\]
The term $t(z_1-z_2)$ vanishes at consensus, so the original feasible set and objective
are unchanged. Using one normalized consensus row, the two channels are
\[
M_1(t)=\tfrac1{\sqrt2}[1\ \ 1\ \ 1+t\ \ 1-t],\qquad
M_2=\tfrac1{\sqrt2}[0\ \ 0\ \ 1\ \ {-1}],
\]
with right-hand sides $1/\sqrt2$ and $0$. Their row spaces are the spans of these
rows: they describe constraint normals, not feasible points. Normalizing each row
to unit length gives
\[
c_F(t)=\frac{|t|}{\sqrt{2+t^2}},\qquad
\kappa([\widehat M_1(t);\widehat M_2])=
\frac{1+c_F(t)}{1-c_F(t)}.
\]
At $t=0$ the channels are orthogonal; increasing $|t|$ makes their normalized rows
nearly parallel without changing the optimization problem. Separate normalization cannot rotate these row spaces, whereas an equivalent redistribution that mixes the two channel classes can. Such cross-channel redistribution changes the resource-separated oracle representation and is not counted as channelwise preconditioning. The Friedrichs cosine below captures this distinction.

\subsection{Full weighted conditioning statement}
\begin{theorem}[Conditioning and optimal scalar balancing of two cross-coupled affine channels]
\label{thm:angle}
For the actual nonzero maps $M_1,M_2$, let
\[
\mu_j=\sigma_{\min}^+(M_j)^2,\qquad L_j=\sigma_{\max}(M_j)^2,
\qquad j\in\{1,2\},
\]
put $I=U\cap V$ and $c=\cF(U,V)<1$, and define $M_\eta=[M_1;\eta M_2]$ for $\eta>0$.
Its nonzero singular values satisfy
\begin{align}
\sigma_{\max}(M_\eta)^2&\le\Lambda_\eta:=
\begin{cases}
\dfrac{L_1+\eta^2L_2+\sqrt{(L_1-\eta^2L_2)^2+4\eta^2L_1L_2c^2}}{2},&I=\{0\},\\[2pt]
L_1+\eta^2L_2,&I\ne\{0\},
\end{cases}\label{eq:maxbound}\\
\sigma_{\min}^+(M_\eta)^2&\ge
\delta_\eta:=\frac{\mu_1+\eta^2\mu_2-
\sqrt{(\mu_1-\eta^2\mu_2)^2+4\eta^2\mu_1\mu_2c^2}}{2}.
\label{eq:minbound}
\end{align}
Hence $\kappa(M_\eta)\le\Lambda_\eta/\delta_\eta$.
If $I=\{0\}$, this envelope, which uses only $(\mu_j,L_j,c)$, is minimized over scalar balances by
\begin{equation}
\eta_\star^2=\sqrt{\frac{\mu_1L_1}{\mu_2L_2}}.
\label{eq:optimal-scalar-balance}
\end{equation}
In that irredundant case, if $c=0$, then
\begin{equation}
\kappa(M_{\eta_\star})\le\max\{\kappa(M_1),\kappa(M_2)\}.
\label{eq:orthogonal-scalar-balance}
\end{equation}
Still assuming $I=\{0\}$, if also $M_1^\top M_1=P_U$ and $M_2^\top M_2=P_V$, then
\begin{equation}
\kappa\!\left(\begin{bmatrix}M_1\\M_2\end{bmatrix}\right)
=\frac{1+c}{1-c}.
\end{equation}
\end{theorem}
\begin{corollary}[Removing exact overlap]
\label{cor:angle-remove-overlap}
Let $I=U\cap V$, $U_0=U\cap I^\perp$, and $V_0=V\cap I^\perp$. Exact common directions contribute the separate benign eigenvalue $a+b$ to $aP_U+bP_V$ and may be consistently deduplicated by Lemma~\ref{lem:affine-dedup}. The cross-channel conditioning is therefore governed by the reduced pair $(U_0,V_0)$, for which $U_0\cap V_0=\{0\}$. The scalar-balancing and exact-whitening conclusions of Theorem~\ref{thm:angle} apply to that reduced stack; the unreduced stack retains the common-direction eigenvalue.
\end{corollary}
\subsection{Proof of Theorem~\ref{thm:angle}}
\paragraph{Common directions in the unreduced stack.}
Lemma~\ref{lem:weighted-projectors} shows that an exact common direction contributes weight $a+b$ to $aP_U+bP_V$. Hence the unreduced stack cannot use the irredundant $1+c_F$ upper edge: common directions remain genuine eigen-directions until the affine system itself is deduplicated by Lemma~\ref{lem:affine-dedup}. We therefore bound the actual stack on $S=U+V$ and invoke irredundancy only for the scalar-balancing and exact-whitening identities below.

Work on $S=U+V$, the range of $M_\eta^\top$, without discarding $I=U\cap V$.
Let $P_U,P_V$ be the orthogonal projectors onto $U,V$. Since $\Ker(M_1)=U^\perp$,
$\Ker(M_2)=V^\perp$, and the nonzero eigenvalues of $M_j^\top M_j$ lie in $[\mu_j,L_j]$,
\begin{equation}
\mu_1P_U\preceq M_1^\top M_1\preceq L_1P_U,
\qquad
\mu_2P_V\preceq M_2^\top M_2\preceq L_2P_V.
\label{eq:projectororder-v3}
\end{equation}
Consequently
\[
\mu_1P_U+\eta^2\mu_2P_V\preceq M_\eta^\top M_\eta
\preceq L_1P_U+\eta^2L_2P_V.
\]
If $I=\{0\}$, Lemma~\ref{lem:weighted-projectors} bounds the largest eigenvalue
of the right-hand side by the first branch of \eqref{eq:maxbound}. Exclusive
one-dimensional components have weights $L_1$ or $\eta^2L_2$, no larger than
that envelope. If $I\ne\{0\}$, the common component has eigenvalue
$L_1+\eta^2L_2$ and attains the projector sum's upper bound, proving the second branch.
For the lower bound, the principal-plane values are at least $\delta_\eta$;
the exclusive weights $\mu_1,\eta^2\mu_2$ and the common weight
$\mu_1+\eta^2\mu_2$ are also at least $\delta_\eta$.
Both comparison operators are positive definite on $S$, so the Rayleigh bounds
prove \eqref{eq:minbound}, including cases without a nontrivial principal plane.

For the scalar-balance and exact-whitening claims, assume from now on that $I=\{0\}$.
It remains to optimize the ratio $\Lambda_\eta/\delta_\eta$. For $a,b>0$ write
\[
\lambda_+(a,b;c)
:=\frac{a+b+\sqrt{(a-b)^2+4abc^2}}2.
\]
Lemma~\ref{lem:weighted-projectors} also gives
$\lambda_+(a,b;c)\lambda_-(a,b;c)=ab(1-c^2)$. Set
\[
G_c(t):=\cosh t+\sqrt{\sinh^2t+c^2}.
\]
Then
\[
\lambda_+(a,b;c)=\sqrt{ab}\,
G_c\!\left(\frac12\log\frac ba\right),
\qquad
\lambda_-(a,b;c)=\frac{\sqrt{ab}(1-c^2)}
{G_c(\frac12\log(b/a))}.
\]
Consequently, with $\vartheta:=\eta^2$,
\begin{equation}
\frac{\Lambda_\eta}{\delta_\eta}
=\frac{\sqrt{\kappa(M_1)\kappa(M_2)}}{1-c^2}
G_c(t_L)G_c(t_\mu),
\quad
t_L:=\frac12\log\frac{\vartheta L_2}{L_1},\quad
t_\mu:=\frac12\log\frac{\vartheta\mu_2}{\mu_1}.
\label{eq:scalar-balance-hyperbolic}
\end{equation}
The difference $t_L-t_\mu$ is independent of $\vartheta$. Moreover $\log G_c$ is even and convex: for $c>0$,
\[
\frac{d}{dt}\log G_c(t)=\frac{\sinh t}{\sqrt{\sinh^2t+c^2}},
\qquad
\frac{d^2}{dt^2}\log G_c(t)=
\frac{c^2\cosh t}{(\sinh^2t+c^2)^{3/2}}\ge0,
\]
and for $c=0$ the limiting function is $\log G_0(t)=|t|$. Hence the product in~\eqref{eq:scalar-balance-hyperbolic} is minimized when $t_L+t_\mu=0$, which is equivalent to
$\vartheta^2=\mu_1L_1/(\mu_2L_2)$ and proves~\eqref{eq:optimal-scalar-balance}; write $\vartheta_\star:=\eta_\star^2$ for this minimizer.

If $c=0$, the projector eigenvalues are simply the two weights. Suppose without loss of generality $\kappa(M_1)\ge\kappa(M_2)$. At $\vartheta=\vartheta_\star$, one has $\vartheta L_2\le L_1$ and $\vartheta\mu_2\ge\mu_1$, so
$\Lambda_{\eta_\star}=L_1$ and $\delta_{\eta_\star}=\mu_1$, proving~\eqref{eq:orthogonal-scalar-balance}. Finally, in the exactly whitened case $M^\top M=P_U+P_V$; Lemma~\ref{lem:weighted-projectors} gives extremal eigenvalues $1-c$ and $1+c$, which yields~\eqref{eq:exactcondition}. \qed

\subsection{Canonical core-model heterogeneity characterization}
\begin{proposition}[Exact heterogeneity characterization of the canonical cross angle]
\label{prop:heterogeneity}
Consider the canonical core model $C_i=D_i=0$ and assume
$\Range(K_1^\top)\cap\Range(K_2^\top)=\{0\}$. Let
\begin{equation}
H_0:=\sum_{i=1}^n(A_iA_i^\top+B_iB_i^\top),\qquad
H_{\rm het}:=\sum_{i=1}^n(B_i-\bar B)(B_i-\bar B)^\top,
\qquad \bar B:=\frac1n\sum_{i=1}^nB_i.
\label{eq:heterogeneity-grams}
\end{equation}
Then the canonical Friedrichs cosine is exactly
\begin{equation}
c_F^2=
\sup_{h:\,h^\top H_0h>0}
\frac{h^\top H_{\rm het}h}{h^\top H_0h}
=
\lambda_{\max}\!\left(H_0^{\dagger/2}H_{\rm het}H_0^{\dagger/2}\right),
\label{eq:heterogeneity}
\end{equation}
where the eigenvalue is taken on $\Range(H_0)$. Moreover
$0\preceq H_{\rm het}\preceq H_0$, so the spectrum in~\eqref{eq:heterogeneity} lies in $[0,1)$ under irredundancy. Thus $\chix$ is computable from an $m\times m$ generalized eigenvalue problem; in particular, $c_F\to1$ iff the largest heterogeneity Rayleigh quotient tends to one.
\end{proposition}
\subsection{Proof of Proposition~\ref{prop:heterogeneity}}
For $C=D=0$, the canonical channels in~\eqref{eq:channels} are
$K_1=[E_mA\;E_mB]$ and $K_2=[0\;\Wp_q]$. Hence
\[
V:=\Range(K_2^\top)=\{0\}\times\Range(\Wp_q),
\]
and, because $K_1$ has only the coupled rows, every vector in
$U:=\Range(K_1^\top)$ is of the form
\[
\frac1{\sqrt n}g(h),\qquad
g(h):=\col\bigl(A_i^\top h,B_i^\top h\bigr)_{i=1}^n,
\qquad h\in\R^m.
\]
Thus the parameterization by $h$ is exhaustive, not merely a restricted family of row-space directions. Under $U\cap V=\{0\}$,
\[
c_F=\sup_{u\in U\setminus\{0\}}\frac{\|P_Vu\|}{\|u\|}.
\]
The projector onto $\Range(\Wp_q)$ subtracts the block mean, so
\[
P_Vg(h)=\col\bigl(0,(B_i-\bar B)^\top h\bigr)_{i=1}^n.
\]
Consequently
\[
\frac{\|P_Vg(h)\|^2}{\|g(h)\|^2}
=
\frac{\sum_i\|(B_i-\bar B)^\top h\|^2}
{\sum_i(\|A_i^\top h\|^2+\|B_i^\top h\|^2)}
=
\frac{h^\top H_{\rm het}h}{h^\top H_0h}.
\]
Taking the supremum proves the first equality in~\eqref{eq:heterogeneity}. Also
\[
H_{\rm het}=\sum_iB_iB_i^\top-n\bar B\bar B^\top
\preceq\sum_iB_iB_i^\top\preceq H_0,
\]
which implies $\Range(H_{\rm het})\subseteq\Range(H_0)$ and converts the Rayleigh quotient into the pseudoinverse generalized-eigenvalue formula in~\eqref{eq:heterogeneity}. Irredundancy excludes eigenvalue one. \qed

\section{Sharpness and Optimal Separate Preconditioning}
\begin{proposition}[Sharp two-dimensional obstruction]\label{prop:2d}
For every $\theta\in(0,\pi/2]$ there are rank-one $M_1,M_2:\R^2\to\R$ with $\kappa(M_1)=\kappa(M_2)=1$ but $\kappa([M_1;M_2])=\cot^2(\theta/2)$.
\end{proposition}
\subsection{Proof of Proposition~\ref{prop:2d}}
Take
\[
M_1=\begin{bmatrix}1&0\end{bmatrix},
\qquad
M_2=\begin{bmatrix}\cos\theta&\sin\theta\end{bmatrix}.
\]
Both rows have unit norm, so each channel has condition number one and their row-space angle is $\theta$. Lemma~\ref{lem:weighted-projectors} with $a=b=1$ gives joint Gram eigenvalues $1\pm\cos\theta$, hence
\[
\kappa([M_1;M_2])
=\frac{1+\cos\theta}{1-\cos\theta}
=\cot^2(\theta/2).
\]
In particular, $\kappa([M_1;M_2])\sim4/\theta^2$ as $\theta\downarrow0$. \qed

\begin{theorem}[Optimality of the cross-angle barrier under arbitrary channelwise preconditioning]
\label{thm:blockwise-opt}
Let $U=\Range(M_1^\top)$ and $V=\Range(M_2^\top)$ be irredundant, $U\cap V=\{0\}$, and let $S_j$ range over all linear maps that are injective on $\Range(M_j)$. Then
\begin{equation}
\inf_{S_1,S_2}
\kappa\!\left(\begin{bmatrix}S_1M_1\\S_2M_2\end{bmatrix}\right)
=\frac{1+c_F(U,V)}{1-c_F(U,V)}=\chix^2.
\label{eq:blockwise-opt}
\end{equation}
The infimum is attained by exact separate whitening,

$S_j=(M_jM_j^\top)^{\dagger/2}$ on $\Range(M_j)$, where $\dagger$ denotes the Moore--Penrose pseudoinverse.
Thus no invertible linear processing performed separately inside the two affine channels can improve the intrinsic cross-angle condition number.
\end{theorem}

\subsection{Proof of Theorem~\ref{thm:blockwise-opt}}
Write $\widetilde M_j=S_jM_j$ and
\[
\mathsf G_1:=\widetilde M_1^\top\widetilde M_1,
\qquad
\mathsf G_2:=\widetilde M_2^\top\widetilde M_2,
\qquad
\mathsf G:=\mathsf G_1+\mathsf G_2.
\]
Injectivity of $S_j$ on $\Range(M_j)$ preserves the row spaces, so $\mathsf G_1$ is positive definite on $U$ and vanishes on $U^\perp$, while $\mathsf G_2$ is positive definite on $V$ and vanishes on $V^\perp$. Since $U\cap V=\{0\}$, $\mathsf G$ is positive definite on $\mathcal S:=U+V$.

Let $Q_U,Q_V$ be orthonormal basis matrices for $U,V$ and set $Q_{UV}:=[Q_U\ Q_V]$.  The map $Q_{UV}:\R^{\dim U}\oplus\R^{\dim V}\to\mathcal S$ is an isomorphism. For some positive definite matrices $H_U,H_V$,
\[
\mathsf G=Q_{UV}\begin{bmatrix}H_U&0\\0&H_V\end{bmatrix}Q_{UV}^\top
\quad\text{on }\mathcal S.
\]
Hence
\[
\mathsf G^{-1}=Q_{UV}^{-\top}\begin{bmatrix}H_U^{-1}&0\\0&H_V^{-1}\end{bmatrix}Q_{UV}^{-1},
\]
and therefore every $u\in U$ and $v\in V$ are orthogonal in the $\mathsf G^{-1}$-metric:
\begin{equation}
u^\top \mathsf G^{-1}v=0.
\label{eq:ginv-orthogonality}
\end{equation}

Choose a Friedrichs principal pair of Euclidean unit vectors $u\in U$, $v\in V$, with $u^\top v=c:=c_F(U,V)$ (flip the sign of $v$ if needed). Let $0<\underline g\le\overline g$ be the smallest and largest eigenvalues of $\mathsf G^{-1}$ on $\mathcal S$. By \eqref{eq:ginv-orthogonality},
\[
(u+v)^\top \mathsf G^{-1}(u+v)=(u-v)^\top \mathsf G^{-1}(u-v).
\]
Using the Rayleigh bounds on the two sides gives
\[
2\underline g(1+c)
\le (u+v)^\top \mathsf G^{-1}(u+v)
=(u-v)^\top \mathsf G^{-1}(u-v)
\le2\overline g(1-c).
\]
Thus
\[
\kappa_+(\mathsf G)=\kappa_+(\mathsf G^{-1}|_{\mathcal S})=\frac{\overline g}{\underline g}\ge\frac{1+c}{1-c}=\chix^2.
\]
Since $\mathsf G$ is the Gram matrix of the stacked preconditioned operator, this proves the lower bound in \eqref{eq:blockwise-opt}. Exact separate whitening gives $\mathsf G_1=P_U$ and $\mathsf G_2=P_V$; the exact-whitening part of Theorem~\ref{thm:angle} then gives equality. \qed

\section{Generic Separate Normalization}\label{app:normalization}
\begin{lemma}[Row-space invariance under left spectral preconditioning]\label{lem:rowspace-invariance}
If $S$ maps $\Range(M)$ bijectively onto itself, then $\Range((SM)^\top)=\Range(M^\top)$.
\end{lemma}
\begin{lemma}[Residual condition after separate normalization]
\label{lem:joint-cheb}
Let $U=\Range(\widehat K_1^\top)$ and
$V=\Range(\widehat K_2^\top)$, with $U+V\ne\{0\}$.
Put $I=U\cap V$, and define $g_I=1$ if $I=\{0\}$ and
$g_I=2$ otherwise. Suppose $0<a\le b$ and
\[
aP_U\preceq\widehat K_1^\top\widehat K_1\preceq bP_U,
\qquad
aP_V\preceq\widehat K_2^\top\widehat K_2\preceq bP_V.
\]
Put $c=c_F(U,V)<1$, with the convention in
\eqref{eq:friedrichs}. Then
\begin{equation}
\frac ab\frac{1+c}{1-c}
\le\kappa([\widehat K_1;\widehat K_2])
\le g_I\frac ba\frac{1+c}{1-c}.
\label{eq:residual-condition}
\end{equation}
In all cases the upper bound $2b/[a(1-c)]$ also holds.
\end{lemma}
\subsection{Proof of Lemma~\ref{lem:rowspace-invariance}}
Because $S$ is injective on $\Range(M)$,
\[
SMx=0 \iff Mx=0,
\]
so $\Ker(SM)=\Ker(M)$. In finite-dimensional Euclidean spaces, $\Range(M^\top)=\Ker(M)^\perp$. Therefore
\[
\Range((SM)^\top)=\Ker(SM)^\perp=\Ker(M)^\perp=\Range(M^\top).
\]
The same argument applies to any polynomial left preconditioner whose scalar polynomial is nonzero on the positive spectrum of $MM^\top$. Hence the Friedrichs angle is invariant. \qed

\subsection{Proof of Lemma~\ref{lem:joint-cheb}}
Set $P=P_U+P_V$ and
$\mathsf G=\widehat K_1^\top\widehat K_1+
\widehat K_2^\top\widehat K_2$.
Then $aP\preceq \mathsf G\preceq bP$ and
$\Ker \mathsf G=(U+V)^\perp$.
The principal-angle decomposition gives eigenvalue $2$ on
$I$, eigenvalue $1$ on exclusive directions, and
$1\pm\cos\theta$ on each nontrivial principal plane.
Consequently $\lambda_{\max}(P)\le2$ and
$\lambda_{\min}^+(P)\ge1-c$.
If $I=\{0\}$, the upper bound sharpens to
$\lambda_{\max}(P)\le1+c$.
These facts prove both asserted upper bounds.
If $c>0$, a principal plane attaining $c$ supplies unit vectors
with Rayleigh quotients $1+c$ and $1-c$; comparison with $\mathsf G$
gives the lower bound in \eqref{eq:residual-condition}.
If $c=0$, that bound follows from $\kappa\ge1\ge a/b$.
No common-direction eigenvalue is discarded from the full stack. \qed

\subsection{Constant-accuracy inexact channel solves}
\begin{corollary}[Inexact whitening preserves cross-optimality]
\label{cor:inexact-whitening}
Let $\widetilde K_1,\widetilde K_2$ have row spaces $U,V$,
with $U+V\ne\{0\}$, and let $I,g_I$ be as in
Lemma~\ref{lem:joint-cheb}. Define
$\chi_{U,V}:=\sqrt{(1+c_F(U,V))/(1-c_F(U,V))}$.
Suppose, for $0\le\delta_{\rm wh}<1$,
\[
(1-\delta_{\rm wh})P_U
\preceq\widetilde K_1^\top\widetilde K_1
\preceq(1+\delta_{\rm wh})P_U,
\qquad
(1-\delta_{\rm wh})P_V
\preceq\widetilde K_2^\top\widetilde K_2
\preceq(1+\delta_{\rm wh})P_V.
\]
Then
\begin{equation}
\chi_{U,V}^2
\le\kappa([\widetilde K_1;\widetilde K_2])
\le g_I\frac{1+\delta_{\rm wh}}{1-\delta_{\rm wh}}\chi_{U,V}^2.
\label{eq:inexact-whitening}
\end{equation}
Thus fixed whitening accuracy preserves the angle dependence.
An outer certified interval of this order gives joint
Chebyshev degree $O(\chi_{U,V})$.
\end{corollary}
\begin{proof}
Apply Lemma~\ref{lem:joint-cheb} with
$a=1-\delta_{\rm wh}$ and $b=1+\delta_{\rm wh}$ for the upper bound.
For $I=\{0\}$ and both row spaces nonzero,
Theorem~\ref{thm:blockwise-opt} gives the lower bound.
For $I\ne\{0\}$ and $c_F(U,V)>0$, put
$U_0=U\cap I^\perp$, $V_0=V\cap I^\perp$ and $S_0=U_0+V_0$.
The compressed maps $\widetilde K_jP_{S_0}$ have irredundant
row spaces $U_0,V_0$ and the same Friedrichs cosine.
Theorem~\ref{thm:blockwise-opt} bounds their stacked condition
below by $\chi_{U,V}^2$.
The extremal Rayleigh quotients of the full positive Gram on
$U+V$ bound those of its compression to $S_0$ from outside,
so its condition number is at least as large.
For $c_F(U,V)=0$, including a zero row space, the lower bound
is simply $\kappa\ge1$.
Since $g_I\le2$, taking square roots of the certified condition
bound proves the degree claim. Fixed inner accuracy suffices
independently of the target outer accuracy.
\end{proof}

\subsection{Constructive polynomial channel normalization}

\begin{proposition}[Polynomial fixed-accuracy channel normalization]
\label{prop:poly-normalization}
Let $M\ne0$, $Q=MM^\top$, $U=\Range(M^\top)$, and suppose
$\operatorname{spec}(Q)\subset\{0\}\cup[a,b]$, where $0<a\le b$.
Put $\kappa_{\rm poly}=b/a$.
For $0<\delta_{\rm wh}<1$ there is a real polynomial $p_N$ with
\[
N+1=O\!\left(\sqrt{\kappa_{\rm poly}}\log\frac6{\delta_{\rm wh}}\right)
\]
such that $\widehat M=p_N(Q)M$ satisfies
\begin{equation}
(1-\delta_{\rm wh})P_U
\preceq\widehat M^\top\widehat M
\preceq(1+\delta_{\rm wh})P_U.
\label{eq:poly-normalization-certificate}
\end{equation}
Moreover $\Ker\widehat M=\Ker M$ and
$\Range(\widehat M^\top)=U$.
For a feasible equation $Mw=v$, its equivalent normalized equation
is $\widehat Mw=p_N(Q)v$.
One forward or adjoint normalized action uses at most $2N+1$
actions of $M,M^\top$.
\end{proposition}
\begin{proof}
Set $\tau=\delta_{\rm wh}/3$.
If $a=b$, take $N=0$ and $p_0=a^{-1/2}$.
Otherwise put $c=(a+b)/2$, $s=(b-a)/2$.
Let $\mathsf T_k$ be the degree-$k$ Chebyshev polynomial
of the first kind and set
\[
N+1=\left\lceil\frac{\sqrt{\kappa_{\rm poly}}}2
                  \log\frac6{\delta_{\rm wh}}\right\rceil,
\qquad R_{N+1}(x)=\mathsf T_{N+1}((c-x)/s).
\]
Let $p_N$ interpolate $x^{-1/2}$ at the $N+1$ roots of $R_{N+1}$
in $(a,b)$. To bound its relative error, use
\[
x^{-1/2}=\frac1\pi\int_0^\infty\frac{t^{-1/2}}{x+t}\,dt.
\]
For $t\ge0$, the polynomial
\[
q_t(x)=\frac{1-R_{N+1}(x)/R_{N+1}(-t)}{x+t}
\]
has degree at most $N$ and interpolates $(x+t)^{-1}$ at
the same roots. Integration of the finite interpolation formula
therefore gives
\[
\left|x^{-1/2}-p_N(x)\right|
\le \frac{x^{-1/2}}{\mathsf T_{N+1}(c/s)}
\quad (a\le x\le b),
\]
because $|R_{N+1}(x)|\le1$ and
$R_{N+1}(-t)=\mathsf T_{N+1}((c+t)/s)\ge\mathsf T_{N+1}(c/s)>0$.
Writing $\rho=(\sqrt{\kappa_{\rm poly}}-1)/(\sqrt{\kappa_{\rm poly}}+1)$,
\[
\frac1{\mathsf T_{N+1}(c/s)}
=\frac{2\rho^{N+1}}{1+\rho^{2(N+1)}}
\le2\exp(-2(N+1)/\sqrt{\kappa_{\rm poly}})\le\tau.
\]
Thus $|\sqrt{x}\,p_N(x)-1|\le\tau$.
For a thin SVD $M=U_M\Sigma V_M^\top$ with positive singular
values $\sigma_i$,
\[
\widehat M^\top\widehat M
=V_M\diag\!\left(\sigma_i^2p_N(\sigma_i^2)^2\right)V_M^\top.
\]
Its positive eigenvalues lie in $[(1-\tau)^2,(1+\tau)^2]$,
which is contained in $[1-\delta_{\rm wh},1+\delta_{\rm wh}]$.
This proves \eqref{eq:poly-normalization-certificate} and both
kernel and row-space assertions.

Exact whitening is $(Q^\dagger)^{1/2}M$, with the
Moore--Penrose inverse acting as zero on $\Ker Q$.
The polynomial need not approximate that operator at zero:
$M$ maps into $\Range(Q)$, so $p_N(0)$ is immaterial.
Feasibility implies $v\in\Range(M)$, and $p_N(Q)$ is
injective on that range, proving equivalence of the affine equations.
A forward action uses one $M$ followed by $N$ products
$Qz=M(M^\top z)$; the adjoint reverses the order.
These use $(N+1,N)$ forward/adjoint actions, or $(N,N+1)$.
There is no additional $\log\kappa_{\rm poly}$ factor.
For fixed $\delta_{\rm wh}$, $N+1=O(\sqrt{\kappa_{\rm poly}})$.
\end{proof}

\section{Certified CC-MAC Local Realization}\label{app:ccmac-normalization}
This section specializes the generic normalization results above to the executable CC-MAC lift and records the spectral certificates used by the primitive-cost ledger.

\paragraph{Structural parameters.}
In the original variable ordering put $H=[A\ B]$ and
$T=[C\ 0]$. A fixed permutation groups these into the local
blocks $H_i=[A_i\ B_i]$ and $\bar C_i=[C_i\ 0]$.
Define
\begin{equation}
\begin{aligned}
h&:=\|H\|^2=\max_i\lambda_{\max}(A_iA_i^\top+B_iB_i^\top),\\
S_{HC}&:=\frac1n\sum_i
 (A_iP_{\Ker C_i}A_i^\top+B_iB_i^\top),\\
d_D&:=\max_i\lambda_{\max}(D_i^\top D_i),\qquad
S_D:=\frac1n\sum_iD_i^\top D_i.
\end{aligned}
\label{eq:normalization-structural}
\end{equation}
The subscript on $d_D$ distinguishes this scale from the primal
dimension $d$. Set $\mu_{HC}=\lambda_{\min}^+(S_{HC})$ if
$S_{HC}\ne0$, and $\mu_{HC}=0$ otherwise.
Set $\mu_D=\lambda_{\min}^+(S_D)$ if $D\ne0$, and $\mu_D=0$
otherwise. The ideal condition parameters are
\[
k_1:=\widetilde\kappa_{HC}
=\begin{cases}h/\mu_{HC},&S_{HC}\ne0,\\1,&S_{HC}=0,\end{cases}
\qquad
k_2:=\widehat\kappa_{D^\top}
=\begin{cases}d_D/\mu_D,&D\ne0,\\1,&D=0.\end{cases}
\]
For $C\ne0$ define
\[
\kappa_C:=
\frac{\max_i\sigma_{\max}(C_i)^2}
     {\min_{i:C_i\ne0}\sigma_{\min}^+(C_i)^2},
\]
and use $\kappa_C=1$ when $C=0$.
The graph convention remains
$\kappa_W=\lambda_{\max}(W)/\lambda_{\min}^+(W)$.
For the one-agent case use the bookkeeping value $\kappa_W=1$
and perform no communication.

\begin{assumption}[Supplied numerical spectral certificates]
\label{ass:spectral-certificates}
For $H\ne0$ the implementation receives $\bar h\ge h$ and
$0<\underline\mu_{HC}\le\bar h$ satisfying
\[
\operatorname{spec}(S_{HC})
\subset\{0\}\cup[\underline\mu_{HC},\bar h].
\]
For $D\ne0$ it receives $\bar d_D\ge d_D$ and
$0<\underline\mu_D\le\bar d_D$ satisfying
$\operatorname{spec}(S_D)\subset\{0\}\cup
[\underline\mu_D,\bar d_D]$.
For $C\ne0$ it receives $0<\ell_C\le u_C$ with
$\operatorname{spec}(CC^\top)\subset\{0\}\cup[\ell_C,u_C]$.
For $n\ge2$ it receives $0<\ell_W\le u_W$ with
$\operatorname{spec}(W)\subset\{0\}\cup[\ell_W,u_W]$.
The outer method receives the objective bounds and either
$\Xi\ge\chix$ or a valid positive-spectrum interval for the
final normalized stack.
\end{assumption}
Define the corresponding certified ratios by
\begin{equation}
\bar k_1=\bar h/\underline\mu_{HC},\qquad
\bar k_2=\bar d_D/\underline\mu_D,\qquad
\bar\kappa_C=u_C/\ell_C,\qquad
\bar\kappa_W=u_W/\ell_W.
\label{eq:normalization-cert-ratios}
\end{equation}
The respective bookkeeping values are $1$ for $H=0$, $D=0$,
$C=0$, and $n=1$, so no absent-block ratio is evaluated.
If $S_{HC}=0$ is known and $H\ne0$, one may take
$\underline\mu_{HC}=\bar h$.
The excluded-gap certificates allow rank deficiency and do not
require the rank as an input.

These are supplied problem/setup parameters, as in
spectrally tuned Chebyshev preconditioning
\citep{salim2022optimal,yarmoshik2025coupled}.
Exact eigenvalues or eigendecompositions are not required.
Conservative bounds suffice; their ratios determine the costs.
Constant-factor-tight certificates recover the ideal orders.
Obtaining certificates from arbitrary black-box access is a
separate setup operation, outside the solve counts below.
Upper bounds from explicitly available local matrix entries may
be used, with any acquisition and aggregation costs charged
separately; access to such entries is not implicit in a matvec oracle.

\paragraph{Preliminary graph and local normalization.}
For $n\ge2$ and $\ell_W<u_W$, put
$c_W=(u_W+\ell_W)/2$, $s_W=(u_W-\ell_W)/2$ and
\[
\nu=\left\lceil\frac{\sqrt{\bar\kappa_W}}2\log4\right\rceil,
\qquad
P_W(x)=1-\frac{\mathsf T_\nu((c_W-x)/s_W)}
                     {\mathsf T_\nu(c_W/s_W)}.
\label{eq:graph-cheb-filter}
\]
For $\ell_W=u_W$ take $P_W(x)=x/\ell_W$.
Then $\widetilde W=P_W(W)$ has the same kernel as $W$ and
positive spectrum in $[1/2,3/2]$.
Thus its positive Gram spectrum lies in $[1/4,9/4]$,
while one action costs $O(\sqrt{\bar\kappa_W})$ raw neighbor
rounds. The filter is applied to $W$, not $W^2$.

For $C\ne0$, apply Proposition~\ref{prop:poly-normalization} with
accuracy $1/2$, and write
\[
C_\star=p_C(CC^\top)C,\qquad c_\star=p_C(CC^\top)\bm c.
\]
Then
$\tfrac12P_{\Range(C^\top)}
\preceq C_\star^\top C_\star
\preceq\tfrac32P_{\Range(C^\top)}$ and
$\Ker C_\star=\Ker C$.
This block-diagonal operation costs
$O(\sqrt{\bar\kappa_C})$ raw $C/C^\top$ products per action
and no communication. If $C=0$, omit this layer and its rows.

\begin{lemma}[Certified spectra of the balanced executable channels]
\label{lem:structured-normalization}
Suppose $H,D\ne0$ and $n\ge2$.
With $\widetilde W_s=\widetilde W\otimes I_s$, choose
\[
\alpha^2=12\bar h,\qquad
\beta^2=6\bar h,\qquad
\gamma^2=8\bar d_D,
\]
and form the final executable channels
\begin{equation}
\mathcal L_1=
\begin{bmatrix}A&B&\alpha\widetilde W_m\\
                \beta C_\star&0&0\end{bmatrix},
\qquad
\mathcal L_2=
\begin{bmatrix}0&D&0\\0&\gamma\widetilde W_q&0\end{bmatrix}.
\label{eq:balanced-normalization-channels}
\end{equation}
Their right-hand sides are
$\col(\bm b,\beta c_\star)$ and $\col(\bm e,0)$, respectively.
Their positive Gram spectra satisfy
\begin{equation}
\begin{aligned}
\operatorname{spec}^+(\mathcal L_1\mathcal L_1^\top)
 &\subset[\underline\mu_{HC}/4,37\bar h],\\
\operatorname{spec}^+(\mathcal L_2\mathcal L_2^\top)
 &\subset[\underline\mu_D/2,19\bar d_D].
\end{aligned}
\label{eq:structured-channel-intervals}
\end{equation}
If $S_{HC}=0$, the first lower endpoint improves to $\bar h$.
The final row spaces obey
$\chi_{\rm loc}\le\sqrt3\,\chix$, and the penalty objective
in Proposition~\ref{prop:decentralized-bridge}, with
$r=\mu_f/(2\bar h)$, has $\kappa_{\Phi_{\rm loc}}=O(\kappa_f)$.
\end{lemma}
\begin{proof}
This is the coupled/local spectral reduction
$A_i\mapsto H_i$, $C_i\mapsto\bar C_i$
of \citet[Lemma~4.4 and Appendix~H]{yarmoshik2026mixed};
the following estimates verify the compatible bridge scaling.

Let $\Pi=(\one\one^\top/n)\otimes I_m$,
$T_\star=[C_\star\ 0]$ and
$\mathcal J=\col(\Pi H,\beta T_\star)$.
The nonzero spectrum of
$\Pi HP_{\Ker T}H^\top\Pi$ equals that of $S_{HC}$.
For $t\in\Range(\mathcal J^\top)$, decompose
$t=y_0+g_\perp$ orthogonally with
$y_0\in\Range(T^\top)$ and
$g_\perp\in\Range(P_{\Ker T}H^\top\Pi)$.
Using $\|a+b\|^2\ge\frac12\|b\|^2-\|a\|^2$ gives
\[
\|\mathcal Jt\|^2
\ge(\beta^2/2-h)\|y_0\|^2+
   (\underline\mu_{HC}/2)\|g_\perp\|^2
\ge(\underline\mu_{HC}/2)\|t\|^2.
\]
If $S_{HC}=0$, then $g_\perp=0$ and the last bound improves to
$2\bar h\|t\|^2$.

For $z=(a,c)\in\Range(\mathcal L_1)$ write
$a_0=\Pi a$, $a_\perp=(\mathrm{Id}-\Pi)a$.
Representing $z=\mathcal L_1(u,y)$ shows
$(a_0,c)=\mathcal Ju\in\Range(\mathcal J)$.
Hence
\[
\|\mathcal L_1^\top z\|^2
\ge\tfrac12\|\mathcal J^\top(a_0,c)\|^2+
        (\alpha^2/4-h)\|a_\perp\|^2
\ge(\underline\mu_{HC}/4)\|z\|^2.
\]
When $S_{HC}=0$, this gives $\bar h\|z\|^2$.
The upper bound is
$\|\mathcal L_1\|^2\le
h+(9/4)\alpha^2+(3/2)\beta^2\le37\bar h$.

For channel 2, work on its active shared coordinates.
Its kernel is $\one\otimes\Ker S_D$.
For $z$ orthogonal to this kernel, decompose
$z=z_0+z_\perp$ into consensus and disagreement.
Then $\|Dz_0\|^2\ge\underline\mu_D\|z_0\|^2$, and
\[
\|Dz\|^2+\gamma^2\|\widetilde W_qz\|^2
\ge\tfrac12\underline\mu_D\|z_0\|^2+
       (\gamma^2/4-d_D)\|z_\perp\|^2
\ge\tfrac12\underline\mu_D\|z\|^2.
\]
The upper bound is $d_D+(9/4)\gamma^2\le19\bar d_D$.

At fixed $\alpha,\widetilde W$, replacing $C$ by $C_\star$
and positive row scaling preserve channel 1's row space.
Channel 2 graph filtering and positive scaling preserve its
row space. Replacing $W$ inside the lifted channel 1 generally
changes its full lifted row space.
Apply Proposition~\ref{prop:decentralized-bridge} to this
final lift: $\underline\sigma_W=1/2$ and
$\alpha^2\underline\sigma_W^2=3\bar h$ give the asserted
angle bound. The same proposition gives
$\mu_{\Phi_{\rm loc}}\ge\mu_f/2$ and $L_{\Phi_{\rm loc}}\le L_f+14\mu_f$.
\end{proof}

\paragraph{Zero blocks and reductions.}
If $H=0$, feasibility makes the original coupled equation
vacuous: remove that equation and the auxiliary variable,
take $\Phi_{\rm loc}=F$, and use only $[C_\star\ 0]$ for channel 1.
Its nonzero Gram spectrum is in $[1/2,3/2]$.
If $D=0$, omit the $D$ rows but retain
$\widetilde W_q$ with $\gamma=1$; its nonzero Gram spectrum
is in $[1/4,9/4]$.
If $C=0$, omit its rows and the $\beta$ scaling.
For $n=1$, omit both graph terms and the auxiliary variable
and take $\Phi_{\rm loc}=F$; the bounds
\eqref{eq:structured-channel-intervals} still hold for the
remaining balanced nonzero $H,D$ branches.
Completely empty channels need no calls.
When $B=0$, $S_{HC}$ and $k_1$ reduce to the predecessor
coupled/local quantities and $\chix=1$.
When $C=0$, $S_{HC}=n^{-1}\sum_iH_iH_i^\top$ and the local
normalization layer disappears.
The substitution is a channel spectral reduction, not a
reduction of the full cross-coupled problem.

\paragraph{Final normalization and outer certificate.}
For each nonempty channel, set $Q_j=\mathcal L_j\mathcal L_j^\top$,
$U_j=\Range(\mathcal L_j^\top)$ and
$\widehat L_j=p_j(Q_j)\mathcal L_j$.
For fixed $0<\delta_{\rm wh}<1$, independent of $\eps$,
Proposition~\ref{prop:poly-normalization} and
\eqref{eq:structured-channel-intervals} give degrees
$N_j^{\rm wh}+1=O(\sqrt{\bar k_j})$ and
\[
(1-\delta_{\rm wh})P_{U_j}
\preceq\widehat L_j^\top\widehat L_j
\preceq(1+\delta_{\rm wh})P_{U_j}.
\]
Transform each final right-hand side by the same $p_j(Q_j)$.
This last transformation preserves each final channel row
space. These spectral estimates concern $\mathcal L_j$,
not automatically the literal unbalanced $L_j$ of
\eqref{eq:local-bridge-channels}.
For $\widehat L=[\widehat L_1;\widehat L_2]$,
Lemma~\ref{lem:joint-cheb} permits common rows to remain and
a supplied $\Xi\ge\chix$ gives
\begin{equation}
\operatorname{spec}^+(\widehat L^\top\widehat L)
\subset\left[
\frac{1-\delta_{\rm wh}}{3\Xi^2},
\,2(1+\delta_{\rm wh})\right].
\label{eq:normalized-joint-certificate}
\end{equation}
The positive graph condition is already constant inside
$\mathcal L_j$. Its physical cost
$O(\sqrt{\bar\kappa_W})$ is therefore charged once per graph
action, not again inside $N_j^{\rm wh}$.
The fixed normalization accuracy contributes no further
$\log(1/\eps)$ factor.

\section{Cross-APAPC Algorithm and Smooth-SC Upper Bound}
\label{app:algorithmproof}
\subsection{Chebyshev correction used inside Cross-APAPC}
Let $Q=\widehat K^\top\widehat K$, $c=\widehat K^\top\widehat v$, and suppose
$0<\lambda_-\le\lambda_{\min}^+(Q)\le\lambda_{\max}(Q)\le\lambda_+$.
Algorithm~\ref{alg:cheb-correction} is the classical Chebyshev semi-iteration for the consistent normal equation $Qz=c$ (Algorithm~2 of \citet{salim2022optimal} with $K,b$ replaced by $\widehat K,\widehat v$).  It is an inner correction routine, not a second outer optimization method.

\begin{algorithm}[H]
\caption{Chebyshev normal-equation correction $\operatorname{Cheb}_N(z_0;\widehat K,\widehat v)$}
\label{alg:cheb-correction}
\begin{algorithmic}[1]
\Require $z_0$, $\widehat K$, $\widehat v$, interval $[\lambda_-,\lambda_+]$, degree $N\ge1$
\State $\varrho\gets(\lambda_+-\lambda_-)^2/16$, $\nu\gets(\lambda_++\lambda_-)/2$, $\gamma_0\gets-\nu/2$
\State $p_0\gets-\widehat K^\top(\widehat Kz_0-\widehat v)/\nu$, $z_1\gets z_0+p_0$
\For{$i=1,\ldots,N-1$}
  \State $\beta_{i-1}\gets\varrho/\gamma_{i-1}$, $\gamma_i\gets-(\nu+\beta_{i-1})$
  \State $p_i\gets[\widehat K^\top(\widehat Kz_i-\widehat v)+\beta_{i-1}p_{i-1}]/\gamma_i$
  \State $z_{i+1}\gets z_i+p_i$
\EndFor
\State \Return $z_N$
\end{algorithmic}
\end{algorithm}
Each inner step uses one multiplication by $\widehat K$ and one by $\widehat K^\top$.  The degree $N=\lceil\sqrt{\lambda_+/\lambda_-}\rceil$ used by Cross-APAPC is what converts the joint normalized condition number into the affine-oracle factor.

\subsection{Complete \textsc{Cross-APAPC} specification}
Choose $J$ execution channels $M_1,\ldots,M_J$ with right-hand sides $v_j$.  For canonical or local \CCMAC{} take $J=2$ and respectively $M_j=K_j$ or $M_j=\mathcal L_j$, including the preliminary transformations in Lemma~\ref{lem:structured-normalization}.  Let $U_j:=\Range(M_j^\top)$ and choose channelwise maps $S_j$ such that
\[
\widehat K_j:=S_jM_j,\qquad
 aP_{U_j}\preceq \widehat K_j^\top\widehat K_j\preceq bP_{U_j},
\qquad 0<a\le b,
\]
with universal $a,b$.  Set $\widehat v_j=S_jv_j$ and
\[
\widehat K:=\col(\widehat K_1,\ldots,\widehat K_J),\qquad
\widehat v:=\col(\widehat v_1,\ldots,\widehat v_J).
\]
Injectivity of $S_j$ on $\Range(M_j)$ preserves feasibility.  Supply bounds
\[
0<\lambda_-\le\lambda_{\min}^+(\widehat K^\top\widehat K),\qquad
\lambda_{\max}(\widehat K^\top\widehat K)\le\lambda_+.
\]
For a generic $J$-channel instance the certified ratio is $O(\chiJ^2)$.  For two-channel \CCMAC{}, any supplied $\Xi\ge\chix$ permits
\[
\lambda_-={a}/{(3\Xi^2)},\qquad \lambda_+=2b,
\]
by Lemma~\ref{lem:joint-cheb} and the canonical-to-local bridge.  Wider certified intervals are also valid; their endpoint ratio determines the Chebyshev degree.

For the standard APAPC estimate, let $P$ be the Chebyshev-preconditioned normal operator, $p^\star=-\nabla\Phi(w^\star)\in\Range(P)$, and $y^\star=P^{\dagger/2}p^\star$. A sufficient supplied initial-energy bound is
\[
R_E^2\ge \|w^0-w^\star\|^2+\frac{\eta}{\theta}\|y^0-y^\star\|^2
+\frac{2\eta(1-\tau)}{\tau}D_\Phi(w_f^0,w^\star),
\qquad y^0=0,
\]
where $D_\Phi(a,b):=\Phi(a)-\Phi(b)-\langle\nabla\Phi(b),a-b\rangle$. This energy, rather than the primal radius alone, controls the convergence prefactor.

The complete outer iteration is Algorithm~\ref{alg:cross-apapc} in Section~\ref{sec:algorithm}; only its inner Chebyshev correction is repeated here because its affine-call count is used explicitly in the proof.

\subsection{Proof of Theorem~\ref{thm:algorithm}}
Fix the execution channels $M_1,\ldots,M_J$ and normalized stack specified above. Feasibility gives $v_j\in\Range(M_j)$; injectivity of $S_j$ there implies $S_j(M_jw-v_j)=0$ iff $M_jw-v_j=0$. Proposition~\ref{prop:multi-channel} gives
\begin{equation}
\kappa(\widehat K)\le\frac ba\,\chiJ^2.
\label{eq:khat-bound}
\end{equation}
For the canonical \CCMAC{} oracle, $J=2$, $M_j=K_j$, $\Phi=F$, and $\chi_{\times,2}\le\sqrt2\,\chix$. For the local implementation, $M_j=\mathcal L_j$; Proposition~\ref{prop:decentralized-bridge} gives $\chi_{\rm loc}\le\sqrt3\,\chix$, while Proposition~\ref{prop:multi-channel} then gives $\chi_{\times,2}^{\rm loc}\le\sqrt2\chi_{\rm loc}=O(\chix)$ and $\kappa_\Phi=\kappa_{\Phi_{\rm loc}}=O(\kappa_f)$. The zero-block branches use the reduced channels specified there.

Algorithm~\ref{alg:cross-apapc} is the affine-constrained accelerated method of \citet{salim2022optimal} for $\Phi$ and $\widehat Kw=\widehat v$. With the supplied joint interval its Chebyshev degree is $N=\lceil\sqrt{\lambda_+/\lambda_-}\rceil$. With the energy certificate above and $T=O(\sqrt{\kappa_\Phi}\log_+(R_E/\eps))$, it uses
\[
O\!\left(\sqrt{\kappa_\Phi}\log_+\frac{R_E}{\eps}\right)
\quad\text{gradients and}\quad
O\!\left(N\sqrt{\kappa_\Phi}\log_+\frac{R_E}{\eps}\right)
\]
normalized joint-stack products. The assumed ratio $O(\chiJ^2)$ gives the normalized affine bound in~\eqref{eq:sc-upper-compact}; in the two-channel \CCMAC{} specialization, using~\eqref{eq:normalized-joint-certificate} gives $N=O(\Xi)$.

Each stack action uses each normalized channel once. The supplied graph realization converts the same stack-product count into the communication bound in~\eqref{eq:sc-upper-compact} after summing the per-channel communication costs. In the local \CCMAC{} case each evaluation of $\nabla\Phi_{\rm loc}$ additionally uses $O(1)$ actions of $H,H^\top$ and $\widetilde W$; these are dominated by the explicit channel costs in Appendix~\ref{app:resources}. The gradient count refers to the original $\nabla f_i$. This proves Theorem~\ref{thm:algorithm}. \qed

\section{Smooth-SC Lower Bounds and Their Scope}\label{app:lower-bound-program}
\paragraph{Oracle and radius conventions.}
The lower bounds use a linear-span oracle with zero initial vector memories. New primal and dual vectors may be formed only by scalar linear combinations of stored vectors and the specified gradient, channel, adjoint-channel, and neighbor-communication operations. Access to matrix entries, arbitrary new basis vectors, exact factorizations, and unconstrained global projections is not included. Free local operations in a communication lower bound mean closure under these allowed local maps, not unrestricted coordinate generation.

For reference, if $\mathfrak P$ is an explicitly specified class and $r\in\{\nabla,\mathrm{aff},\mathrm{comm}\}$, define
\begin{equation}
\mathfrak C_r(\eps;\mathfrak P):=\inf_{\mathcal A}\sup_{\mathcal I\in\mathfrak P}N_r(\mathcal A,\mathcal I,\eps).
\label{eq:minimax-resource-def}
\end{equation}
The energy-controlled upper class of Theorem~\ref{thm:algorithm} and the primal-radius-normalized hard families below are not silently identified. The statements that follow isolate particular resource mechanisms and record the radius hypotheses used by each construction.

\begin{theorem}[Two-channel affine-oracle lower bound]\label{thm:affine-lower}
For every $N\ge4$, $\kappa_f\ge4$, and sufficiently small universal $\eps>0$, there is a smooth strongly convex one-agent CC-MAC instance, normalized so that $\|w^0-w^\star\|=1$, with two individually whitened channels and $\chix=\cot(\pi/(2N))=\Theta(N)$, requiring $\Omega(\sqrt{\kappa_f}\chix\log(1/\eps))$ composite normalized affine calls and $\Omega(\sqrt{\kappa_f}\log(1/\eps))$ gradients in the stated linear-span oracle.
\end{theorem}
We now prove Theorem~\ref{thm:affine-lower}. The construction is a two-coloring of the edges of a path. It is useful because each color class is a matching and therefore an isometric affine channel, whereas the union retains the full path spectrum.

\begin{lemma}[Alternating-matchings path geometry]
\label{lem:path-matching}
Let $E_N\in\R^{(N-1)\times N}$ be the oriented incidence matrix of the path $P_N$, with row $i$ equal to $e_i^\top-e_{i+1}^\top$. Let $E_{\rm odd}$ and $E_{\rm even}$ contain respectively the odd and even rows of $E_N$, and define
\[
M_{\rm odd}=2^{-1/2}E_{\rm odd},\qquad M_{\rm even}=2^{-1/2}E_{\rm even},\qquad
M=\begin{bmatrix}M_{\rm odd}\\M_{\rm even}\end{bmatrix}.
\]
Then every positive singular value of each $M_{\rm odd}$ and $M_{\rm even}$ equals one, while
\begin{equation}
M^\top M=\frac12L_{P_N},\qquad
\kappa(M)=\frac{1+\cos(\pi/N)}{1-\cos(\pi/N)}=\cot^2\frac{\pi}{2N}.
\label{eq:path-gram}
\end{equation}
Consequently the Friedrichs cosine of the two row spaces equals $\cos(\pi/N)$ and their cross factor is $\cot(\pi/(2N))$.
\end{lemma}

\begin{proof}
Within either parity class no two path edges share a vertex. Hence the nonzero rows of $E_{\rm odd}$ (and of $E_{\rm even}$) are mutually orthogonal and each has squared norm $2$. After multiplication by $2^{-1/2}$ the rows are orthonormal, so $M_jM_j^\top=\mathrm{Id}$ on its row space and every positive singular value is one.

Stacking the two parity classes merely permutes the rows of $2^{-1/2}E_N$, hence $M^\top M=E_N^\top E_N/2=L_{P_N}/2$. The path Laplacian has eigenvalues
\[
2-2\cos\frac{k\pi}{N},\qquad k=0,1,\ldots,N-1.
\]
Therefore the smallest positive and largest eigenvalues of $M^\top M$ are $1-\cos(\pi/N)$ and $1+\cos(\pi/N)$, respectively, which proves \eqref{eq:path-gram}. Since the two channels are individually whitened, the exact-whitening identity in Theorem~\ref{thm:angle} gives
\[
\frac{1+c_F}{1-c_F}=\kappa(M)=\frac{1+\cos(\pi/N)}{1-\cos(\pi/N)}.
\]
The map $c\mapsto(1+c)/(1-c)$ is injective on $[0,1)$, so $c_F=\cos(\pi/N)$ and $\chix=\cot(\pi/(2N))$.
\end{proof}

\subsection{Proof of Theorem~\ref{thm:affine-lower}}
Fix $N\ge4$. Use the smooth strongly convex hard functions on a linear graph from the proof of Theorem 2 of \citet{scaman2017optimal}, with local condition number $\kappa_f$. Their coordinate-tail argument (before the final conversion to function error) gives a product objective $F_N$ for which every black-box procedure using the path gossip operator
\[
W_N:=\frac12L_{P_N}
\]
requires
\begin{equation}
\Omega\!\left(\sqrt{\kappa_f}\,\frac1{\sqrt{\gamma_N}}\log\frac1\eps\right)
\quad\text{applications of }W_N,
\qquad
\gamma_N:=\frac{\lambda_{\min}^+(W_N)}{\lambda_{\max}(W_N)},
\label{eq:scaman-path-lb}
\end{equation}
and $\Omega(\sqrt{\kappa_f}\log(1/\eps))$ gradient computations. For the path spectrum above,
\[
\gamma_N=\frac{1-\cos(\pi/N)}{1+\cos(\pi/N)}=\chix^{-2}.
\]

Embed this problem into a one-agent \CCMAC{} instance. Introduce a dummy scalar local variable $x_1$ with the decoupled quadratic $\mu_f x_1^2/2$, set $A_1=C_1=0$, and let the shared variable $z$ equal the product variable of $F_N$. Take $B_1=M_{\rm odd}\otimes\mathrm{Id}$, $D_1=M_{\rm even}\otimes\mathrm{Id}$, and zero right-hand sides. The dummy block does not change the condition number or the oracle lower bound. (A finite hard-coordinate dimension large enough for the requested accuracy is sufficient for the standard resisting-oracle construction.) Thus the two nonzero affine channels are precisely
\[
K_1=M_{\rm odd}\otimes\mathrm{Id},\qquad K_2=M_{\rm even}\otimes\mathrm{Id},
\]
and Lemma~\ref{lem:path-matching} gives their individual condition numbers and cross factor.

It remains to match the oracle classes. Let $K=[K_1;K_2]$ and $b=0$. In the black-box first-order (BBFO) linear-span model of \citet{salim2022optimal}, every primal vector that can be generated using $K$ and $K^\top$ belongs to spans containing terms of the form
\[
K^\top Kx=W_Nx,\qquad K^\top K\nabla F_N(x)=W_N\nabla F_N(x),
\]
(up to the identity on the hard coordinate). More explicitly,
\[
K^\top\operatorname{Span}\{Kx^s,K\nabla F_N(x^s):s\le t\}
=\operatorname{Span}\{W_Nx^s,W_N\nabla F_N(x^s):s\le t\}.
\]
Thus the exact span-reduction used in the proof of Theorem 1 of \citet{salim2022optimal} applies verbatim: a BBFO procedure with $T$ composite $K/K^\top$ oracle calls induces a black-box procedure using at most a constant multiple of $T$ applications of $W_N$. Equation~\eqref{eq:scaman-path-lb} therefore implies
\[
T=\Omega\!\left(\sqrt{\kappa_f}\,\chix\log\frac1\eps\right).
\]
Because a composite oracle is at least as informative per charged call as querying either channel separately, the same lower bound holds for the total number of separately charged channel matrix--vector calls. The gradient lower bound is inherited from the same Scaman construction. The unit-radius normalization in the theorem is without loss: for any quadratic hard instance with solution $w^\star$, the transformation
\[
F_c(w):=c^2F(w/c),\qquad v_c:=cv
\]
leaves $L_f$, $\mu_f$, all channel row spaces, and all matrix/network condition numbers unchanged while scaling the optimizer and every oracle-generated primal span by $c$. Choosing $c=1/\|w^0-w^\star\|$ therefore gives unit initial radius and replaces absolute accuracy by the corresponding relative accuracy. This proves Theorem~\ref{thm:affine-lower}. \qed

\begin{theorem}[Radius-explicit decentralized communication lower bound]\label{thm:comm-lower}
For the path family constructed below, let $t>0$, $r=\exp[-\operatorname{arcosh}(1+3/(\kappa_ft^2))]$, $E_I=(1-r)(3+r)/2$, and let the three physical groups have size $m_{\rm grp}$. After normalizing the full execution-space radius to one, every legal linear-span method satisfies
\begin{equation}
N_{\rm comm}\ge(m_{\rm grp}+1)\left[\frac{1}{-\log r}
\log\!\left(\frac{\sqrt{3/(3+t^2E_I)}}{\eps}\right)-1\right]_+.
\label{eq:target-comm-lb}
\end{equation}
This construction displays the cross-angle/network propagation mechanism, but the prefactor depends on $t$ and is not asserted as a uniform $\Omega(\chix\sqrt{\kappa_f\kappa_W}\log(1/\eps))$ bound at fixed absolute accuracy.
\end{theorem}
\subsection{Radius-explicit decentralized communication lower bound}
We now prove Theorem~\ref{thm:comm-lower}. The construction adapts the alternating-coordinate propagation argument of \citet[Appendix C]{yarmoshik2025coupled}, but the bad affine parameter is generated by the angle between a local cross-link channel and the shared-variable consensus channel. In particular, the local cross-link channel itself has condition number one.

\paragraph{Alternating partial isometries.}
Work first on the hard coordinate space $\ell_2(\mathbb N)$. Define $R_1,R_2$ by
\begin{align}
(R_1u)_1&=u_1, & (R_1u)_{k+1}&=\frac{u_{2k}-u_{2k+1}}{\sqrt2},\quad k\ge1,\label{eq:R1-def}\\
(R_2u)_k&=\frac{u_{2k-1}-u_{2k}}{\sqrt2},&&\quad k\ge1.\label{eq:R2-def}
\end{align}
The rows within each operator have disjoint supports and unit norm, hence
\begin{equation}
R_1R_1^\top=\mathrm{Id},\qquad R_2R_2^\top=\mathrm{Id}.
\label{eq:R-isometry}
\end{equation}
Let $P_j=R_j^\top R_j$. For $k\ge2$,
\begin{equation}
[(P_1+P_2)u]_k=u_k-\frac12u_{k-1}-\frac12u_{k+1},
\label{eq:P-interior}
\end{equation}
while the first coordinate is $\frac32u_1-\frac12u_2$. Thus $P_1$ and $P_2$ are complementary matching projectors whose union is a one-dimensional zero chain.

\paragraph{A legal cross-coupled instance.}
Let the physical graph be the path on $n=3m_{\rm grp}$ agents and partition its consecutive vertices into $V_1,V_2,V_3$, each of size $m_{\rm grp}$. For $i\in V_1$ introduce a local variable $x_i$ and the affine relation
\begin{equation}
x_i=tR_1u,
\label{eq:link-v1}
\end{equation}
while for $i\in V_3$ use $x_i=tR_2u$. Agents in $V_2$ have no hard local variable (equivalently, one may add a decoupled dummy scalar if strictly positive local dimensions are required). These relations are a special case of the original cross-coupled equality: take the coupled output space to be the direct sum of one copy of the row space for every active agent, let $J_i$ inject into the $i$th summand, and set
\begin{equation}
A_i=J_i,\qquad B_i=-tJ_iR_1\quad(i\in V_1),
\qquad
A_i=J_i,\qquad B_i=-tJ_iR_2\quad(i\in V_3),
\label{eq:hard-A-B}
\end{equation}
with zero blocks on $V_2$. Because the ranges of the $J_i$ are mutually orthogonal, the aggregate equality is equivalent to the local links \eqref{eq:link-v1} and their $V_3$ analogues. Hence this is a core \CCMAC{} instance with $C_i=D_i=0$ and nonzero cross blocks $B_i$. For this subclass the aggregate output is already a direct sum of locally owned blocks, so the generic auxiliary localization variable used elsewhere in the paper is unnecessary; adding it would only enlarge the oracle available to the algorithm and is therefore omitted in the lower-bound representation.

After introducing local copies $u_i$ of the shared variable, the first affine channel is
\begin{equation}
\mathcal K_1:(x,\bm u)\mapsto
\col\bigl(x_i-tR_{g(i)}u_i\bigr)_{i\in V_1\cup V_3},
\qquad g(i)=\begin{cases}1,&i\in V_1,\\2,&i\in V_3,\end{cases}
\label{eq:hard-K1}
\end{equation}
and the second channel is the consensus operator $\mathcal K_2=(W\otimes\mathrm{Id})$ acting on $\bm u$. Equation~\eqref{eq:R-isometry} gives
\begin{equation}
\mathcal K_1\mathcal K_1^\top=(1+t^2)\mathrm{Id},
\qquad \kappa(\mathcal K_1)=1.
\label{eq:hard-K1-cond}
\end{equation}
Let $U=\Range(\mathcal K_1^\top)$ and let $V$ be the row space of the consensus channel; for a connected graph, $V$ is precisely the block-disagreement subspace in the $u$ coordinates. Every $a$ in the row-coefficient space of $\mathcal K_1$ satisfies
\[
\|\mathcal K_1^\top a\|^2=(1+t^2)\|a\|^2,
\qquad
\|P_V\mathcal K_1^\top a\|\le t\|a\|.
\]
Equality is attained by choosing two agents in the same outer group with opposite row coefficients. Therefore
\begin{equation}
\cF(U,V)=\frac{t}{\sqrt{1+t^2}},
\qquad
\chix=\sqrt{\frac{1+\cF(U,V)}{1-\cF(U,V)}}=\sqrt{1+t^2}+t.
\label{eq:hard-cross-angle-proof}
\end{equation}
Separate left preconditioning of either channel does not change these row spaces, by Lemma~\ref{lem:rowspace-invariance}.

\paragraph{The hard objective and its geometric tail.}
If $t=0$, then $\chix=1$ and the standard consensus hard block used below already proves the claimed order. Hence it remains to treat $t>0$. On the cross block let the active agents use
\begin{equation}
f_i^\times(x_i,u)=
\begin{cases}
\frac{L_f}{2}\|x_i-\xi e_1\|^2+\frac{\mu_f}{2}\|u\|^2,&i\in V_1,\\
\frac{L_f}{2}\|x_i\|^2+\frac{\mu_f}{2}\|u\|^2,&i\in V_3,
\end{cases}
\label{eq:hard-local-objective}
\end{equation}
and let agents in $V_2$ contribute only $\mu_f\|u\|^2/2$. These functions are $L_f$-smooth and $\mu_f$-strongly convex in their existing variables. On the feasible set, after consensus and elimination of the $x_i$, the cross block equals, up to the positive factor $m_{\rm grp}$ and an irrelevant constant,
\begin{equation}
\phi_\times(u)=\frac{3\mu_f}{2}\|u\|^2
+\frac{L_ft^2}{2}\langle u,(P_1+P_2)u\rangle
-L_ft\xi u_1.
\label{eq:hard-effective-objective}
\end{equation}
Set
\begin{equation}
\delta:=\frac{3\mu_f}{L_ft^2}=\frac{3}{\kappa_ft^2},
\qquad
\alpha:=\operatorname{arcosh}(1+\delta),
\qquad
r:=e^{-\alpha}\in(0,1).
\label{eq:hard-r}
\end{equation}
For $c_r:=\sqrt{1-r^2}$ choose
\begin{equation}
\xi=\frac{c_r}{L_ft}\left(3\mu_f+\frac{L_ft^2}{2}(3-r)\right).
\label{eq:hard-xi}
\end{equation}
Then the unique minimizer of \eqref{eq:hard-effective-objective} is
\begin{equation}
u^\star=c_r(1,r,r^2,\ldots),\qquad \|u^\star\|=1.
\label{eq:hard-minimizer}
\end{equation}
Indeed, the first optimality equation is exactly enforced by \eqref{eq:hard-xi}; for every $k\ge2$, \eqref{eq:P-interior} reduces the optimality condition to
\[
(1+\delta)u_k-\frac12u_{k-1}-\frac12u_{k+1}=0,
\]
which holds because $(r+r^{-1})/2=\cosh\alpha=1+\delta$. The linked local variables contribute to the full execution-space radius. Direct evaluation gives
\[
E_I:=\langle u^\star,(P_1+P_2)u^\star\rangle=\frac{(1-r)(3+r)}2,
\qquad R_{\rm full}^2=m_{\rm grp}(3+t^2E_I).
\]
After scaling the complete execution vector to unit radius, any vector supported on the first $s$ hard coordinates has
\begin{equation}
\|w-w^\star\|\ge \sqrt{\frac{3}{3+t^2E_I}}\,r^s.
\label{eq:hard-tail}
\end{equation}

\paragraph{Propagation through the physical path.}
We may only strengthen the algorithm when proving a lower bound. Thus, between communication rounds, give every node in $V_1$ free access to the projector $P_1$, every node in $V_3$ free access to $P_2$, and every node arbitrary local linear algebra and gradients. This oracle contains all hard-coordinate information obtainable from the original local maps $R_j,R_j^\top$ and the quadratic gradients: returning from the local $x_i$ space to the shared coordinate can only produce $R_j^\top R_j=P_j$, together with the fixed seed $e_1$ in $V_1$.

Starting from the zero span (the standard resisting-oracle reduction handles arbitrary fixed initial vectors), let $s(N_{\rm round})$ be the largest hard-coordinate index that can occur at any node after $N_{\rm round}$ neighbor communication rounds. The matching structure gives
\begin{equation}
P_1\operatorname{span}\{e_1,\ldots,e_s\}\subseteq
\begin{cases}
\operatorname{span}\{e_1,\ldots,e_{s+1}\},&s\text{ even},\\
\operatorname{span}\{e_1,\ldots,e_s\},&s\text{ odd},
\end{cases}
\label{eq:P1-support}
\end{equation}
and the two cases are reversed for $P_2$. Hence creating coordinate $2$ requires the seed from $V_1$ to reach $V_3$; creating coordinate $3$ then requires the new information to return to $V_1$, and so on. We give the propagation induction explicitly.

For $s\ge1$, let $T_s$ be the first communication round after which some node can contain a vector with a nonzero $s$th hard coordinate. The linear term in \eqref{eq:hard-local-objective} and one free local return through $R_1^\top$ give $T_1=0$ in $V_1$. Suppose coordinate $s$ is first exposed. If $s$ is odd, then the boundary pair $(s,s+1)$ belongs to the $P_2$ matching, whereas $P_1$ preserves $\operatorname{span}\{e_1,\ldots,e_s\}$; if $s$ is even the roles are reversed. Thus coordinate $s+1$ can be created only in the outer group opposite to the one whose local matching created coordinate $s$. Before the information containing coordinate $s$ reaches that opposite group, arbitrary local linear combinations, gradients, and applications of the available projector remain inside $\operatorname{span}\{e_1,\ldots,e_s\}$ by \eqref{eq:P1-support}. The closest vertices of $V_1$ and $V_3$ are $m_{\rm grp}+1$ physical edges apart, and one communication round moves information across at most one edge. Consequently
\[
T_{s+1}\ge T_s+(m_{\rm grp}+1).
\]
Induction gives $T_s\ge(s-1)(m_{\rm grp}+1)$. Therefore, after $N_{\rm round}$ neighbor communication rounds, every locally available hard-coordinate vector is supported on the first
\begin{equation}
s(N_{\rm round})\le 1+\left\lfloor\frac{N_{\rm round}}{m_{\rm grp}+1}\right\rfloor
\label{eq:hard-propagation}
\end{equation}
coordinates. This is the same alternating-frontier mechanism as the propagation lemma of \citet[Appendix C]{yarmoshik2025coupled}, written here for the cross-link/consensus pair.

Combining \eqref{eq:hard-tail} and \eqref{eq:hard-propagation}, an $\eps$-accurate output must satisfy
\begin{equation}
N_{\rm round}\ge(m_{\rm grp}+1)\left[\frac{1}{-\log r}
\log\!\left(\frac{\sqrt{3/(3+t^2E_I)}}{\eps}\right)-1\right]_+.
\label{eq:q-before-bound}
\end{equation}
Since $-\log r=\operatorname{arcosh}(1+\delta)$ and $\cosh x\ge1+x^2/2$,
\begin{equation}
-\log r\le\sqrt{2\delta}=\sqrt{\frac{6}{\kappa_ft^2}}.
\label{eq:r-rate}
\end{equation}
Thus the propagation rate contains the factor $m_{\rm grp}t\sqrt{\kappa_f}$, but the logarithmic term retains the full-radius prefactor in \eqref{eq:q-before-bound}; it is not uniform in $t$ at fixed absolute accuracy.
For the unweighted physical path,
\begin{equation}
\kappa_W=\frac{1+\cos(\pi/n)}{1-\cos(\pi/n)}
=\cot^2\frac{\pi}{2n},
\qquad
m_{\rm grp}=\Theta(\sqrt{\kappa_W}).
\label{eq:physical-path-kappa}
\end{equation}
Hence the path contributes $m_{\rm grp}=\Theta(\sqrt{\kappa_W})$ to the radius-explicit lower bound in \eqref{eq:q-before-bound}.

To cover the regime $t<1$ without changing the prescribed angle, append an independent standard consensus hard block to the shared variable and to each local objective, as in \citet{scaman2017optimal}. The cross-link channel acts only on the first block, while the consensus channel acts on both, so \eqref{eq:hard-cross-angle-proof} is unchanged. The direct-sum objective remains $L_f$-smooth and $\mu_f$-strongly convex, and the standard block requires
\begin{equation}
\Omega\!\left(\sqrt{\kappa_f\kappa_W}\log\frac1\eps\right)
\label{eq:standard-consensus-block}
\end{equation}
communications. Every accurate output must solve both blocks. The relation
\begin{equation}
\chix=\sqrt{1+t^2}+t\le(1+\sqrt2)\max\{1,t\},
\label{eq:chi-vs-t}
\end{equation}
shows how $t$ controls the cross factor, but it does not remove the $t$-dependent full-radius prefactor. Equation~\eqref{eq:q-before-bound} is precisely \eqref{eq:target-comm-lb} and proves Theorem~\ref{thm:comm-lower}. Finite-dimensional instances are obtained by truncating the hard coordinate after more coordinates than any method can expose before the claimed lower-bound horizon, exactly as in standard zero-chain lower bounds. \qed

\paragraph{What these lower bounds establish.}
The affine construction and the ordinary unconstrained subclass establish the two lower orders recorded in Theorem~\ref{thm:minimax-sc}. The communication family establishes alternating physical propagation together with the radius-explicit dependence in Theorem~\ref{thm:comm-lower}. A uniform product lower bound over a class controlling only the primal radius would additionally require a hard family whose complete execution-space radius remains uniformly controlled as $t$ varies; that stronger conclusion is not used here.
\section{Smooth-Convex Log-Free Continuation and Matching Lower Bounds}
\label{app:smooth-convex-proof}
\paragraph{Accuracy class and multiplier scale.}
Assume $F$ is convex with $L_f$-Lipschitz gradient, $\|w^0-w^\star\|\le R$, and choose the minimum-norm multiplier $\lambda^\star\in\Range(\bar K)$. Define
\begin{equation}
\Lambda_{\rm cvx}:=\frac{\sigma_{\min}^+(\bar K)\|\lambda^\star\|}{L_fR},
\qquad
F(\bar w)-F^\star\le\eps,\quad
\frac{L_fR}{\sigma_{\min}^+(\bar K)}\|\bar K\bar w-\bar v\|\le\eps.
\label{eq:cvx-accuracy}
\end{equation}
When $\bar K=0$ the feasibility condition is omitted and $\Lambda_{\rm cvx}=0$. The class fixes $L_f,R$, the geometric budgets $\chix\le\Xi$, $\kappa_W\le\Omega$, and $\Lambda_{\rm cvx}\le\bar\Lambda$. Set $T_{\rm cvx}=\sqrt{L_fR^2/\eps}$.

The continuation construction below avoids a multiplicative target-accuracy logarithm in the affine and communication ledgers.  Throughout, $\bar K$ denotes the fixed separately normalized joint equality operator, $\mathcal A:=\{w:\bar Kw=\bar v\}\ne\varnothing$, $\mathcal N:=\Range(\bar K^\top)$, $Q:=\bar K^\top \bar K$, and
\begin{equation}
0<\ell\le\lambda_{\min}^+(Q),\qquad \lambda_{\max}(Q)\le \lambda_{\rm hi},
\qquad \widehat\chi:=\sqrt{\lambda_{\rm hi}/\ell}.
\label{eq:cvxlf-cert}
\end{equation}
The supplied normalization certificates of the main text give $\widehat\chi=O(\chix)$ instancewise and $\widehat\chi=O(\Xi)$ on the class $\chix\le\Xi$.  Let $L:=L_f$, $\sigma:=\sigma_{\min}^+(\bar K)$, and assume $0<\eps\le LR^2$ and $\Lambda_{\rm cvx}\le\bar\Lambda$.

\subsection{Executable filters and one regularized stage}
Let $T_j,U_j$ be Chebyshev polynomials of the first and second kind and put $d_{\rm RF}=d_{\rm UF}:=\lceil\widehat\chi\rceil$.  For $\lambda_{\rm hi}>\ell$ define
\begin{equation}
r_{d_{\rm RF}}(t):=\frac{T_{d_{\rm RF}}((\lambda_{\rm hi}+\ell-2t)/(\lambda_{\rm hi}-\ell))}{T_{d_{\rm RF}}((\lambda_{\rm hi}+\ell)/(\lambda_{\rm hi}-\ell))},
\qquad
s_{d_{\rm UF}}(t):=\frac{U_{d_{\rm UF}}(1-t/\lambda_{\rm hi})}{d_{\rm UF}+1},
\label{eq:cvxlf-polys}
\end{equation}
and for $\lambda_{\rm hi}=\ell$ use $r_1(t)=1-t/\lambda_{\rm hi}$.  Standard Chebyshev identities give
\begin{equation}
r_{d_{\rm RF}}(0)=s_{d_{\rm UF}}(0)=1,\quad |r_{d_{\rm RF}}(t)|<\tfrac12\ (t\in[\ell,\lambda_{\rm hi}]),\quad
|s_{d_{\rm UF}}(t)|\le1,\quad \sqrt t\,|s_{d_{\rm UF}}(t)|\le\sqrt\ell.
\label{eq:cvxlf-filter-bounds}
\end{equation}
Indeed, $\operatorname{arcosh}((\lambda_{\rm hi}+\ell)/(\lambda_{\rm hi}-\ell))\ge2/\widehat\chi$, so the denominator defining $r_{d_{\rm RF}}$ is at least $\cosh2$; for $s_{d_{\rm UF}}$, write $\cos\theta=1-t/\lambda_{\rm hi}$ and use $U_{d_{\rm UF}}(\cos\theta)=\sin((d_{\rm UF}+1)\theta)/\sin\theta$.  Consequently, with
\begin{equation}
P:=\mathrm{Id}-r_{d_{\rm RF}}(Q),
\qquad \tfrac12\Pi_{\mathcal N}\preceq P\preceq\tfrac32\Pi_{\mathcal N},
\qquad \Ker(P)=\Ker(\bar K).
\label{eq:cvxlf-P}
\end{equation}
Define the residual $\mathcal R_{\bar v}(x):=\bar K^\top(\bar Kx-\bar v)$.

\begin{algorithm}[H]\footnotesize
\caption{First-kind residual filter $\RF$}\label{alg:cvx-rf}
\begin{algorithmic}[1]
\Require $x,\bar K,\bar v,\ell,\lambda_{\rm hi}$, $d_{\rm RF}=\lceil\sqrt{\lambda_{\rm hi}/\ell}\rceil$
\If{$\lambda_{\rm hi}=\ell$}\State \Return $x-\mathcal R_{\bar v}(x)/\lambda_{\rm hi}$\EndIf
\State $c\gets(\lambda_{\rm hi}+\ell)/2$, $h\gets(\lambda_{\rm hi}-\ell)/2$, $a\gets c/h$
\State $x_-\gets x$, $x_c\gets x-\mathcal R_{\bar v}(x)/c$, $\rho\gets1/a$
\For{$j=1,\ldots,d_{\rm RF}-1$}
\State $\rho_+\gets(2a-\rho)^{-1}$
\State $x_+\gets x_c+\rho_+\rho(x_c-x_-)-(2\rho_+/h)\mathcal R_{\bar v}(x_c)$
\State $(x_-,x_c,\rho)\gets(x_c,x_+,\rho_+)$
\EndFor
\State \Return $x_c$
\end{algorithmic}
\end{algorithm}

\begin{algorithm}[H]\footnotesize
\caption{Second-kind weighted filter $\UF$}\label{alg:cvx-uf}
\begin{algorithmic}[1]
\Require $x,\bar K,\bar v,\lambda_{\rm hi}$, integer $d_{\rm UF}\ge1$
\State $x_-\gets x$, $x_c\gets x-\mathcal R_{\bar v}(x)/\lambda_{\rm hi}$
\For{$j=1,\ldots,d_{\rm UF}-1$}
\State $x_+\gets\dfrac{2(j+1)}{j+2}(x_c-\mathcal R_{\bar v}(x_c)/\lambda_{\rm hi})-\dfrac{j}{j+2}x_-$
\State $(x_-,x_c)\gets(x_c,x_+)$
\EndFor
\State \Return $x_c$
\end{algorithmic}
\end{algorithm}

For every analytical feasible reference $w_{\rm ref}\in\mathcal A$, Algorithms~\ref{alg:cvx-rf}--\ref{alg:cvx-uf} return
\begin{equation}
\RF(x)=w_{\rm ref}+r_{d_{\rm RF}}(Q)(x-w_{\rm ref}),\qquad
\UF(x)=w_{\rm ref}+s_{d_{\rm UF}}(Q)(x-w_{\rm ref}),
\label{eq:cvxlf-maps}
\end{equation}
although $w_{\rm ref}$ is never supplied or computed.  This follows by dividing the three-term recurrences for $T_j$ and $U_j$ by their normalizing factors.  Each filter evaluation uses its polynomial degree many residual evaluations; each residual evaluation is one forward and one adjoint affine call.

Define, only for analysis,
\begin{align}
\mathcal H(w)&:=F(w)+\frac L2\langle w-w_{\rm ref},P(w-w_{\rm ref})\rangle,\nonumber\\
\phi_\nu(w)&:=\mathcal H(w)+\frac\nu2\|w-w^0\|^2,\qquad 0<\nu\le L.
\label{eq:cvxlf-phi}
\end{align}
The augmentation and its gradient vanish on $\mathcal A$, while its executable gradient is
\begin{equation}
\nabla\phi_\nu(w)=\nabla F(w)+L\bigl(w-\RF(w)\bigr)+\nu(w-w^0).
\label{eq:cvxlf-phigrad}
\end{equation}
Let $a_\nu$ minimize $\phi_\nu$ on $\mathcal A$, set $p_\nu=-\nabla\phi_\nu(a_\nu)\in\mathcal N$, and let $\bar K^\top\lambda_\nu=p_\nu$ with minimum-norm $\lambda_\nu\in\Range(\bar K)$.  Comparison with $w^\star$ and Lipschitzness of $\nabla F$ give
\begin{equation}
\|a_\nu-w^0\|\le R,\quad \|a_\nu-w^\star\|\le2R,\quad
F(a_\nu)-F^\star\le\frac\nu2R^2,\quad
\|\lambda_\nu\|\le(\bar\Lambda+3)LR/\sigma.
\label{eq:cvxlf-reggeom}
\end{equation}
Set $\tau=\tfrac12\sqrt{\nu/L}$, $\eta=(12\sqrt{L\nu})^{-1}$, and $\theta=4\sqrt{L\nu}$.

\begin{algorithm}[H]\footnotesize
\caption{One augmented Cross-APAPC continuation step}\label{alg:cvx-stage}
\begin{algorithmic}[1]
\Require $x,f,p,w^0,F,\bar K,\bar v,L,\nu,\ell,\lambda_{\rm hi},d_{\rm RF}$, with $p\in\mathcal N$
\State $g\gets\tau x+(1-\tau)f$
\State $c_g\gets g-\RF(g;\bar K,\bar v,\ell,\lambda_{\rm hi},d_{\rm RF})$
\State $g_\phi\gets\nabla F(g)+Lc_g+\nu(g-w^0)$
\State $x_{1/2}\gets\bigl(x-\eta(g_\phi-\nu g+p)\bigr)/(1+\eta\nu)$
\State $c_{1/2}\gets x_{1/2}-\RF(x_{1/2};\bar K,\bar v,\ell,\lambda_{\rm hi},d_{\rm RF})$
\State $\Delta p\gets\theta c_{1/2}$, $p^+\gets p+\Delta p$
\State $x^+\gets x_{1/2}-\eta\Delta p/(1+\eta\nu)$
\State $f^+\gets g+\dfrac{2\tau}{2-\tau}(x^+-x)$
\State \Return $(x^+,f^+,p^+)$
\end{algorithmic}
\end{algorithm}

For the analytical energy
\begin{align}
E_\nu(x,f,p):={}&12\|x-a_\nu\|^2
+\frac{1}{12L\nu}\left\langle p-p_\nu,
\left(3P^\dagger-\frac{\mathrm{Id}}{1+\eta\nu}\right)(p-p_\nu)\right\rangle\nonumber\\
&+\frac{4(1-\tau)}{\nu}D_{\phi_\nu}(f,a_\nu),
\label{eq:cvxlf-energy}
\end{align}
one has
\begin{equation}
E_\nu(x^+,f^+,p^+)\le
\left(1+\frac1{24}\sqrt{\nu/L}\right)^{-1}E_\nu(x,f,p).
\label{eq:cvxlf-contract}
\end{equation}
To verify this, take the proof-only equality operator $A_P:=P^{1/2}$ and write $p=A_P^\top y$, $p_\nu=A_P^\top y_\nu$.  Algorithm~\ref{alg:cvx-stage} is exactly the APAPC step of \citet{salim2022optimal} applied to $\phi_\nu$ with valid certificates $L_{\rm in}=6L$, $\mu_{\rm in}=\nu$, $\lambda_1=3$, $\lambda_2=1/2$.  On $\mathcal N$, $\mathrm{Id}\preceq3P^\dagger-(1+\eta\nu)^{-1}\mathrm{Id}\preceq6\mathrm{Id}$, so their metric energy is precisely~\eqref{eq:cvxlf-energy} up to the common scalar $\sqrt{L\nu}$.  Their one-step inequality is valid for arbitrary $y-y_\nu\in\Range(A_P)$; no zero-dual initialization is used.  Hence the currently retained dual state is admissible, which is the point needed by the continuation argument.

\subsection{Warm starts, complete algorithm, and exact ledger}
Let $g_0=\nabla F(w^0)$ and initialize
\begin{equation}
p^0:=\UF(g_0;\bar K,0,\lambda_{\rm hi},d_{\rm UF})-g_0.
\label{eq:cvxlf-init}
\end{equation}
The weighted bound in~\eqref{eq:cvxlf-filter-bounds}, applied to $-\bar K^\top\lambda^\star$, yields
\begin{equation}
E_L(w^0,w^0,p^0)\le\left[19+\frac{(\bar\Lambda+5)^2}{2}\right]R^2.
\label{eq:cvxlf-init-energy}
\end{equation}
Moreover, if $E_\nu(x,f,p)\le R^2$ and $\nu'=\nu/2$, then retaining $p$ and resetting $x\gets f$ gives
\begin{equation}
E_{\nu'}(f,f,p)\le C_{\rm ws}R^2<178R^2<200R^2.
\label{eq:cvxlf-warm}
\end{equation}
For completeness, the estimates behind~\eqref{eq:cvxlf-warm} are
\begin{equation}
\|p-p_\nu\|\le\sqrt{12L\nu}R,\quad
D_{\phi_\nu}(f,a_\nu)\le\nu R^2/2,\quad
\|f-\Pi_{\mathcal A}f\|\le\sqrt{2\nu/L}R,
\label{eq:cvxlf-completed}
\end{equation}
and, for $a=a_\nu$, $a'=a_{\nu/2}$,
\begin{equation}
\|\nabla F(a)-\nabla F(a')\|\le\sqrt{L\nu}R,\qquad
\|p_\nu-p_{\nu/2}\|\le\tfrac52\sqrt{L\nu}R.
\label{eq:cvxlf-movement}
\end{equation}
The first line follows from~\eqref{eq:cvxlf-energy}, strong convexity, and~\eqref{eq:cvxlf-P}.  Pairing the two regularized KKT systems with $a-a'\in\Ker(\bar K)$ and using cocoercivity proves~\eqref{eq:cvxlf-movement}.  Expanding the two Bregman divergences gives the exact transfer identity
\begin{align}
D_{\phi_{\nu/2}}(f,a')-D_{\phi_\nu}(f,a)
={}&-\frac\nu4\|f-w^0\|^2+\phi_\nu(a)-\phi_\nu(a')\nonumber\\
&+\frac\nu4\|a'-w^0\|^2
+\langle p_{\nu/2}-p_\nu,f-\Pi_{\mathcal A}f\rangle.
\label{eq:cvxlf-transfer}
\end{align}
which with~\eqref{eq:cvxlf-completed}--\eqref{eq:cvxlf-movement} gives
$C_{\rm ws}=108+(\sqrt{12}+5/2)^2+8(3/4+(5/2)\sqrt2)=177.854779\ldots$.

Set
\begin{align}
B_{\rm cont}&:=\max\{200,20+(\bar\Lambda+6)^2\},\qquad
S_{\rm cont}:=\left\lceil\log_2\frac{64LR^2}{\eps}\right\rceil,\qquad \nu_s:=L2^{-s},\nonumber\\
k_s&:=\left\lceil\frac{\log B_{\rm cont}}{\log(1+2^{-s/2}/24)}\right\rceil,\qquad
k:=\max\left\{0,\left\lceil\log_2\frac{16\sqrt2(\bar\Lambda+6)LR^2}{\eps}\right\rceil\right\}.
\label{eq:cvxlf-schedule}
\end{align}

\begin{algorithm}[H]\footnotesize
\caption{Log-free smooth-convex continuation with affine-only output repair}\label{alg:cvx-continuation}
\begin{algorithmic}[1]
\Require $F,w^0,L,R,\eps,\bar\Lambda$; $\bar K,\bar v$ and, if $\bar K\ne0$, certificates $\ell,\lambda_{\rm hi}$
\If{$\bar K$ is absent}\State \Return Algorithm~\ref{alg:cvx-agd}$(F,w^0,L,R,\eps)$\EndIf
\State $d_{\rm RF}\gets\lceil\sqrt{\lambda_{\rm hi}/\ell}\rceil$, $d_{\rm UF}\gets d_{\rm RF}$; set $B_{\rm cont},S_{\rm cont},\nu_s,k_s,k$ by~\eqref{eq:cvxlf-schedule}
\State $g_0\gets\nabla F(w^0)$; $p\gets\UF(g_0;\bar K,0,\lambda_{\rm hi},d_{\rm UF})-g_0$; $x\gets w^0$; $f\gets w^0$
\For{$s=0,\ldots,S_{\rm cont}$}
\If{$s>0$}\State $x\gets f$ \Comment{retain $f$ and the current dual $p$}\EndIf
\For{$j=1,\ldots,k_s$}
\State $(x,f,p)\gets$ Algorithm~\ref{alg:cvx-stage}$(x,f,p,w^0,F,\bar K,\bar v,L,\nu_s,\ell,\lambda_{\rm hi},d_{\rm RF})$
\EndFor
\EndFor
\State $y\gets f$
\For{$j=1,\ldots,k$}\State $y\gets\RF(y;\bar K,\bar v,\ell,\lambda_{\rm hi},d_{\rm RF})$\EndFor
\State $\bar w\gets\UF(y;\bar K,\bar v,\lambda_{\rm hi},d_{\rm UF})$
\State \Return $\bar w$
\end{algorithmic}
\end{algorithm}

\begin{algorithm}[H]\footnotesize
\caption{Accelerated-gradient branch when the affine system is absent}\label{alg:cvx-agd}
\begin{algorithmic}[1]
\Require $F,w^0,L,R,\eps$
\State $x\gets w^0$, $y\gets w^0$, $t\gets1$, $N_{\rm AGD}\gets\lceil\sqrt{2LR^2/\eps}\rceil$
\For{$j=1,\ldots,N_{\rm AGD}$}
\State $x_+\gets y-L^{-1}\nabla F(y)$; $t_+\gets(1+\sqrt{1+4t^2})/2$
\State $y\gets x_++(t-1)(x_+-x)/t_+$; $(x,t)\gets(x_+,t_+)$
\EndFor
\State \Return $x$
\end{algorithmic}
\end{algorithm}

Induction using~\eqref{eq:cvxlf-contract} and~\eqref{eq:cvxlf-warm} gives $E_{\nu_s}(x,f,p)\le R^2$ at every completed stage and
$\eps/(128R^2)<\nu_{S_{\rm cont}}\le\eps/(64R^2)$. At the last stage write $z=\Pi_{\mathcal A}f$ and $n=f-z$. Then
\begin{equation}
F(z)-F^\star<8\nu_{S_{\rm cont}}R^2\le\eps/8,
\qquad \|n\|\le\sqrt{2\nu_{S_{\rm cont}}/L}R.
\label{eq:cvxlf-virtual}
\end{equation}
The executed return is
\begin{equation}
\bar w=z+e,
\qquad e=s_{d_{\rm UF}}(Q)r_{d_{\rm RF}}(Q)^k n\in\mathcal N,
\qquad
\|e\|\le\frac{\|\bar Ke\|}{\sigma}
\le\frac{\eps}{16(\bar\Lambda+6)LR}.
\label{eq:cvxlf-repair}
\end{equation}
The key step is the weighted second-kind estimate $\|\bar Ks_{d_{\rm UF}}(Q)q\|\le\sqrt\ell\|q\|$, which avoids multiplying the repair by $\|\bar K\|/\sigma$.  Using~\eqref{eq:cvxlf-reggeom}, smoothness, and the regularized KKT relation,
\begin{equation}
F(\bar w)-F(z)
\le\|\lambda_{\nu_{S_{\rm cont}}}\|\,\|\bar Ke\|+(\nu_{S_{\rm cont}}R+L\|z-a_{\nu_{S_{\rm cont}}}\|)\|e\|+\tfrac L2\|e\|^2,
\end{equation}
so~\eqref{eq:cvxlf-virtual}--\eqref{eq:cvxlf-repair} imply exactly~\eqref{eq:cvx-accuracy}.

Let $A_{\rm tot}:=\sum_{s=0}^{S_{\rm cont}} k_s$. Since $\log(1+a)\ge a/(1+a)$,
\begin{equation}
A_{\rm tot}<\frac{\sqrt{128}(1+25\log B_{\rm cont})}{1-1/\sqrt2}\,T_{\rm cvx},
\qquad
N_{\nabla f}=1+A_{\rm tot},\qquad
N_{\rm aff}=d_{\rm RF}(4A_{\rm tot}+2k+4).
\label{eq:cvxlf-ledger}
\end{equation}
Here $d_{\rm RF}\le2\widehat\chi$ and
$k\le1+\log_2(16\sqrt2(\bar\Lambda+6))+2\log_2T_{\rm cvx}$.  Hence the final repair contributes only the additive term $O_{\bar\Lambda}(\widehat\chi\log T_{\rm cvx})$, dominated by $O_{\bar\Lambda}(\widehat\chi T_{\rm cvx})$.  Thus
\begin{equation}
N_{\nabla f}=O_{\bar\Lambda}(T_{\rm cvx}),\qquad
N_{\rm aff}=O_{\bar\Lambda}(\widehat\chi T_{\rm cvx}),\qquad
N_{\rm comm}=O_{\bar\Lambda}(\widehat\chi\sqrt{\kappa_W}\,T_{\rm cvx})
\label{eq:cvxlf-upper}
\end{equation}
under the supplied normalized-channel realization.  Thus the affine and communication bounds have no multiplicative target-accuracy logarithm; no raw $A/B/C/D/W$ theorem is inferred beyond the primitive conversion already stated in Appendix~\ref{app:resources}.

\subsection{Radius-uniform lower bounds}
We give a finite construction that attains the matching affine and communication orders while controlling the \emph{full execution radius} uniformly in the cross-angle parameter.

For even $N_c$, let $R_1,R_2$ be the alternating normalized edge matchings
\begin{equation}
(R_1u)_j=(u_{2j}-u_{2j+1})/\sqrt2,\qquad
(R_2u)_j=(u_{2j-1}-u_{2j})/\sqrt2,
\label{eq:cvxlf-matchings}
\end{equation}
with the natural index ranges.  Writing $P_j=R_j^\top R_j$ gives
$P_1+P_2=\tfrac12L_{{\rm path},N_c}=:P_{\rm path}$.  Let $c,\mu>0$, $\delta=3\mu/c$, $\alpha=\operatorname{arcosh}(1+\delta)\le1/4$, $r=e^{-\alpha}$,
\begin{equation}
N_c=2\lceil10/\alpha\rceil,\quad \rho_c=r^{2N_c},\quad a=1-r,\quad
v_j=r^{j-1}+r^{2N_c-j},\quad Z_c=(\sum_jv_j^2)^{-1/2}.
\label{eq:cvxlf-reflected}
\end{equation}
Then $u^\star=Z_cv$ is the unique minimizer of
\begin{equation}
\frac{3\mu}{2}\|u\|^2+\frac c2\langle u,P_{\rm path}u\rangle-\beta u_1,
\qquad \beta=\frac{cZ_ca(1-\rho_c)}{2r},
\label{eq:cvxlf-chain}
\end{equation}
and $\|u^\star\|=1$.  The reflected term makes the last-coordinate KKT equation exact: $v_{N_c+1}=v_{N_c}$; the first-coordinate equation is
$(\delta+1/2)v_1-v_2/2=a(1-\rho_c)/(2r)$.  Hence there is no infinite-chain truncation residue.  Moreover
\begin{equation}
E_c:=\langle u^\star,P_{\rm path}u^\star\rangle\le\frac{a^2}{2(1-\rho_c)},
\qquad
\sum_{j>s}(u_j^\star)^2\ge0.998/e,
\quad s=\lfloor(2\alpha)^{-1}\rfloor.
\label{eq:cvxlf-tail}
\end{equation}
The energy bound follows from $0\le v_j-v_{j+1}\le ar^{j-1}$; the tail bound follows from the exact normalization
$\|v\|^2=(1-r^{4N_c})/(1-r^2)+2N_cr^{2N_c-1}$ and $20\le\alpha N_c<20.5$.

For communication, use $3M$ agents in three consecutive groups, put $t\ge1$, $\kappa=L/\mu\ge100$, and take~\eqref{eq:cvxlf-chain} with $c=Lt^2$.  The outer groups have constraints $x_i=tR_1u_i$ and $x_i=tR_2u_i$, respectively, while all $u_i$ are in consensus; the objective is
\begin{equation}
f_i=\begin{cases}
\tfrac L2\|x_i\|^2+\tfrac\mu2\|u_i\|^2-\beta u_{i,1},&i\in V_1,\\
\tfrac\mu2\|u_i\|^2,&i\in V_2,\\
\tfrac L2\|x_i\|^2+\tfrac\mu2\|u_i\|^2,&i\in V_3.
\end{cases}
\label{eq:cvxlf-hardobj}
\end{equation}
Its unique feasible optimizer has $u_i=u^\star$ and $x_i=tR_ju^\star$.  Crucially, its \emph{full} squared radius from zero is
\begin{equation}
R_0^2=M(3+t^2E_c),\qquad 3M\le R_0^2\le M(3+4/\kappa),
\label{eq:cvxlf-radius}
\end{equation}
because $t^2E_c\le4/\kappa$.  Thus the hard shared tail remains a fixed fraction of the full radius even as the channels become nearly parallel.

Normalize the link channel by $h=\sqrt{1+t^2}$ and the consensus channel by the disagreement projector.  Then their row spaces are disjoint and
\begin{equation}
c_F=t/h,\qquad \chix=h+t,
\qquad \sigma^2=1-t/h=1/[h(h+t)].
\label{eq:cvxlf-hardgeom}
\end{equation}
A valid KKT multiplier is explicit. For the normalized link blocks take $\lambda_{1,i}=-Lh x_i^\star$ and for the consensus block take
\[
d_i=\begin{cases}
\beta e_1-\mu u^\star-Lt^2P_1u^\star,&i\in V_1,\\
-\mu u^\star,&i\in V_2,\\
-\mu u^\star-Lt^2P_2u^\star,&i\in V_3.
\end{cases}
\]
The finite-chain stationarity equation gives $\sum_i d_i=0$, while the link stationarity equations are $Lx_i^\star+\lambda_{1,i}/h=0$. Hence $(\lambda_1,d)$ is a genuine multiplier. Using $t\sqrt{E_c}\le2/\sqrt\kappa$, $ta\le\sqrt{6/\kappa}$, and $R_0\ge\sqrt{3M}$ yields
\begin{equation}
\frac{\sigma\|\lambda^\star\|}{LR_0}
\le\frac{7/\sqrt\kappa+2/\kappa}{\sqrt3}<1.
\label{eq:cvxlf-hardmult}
\end{equation}
Only $V_1$ initially contains the seed $e_1$; local closure in $V_1$ uses $P_1$ and in $V_3$ uses $P_2$.  Since the two matchings alternately expose successive coordinates, coordinate $j$ cannot appear before $(j-1)\Delta$ neighbor steps when the outer-group distance is $\Delta$.  Combining this support invariant with~\eqref{eq:cvxlf-tail}--\eqref{eq:cvxlf-radius} gives $\Omega(Mt\sqrt\kappa)$ communication rounds on the path. Here $\kappa_W=\cot^2(\pi/(6M))=\Theta(M^2)$. For bounded network budget the same construction on the complete six-node graph has $\kappa_W=1$ and gives $\Omega(t\sqrt\kappa)$. For bounded angle use only consensus equalities and the same reflected chain with $c=L-\mu$ and local objectives
\[
f_i(u_i)=\begin{cases}
\frac\mu2\|u_i\|^2+\frac{L-\mu}{2}\langle u_i,P_1u_i\rangle-\beta u_{i,1},&i\in V_1,\\
\frac\mu2\|u_i\|^2,&i\in V_2,\\
\frac\mu2\|u_i\|^2+\frac{L-\mu}{2}\langle u_i,P_2u_i\rangle,&i\in V_3.
\end{cases}
\]
Then $R_0^2=3M$, $\chix=1$, and the minimum-norm consensus multiplier obeys
\[
\frac{\sigma\|\lambda^\star\|}{LR_0}
\le\frac{2\sqrt{6/(\kappa-1)}+\sqrt3/\kappa}{\sqrt3}<1
\qquad(\kappa\ge100),
\]
because the stacked gradient at the feasible optimizer is normal to consensus and its seed, scalar, and matching components satisfy the displayed bound. The identical alternating-exposure argument gives $\Omega(M\sqrt\kappa)$ rounds on the path and $\Omega(\sqrt\kappa)$ on the complete six-node graph.

For the affine resource, use one physical agent and a virtual path with $N_{\rm virt}=3M$ blocks.  Split its incidence rows into odd and even matchings and set
\begin{equation}
\widehat K_1=2^{-1/2}E_{\rm odd}\otimes I_{N_c},\qquad
\widehat K_2=2^{-1/2}E_{\rm even}\otimes I_{N_c}.
\label{eq:cvxlf-affinehard}
\end{equation}
Let $\bar K:=\col(\widehat K_1,\widehat K_2)$. Both channels are individually normalized, while
\begin{equation}
\chix=\cot\frac{\pi}{2N_{\rm virt}}=\Theta(N_{\rm virt}),\qquad \kappa_W=1.
\label{eq:cvxlf-affinegeom}
\end{equation}
A $\bar K$ or $\bar K^\top$ action crosses at most one virtual edge, so the same alternating-frontier argument and~\eqref{eq:cvxlf-tail} require $\Omega(N_{\rm virt}\sqrt\kappa)$ affine calls even if all gradient calls are free.  A six-block one-channel member handles bounded $\Xi$ and still requires $\Omega(\sqrt\kappa)$ affine calls.

Finally rescale amplitudes so the full initial radius is exactly $R$ and set
\begin{equation}
\mu:=24\eps/R^2,\qquad \kappa=L/\mu=T_{\rm cvx}^2/24.
\label{eq:cvxlf-nesting-mu}
\end{equation}
All row spaces, support frontiers, and the scale $\sigma\|\lambda^\star\|/(LR)$ are invariant under this amplitude rescaling.  Whenever~\eqref{eq:cvx-accuracy} holds and the multiplier scale is at most one, strong convexity and KKT give
\begin{equation}
F(w)-F^\star\ge-\eps+\frac\mu2\|w-w^\star\|^2,
\label{eq:cvxlf-nesting}
\end{equation}
so an accurate point must lie within $R/2$ of the optimizer.  For $T_{\rm cvx}\ge50$, $\kappa\ge100$, and the preceding fixed-contraction obstructions therefore yield
\begin{equation}
\mathfrak C^{\rm cvx}_{\rm aff}(\eps)=\Omega(\Xi T_{\rm cvx}),\qquad
\mathfrak C^{\rm cvx}_{\rm comm}(\eps)=\Omega(\Xi\sqrt\Omega\,T_{\rm cvx}).
\label{eq:cvxlf-lower}
\end{equation}
For $\Xi\ge3$ take $t=(\Xi-\Xi^{-1})/2$; for $\Omega\ge16$ take $M=\lfloor\sqrt\Omega/2\rfloor$. The bounded-budget constructions above lose only universal constants. The separate gradient lower bound is the standard smooth convex zero-chain with all affine constraints absent, giving $\Omega(T_{\rm cvx})$. Together with~\eqref{eq:cvxlf-upper}, this proves Theorem~\ref{thm:smooth-convex-logfree} for every fixed $\bar\Lambda\ge1$ and threshold $T_{\rm cvx}\ge50$.

\paragraph{Scope and relation to prior algorithms.}
The continuation wrapper is not claimed as the first smooth-convex APAPC acceleration.  \citet{condat2026accelerated}, Theorem~3.11 and Remark~3.12, already give smooth-convex primal--dual acceleration under a linear equality and explicitly discuss quadratic equality augmentation.  The contribution here is the executable CC-MAC wrapper: fixed normal augmentation, weighted range-valued initialization, retained-dual warm starts, the weighted final repair, and the radius-uniform resource-separated lower constructions that together yield the exact class-level orders above.  The result assumes fixed supplied spectral certificates and exact arithmetic; certificate acquisition and arbitrary time-varying inexact products are outside the theorem.

\section{Nonsmooth Extension: Cross-Sliding Upper Bounds}
\label{app:nonsmooth-proof}
Assume $F$ is finite and convex on the ambient Euclidean space, the affine system is feasible, $\|w^0-w^\star\|\le R$, and
\begin{equation}
M_{3R}:=\sup_{w\in B(w^\star,3R)}\sup_{g\in\partial F(w)}\|g\|<\infty.
\label{eq:local-lipschitz}
\end{equation}
For nonzero $\bar K$, with $\sigma=\sigma_{\min}^+(\bar K)$, the target criterion is
\begin{equation}
F(\bar w)-F^\star\le\eps,\qquad (M_{3R}/\sigma)\|\bar K\bar w-\bar v\|\le\eps.
\label{eq:ns-accuracy}
\end{equation}
If $\bar K=0$ the second inequality is omitted. Set $T_{\rm ns}=M_{3R}R/\eps$ and $T_{\rm nsc}=M_{3R}/\sqrt{\mu_f\eps}$.

We prove Theorem~\ref{thm:nonsmooth-upper}.  Throughout this appendix, $\bar K$ denotes the separately normalized joint affine operator, $\sigma:=\sigma_{\min}^+(\bar K)$, and $L_{\bar K}:=\|\bar K\|$.  The equality is scale-invariant: replacing $(\bar K,\bar v)$ by $(a\bar K,a\bar v)$ for $a>0$ leaves the feasible set, $\kappa(\bar K)$, and the criterion~\eqref{eq:ns-accuracy} unchanged.

\subsection{Dual radius and a linear-operator sliding primitive}
\begin{lemma}[Minimum-norm multiplier bound under local regularity]
\label{lem:ns-dual-radius}
Let $F:\R^d\to\R$ be finite and convex, let $w^\star$ solve $\min\{F(w):\bar Kw=\bar v\}$, and choose a KKT multiplier $\lambda^\star\in\Range(\bar K)$ of minimum norm.  Under~\eqref{eq:local-lipschitz},
\begin{equation}
\|\lambda^\star\|\le \frac{M_{3R}}{\sigma}.
\label{eq:dual-radius-ns}
\end{equation}
\end{lemma}
\begin{proof}
Because $F$ is finite on the ambient space, the standard affine-constraint qualification gives
$0\in\partial F(w^\star)+\bar K^\top\lambda^\star$; redundant multiplier components may be projected away, so $\lambda^\star\in\Range(\bar K)$.  Hence for some $g^\star\in\partial F(w^\star)$,
$\bar K^\top\lambda^\star=-g^\star$.  The restriction of $\bar K^\top$ to $\Range(\bar K)$ has smallest singular value $\sigma$, and therefore
$\sigma\|\lambda^\star\|\le\|\bar K^\top\lambda^\star\|=\|g^\star\|\le M_{3R}$.
\end{proof}

Set $R_\lambda:=2M_{3R}/\sigma$ and $\mathcal Y:=\{\lambda:\|\lambda\|\le R_\lambda\}$.  Define
\begin{equation}
\operatorname{Gap}(w):=F(w)-F^\star+\max_{\lambda\in\mathcal Y}\langle\lambda,\bar Kw-\bar v\rangle
=F(w)-F^\star+R_\lambda\|\bar Kw-\bar v\|.
\label{eq:restricted-gap}
\end{equation}
The KKT inequality
$F(w)-F^\star\ge-\|\lambda^\star\|\|\bar Kw-\bar v\|$ gives
\begin{equation}
\operatorname{Gap}(w)\ge F(w)-F^\star,
\qquad
\operatorname{Gap}(w)\ge\frac{R_\lambda}{2}\|\bar Kw-\bar v\|.
\label{eq:gap-to-accuracy}
\end{equation}
Thus $\operatorname{Gap}(w)\le\eps$ implies~\eqref{eq:ns-accuracy} up to a universal factor two.

We use the bilinear recursive-sliding theorem of \citet{nguyen2026holder}.  Suppose a phase center $w^0$ satisfies $\|w^0-w^\star\|\le R_s\le R$ and restrict the phase to the Euclidean ball
$\mathcal X_s:=B(w^0,2R_s)$.  Every queried primal point then lies in $B(w^\star,3R)$, so~\eqref{eq:local-lipschitz} controls every subgradient used in the phase.  Rescale
\begin{equation}
\bar K_a:=a_s\bar K,\qquad \bar v_a:=a_s\bar v,\qquad a_s:=\frac{R_\lambda}{R_s},
\label{eq:dcs-balance-scale}
\end{equation}
and write $\lambda=a_sy$.  The saddle problem is
\begin{equation}
\min_{w\in\mathcal X_s}\max_{\|y\|\le R_s}
F(w)+\langle y,\bar K_aw-\bar v_a\rangle.
\label{eq:scaled-saddle-phase}
\end{equation}
Since $\|\lambda^\star/a_s\|\le R_s/2$, the restricted dual set contains an optimal multiplier.  In the notation of \citet[Section~4, degenerate case]{nguyen2026holder}, the primal component has H\"older exponent $0$ and constant $O(M_{3R})$, while the bilinear constant is $\|\bar K_a\|$.  Their separated componentwise complexity therefore specializes to
\begin{equation}
N_{\bar K}=O\!\left(\frac{L_{\bar K}R_\lambda R_s}{\eta}\right),
\qquad
N_{\partial f}=O\!\left(\frac{M_{3R}^2R_s^2}{\eta^2}\right)
\label{eq:balanced-sliding}
\end{equation}
for restricted gap target $\eta$.  These are calls to the normalized $\bar K,\bar K^\top$ interface and subgradient calls, respectively; primitive realization is accounted for in Appendix~\ref{app:resources}.

\begin{algorithm}[h]
\caption{\textsc{Cross-Sliding}: one balanced nonsmooth phase}
\label{alg:cross-sliding}
\begin{algorithmic}[1]
\Require $F$, $(\bar K,\bar v)$, $M_{3R}$, center $w^0$, radius $R_s$, target restricted gap $\eta$, spectral norm/gap bounds, exact projections onto the phase balls
\State $a_s\gets R_\lambda/R_s$, $\bar K_a\gets a_s\bar K$, $\bar v_a\gets a_s\bar v$
\State $\mathcal X_s\gets B(w^0,2R_s)$, $\mathcal Y_s\gets\{y:\|y\|\le R_s\}$
\State $(\bar w,\bar y)\gets\operatorname{RS}(F,\,\bar K_a,\,\bar v_a,\,\mathcal X_s,\,\mathcal Y_s,\,\eta)$
\State \Return $\bar w$
\end{algorithmic}
\end{algorithm}
Here $\operatorname{RS}$ is the recursive-sliding method of \citet[Section~4, degenerate bilinear case]{nguyen2026holder} with the displayed primal and dual balls, the supplied subgradient and linear-operator constants, and termination budget $\eta$. The theorem-level count includes calls to its projection oracles. When vectors are distributed by blocks, projection onto a Euclidean ball requires a global norm reduction; that communication is not hidden inside a normalized affine call.

\subsection{Proof of Theorem~\ref{thm:nonsmooth-upper}}
For the convex case take $R_s=R$ and $\eta=\eps$ in~\eqref{eq:balanced-sliding}.  Lemma~\ref{lem:ns-dual-radius} gives $R_\lambda\le2M_{3R}/\sigma$.  Fixed-accuracy separate normalization gives
$L_{\bar K}/\sigma=\sqrt{\kappa(\bar K)}=O(\chix)$, hence
\[
N_{\bar K}=O\!\left(\chix\frac{M_{3R}R}{\eps}\right)=O(\chix T_{\rm ns}),
\qquad
N_{\partial f}=O(T_{\rm ns}^2).
\]
Equation~\eqref{eq:gap-to-accuracy} gives~\eqref{eq:ns-accuracy}.

For the strongly convex case restart this phase.  Let $R_0=R$ and at stage $s$ request
$\eta_s:=\mu_fR_s^2/8$.  If $\operatorname{Gap}(w_s)\le\eta_s$, strong convexity and the KKT inequality imply
\[
\operatorname{Gap}(w_s)
\ge \frac{\mu_f}{2}\|w_s-w^\star\|^2
 +(R_\lambda-\|\lambda^\star\|)\|\bar Kw_s-\bar v\|,
\]
so $\|w_s-w^\star\|\le R_s/2$.  Thus $R_{s+1}=R_s/2$ is valid.  Equation~\eqref{eq:balanced-sliding} yields
\[
N_{\bar K}^{(s)}=O\!\left(\frac{L_{\bar K}R_\lambda}{\mu_fR_s}\right),
\qquad
N_{\partial f}^{(s)}=O\!\left(\frac{M_{3R}^2}{\mu_f^2R_s^2}\right).
\]
The geometric sums are dominated by the final stage.  Stopping at $R_S^2\asymp\eps/\mu_f$ gives
\[
N_{\bar K}=O\!\left(\frac{L_{\bar K}R_\lambda}{\sqrt{\mu_f\eps}}\right)
=O(\chix T_{\rm nsc}),
\qquad
N_{\partial f}=O(T_{\rm nsc}^2).
\]
For any fixed $J$, the same normalized-affine proof replaces $\chix$ by $\chiJ$. A physical-round statement would additionally have to implement and count the global ball projections and scalar reductions, as well as the normalized channels; the certified smooth-SC primitive lift in Appendix~\ref{app:resources} is not reused here without that accounting and a separate lift-radius argument. \qed

\subsection{What is and is not proved on the nonsmooth lower-bound side}
The standard centralized oracle constructions give the first-order lower bounds quoted in Theorem~\ref{thm:nonsmooth-upper}; they are obtained on unconstrained (hence admissible) subclasses and do not involve the affine oracle.  A matching resource-separated lower bound
$N_{\rm aff}=\Omega(\chix T_{\rm ns})$ (and its communication analogue) requires a hard family in which \emph{both} the subgradient oracle and the affine oracle respect the same cross-angle information frontier.  A naive coordinate-dilation argument is insufficient, because a subgradient oracle may expose a terminal micro-coordinate before the affine path has propagated to it.  We therefore leave the matching $\chix$-dependent nonsmooth affine/communication lower bound open rather than assuming such a frontier without proof.

\section{Primitive Resource Accounting for the Smooth-SC Solver}
\label{app:resources}
\paragraph{Charged-operation ledger.}
\begin{center}
\scriptsize
\begin{tabular}{@{}p{0.20\textwidth}p{0.60\textwidth}p{0.12\textwidth}@{}}
\toprule
Resource & One charged operation & Count \\
\midrule
First-order oracle & One parallel local gradient (or subgradient in the nonsmooth extension). & $N_{\nabla f}$ or $N_{\partial f}$ \\
Normalized affine oracle & Apply the separately normalized joint stack or its adjoint. & $N_{\rm aff}$ \\
Communication & One synchronous neighbor round / sparse gossip multiplication. & $N_{\rm comm}$ \\
Primitive implementation & One block-parallel original $A,B,C,D$ family product or adjoint. & $N_A,\ldots,N_D$ \\
\bottomrule
\end{tabular}
\end{center}

This appendix concerns the executable smooth strongly convex solver of Theorem~\ref{thm:algorithm}. The convex and nonsmooth extension theorems use their own normalized-channel ledgers; no unchanged raw-lift ledger is claimed for those regimes. All logarithmic factors below are $\log_+(R_E/\eps)$ for the supplied initial primal--dual energy certificate. Fix $0<\delta_{\rm wh}<1$ independently of $\eps$ and use
the certified construction in Appendix~\ref{app:ccmac-normalization}.
We count block-parallel products by the original matrices
$A,B,C,D$ and their transposes, and raw physical neighbor rounds.
For one normalized forward \emph{or} adjoint call, the upper
bounds are
\begingroup\scriptsize
\[
\begin{array}{c|ccccc}
 & A/A^\top & B/B^\top & C/C^\top & D/D^\top & \text{rounds}\\ \hline
\widehat L_1,\widehat L_1^\top
 & O(\sqrt{\bar k_1}) & O(\sqrt{\bar k_1})
 & O(\sqrt{\bar k_1}\sqrt{\bar\kappa_C}) & 0
 & O(\sqrt{\bar k_1}\sqrt{\bar\kappa_W})\\
\widehat L_2,\widehat L_2^\top
 & 0&0&0&O(\sqrt{\bar k_2})
 & O(\sqrt{\bar k_2}\sqrt{\bar\kappa_W})
\end{array}
\]
\endgroup
Since $Q_jv=\mathcal L_j(\mathcal L_j^\top v)$, degree-$N_j^{\rm wh}$ normalization uses $2N_j^{\rm wh}+1=O(\sqrt{\bar k_j})$ channel actions. The extra factors $\sqrt{\bar\kappa_C}$ and $\sqrt{\bar\kappa_W}$ arise only from $C_\star$ and $\widetilde W$, respectively. With equal unit cost for the four matrix families, set
\begin{equation}
\begin{aligned}
s_{\rm CL}
 &=O\!\left(\sqrt{\bar k_1}(1+\sqrt{\bar\kappa_C})\right),
& s_{\rm SH}&=O(\sqrt{\bar k_2}),\\
\Lambda_{\rm in}&:=s_{\rm CL}+s_{\rm SH},
& \Gamma_{\rm in}&:=\sqrt{\bar k_1}+\sqrt{\bar k_2}.
\end{aligned}
\label{eq:explicit-normalization-costs}
\end{equation}
The budgets absorb universal implementation constants; $\Lambda_{\rm in}$ counts raw affine work and $\Gamma_{\rm in}$ communication-bearing channel work. Absent matrix families contribute zero, and for $n=1$ all communication counts vanish. For the supplied $\Xi\ge\chix$, put
$T_{\Xi,E}:=\sqrt{\kappa_f}\,\Xi\log_+(R_E/\eps)$.
The complete smooth strongly convex solve upper bounds are
\begin{equation}
\begin{aligned}
N_{\nabla f}&=O\!\left(\sqrt{\kappa_f}\log_+(R_E/\eps)\right),\\
N_A,\ N_B&=O(T_{\Xi,E}\sqrt{\bar k_1}),\\
N_C&=O(T_{\Xi,E}\sqrt{\bar k_1}\sqrt{\bar\kappa_C}),\\
N_D&=O(T_{\Xi,E}\sqrt{\bar k_2}),\\
N_{\rm aff}^{\rm prim}&=O(T_{\Xi,E}\Lambda_{\rm in}),\\
N_{\rm comm}&=O(T_{\Xi,E}\Gamma_{\rm in}\sqrt{\bar\kappa_W}).
\end{aligned}
\label{eq:explicit-primitive-complexity}
\end{equation}
Penalty-gradient products and one-time transformed right-hand sides are absorbed above; certificate acquisition is a separate setup cost under Assumption~\ref{ass:spectral-certificates}. Constant-factor-tight certificates, including $\Xi=O(\chix)$, replace barred ratios by $k_1,k_2,\kappa_C,\kappa_W$ and $T_{\Xi,E}$ by $O(\sqrt{\kappa_f}\chix\log_+(R_E/\eps))$. Fixed $\delta_{\rm wh}$ adds no target-accuracy logarithm.

\end{document}